\documentclass[11pt]{article}
\PassOptionsToPackage{serbian}{babel}
\makeatletter
\renewcommand\section{\@startsection {section}{1}{\z@}%
	{-3.5ex \@plus -1ex \@minus -.2ex}%
	{2.3ex \@plus.2ex}%
	{\normalfont\fontfamily{phv}\fontsize{16}{19}\bfseries}}
\renewcommand\subsection{\@startsection{subsection}{2}{\z@}%
	{-3.25ex\@plus -1ex \@minus -.2ex}%
	{1.5ex \@plus .2ex}%
	{\normalfont\fontfamily{phv}\fontsize{14}{17}\bfseries}}
\renewcommand\subsubsection{\@startsection{subsubsection}{3}{\z@}%
	{-3.25ex\@plus -1ex \@minus -.2ex}%
	{1.5ex \@plus .2ex}%
	{\normalfont\normalsize\fontfamily{phv}\fontsize{14}{17}\selectfont}}

\makeatother

\usepackage{hyperref}
\usepackage{amsmath}
\usepackage{graphicx}
\usepackage{enumerate}
\usepackage{color}

\usepackage{url} % not crucial - just used below for the URL
\usepackage[utf8]{inputenc} % Required for inputting international characters
\usepackage[T1]{fontenc} % Output font encoding for international characters

\usepackage[backend=bibtex,natbib=true]{biblatex} % Use the bibtex backend with the authoryear citation style (which resembles APA) ,style=authoryear,natbib=true]

\usepackage[autostyle=true]{csquotes} % Required to generate language-dependent quotes in the bibliography
\usepackage{mathpazo} % Use the Palatino font by default
\usepackage{amsthm}
\usepackage{amssymb}
\usepackage{times}

\usepackage{bm}
\graphicspath{{../Figures/}} %Setting the graphicspath
\usepackage{wrapfig}
\usepackage{float}
\usepackage{algorithm}
\usepackage{algpseudocode}
\usepackage{endnotes}

\restylefloat{table}

\newcommand{\Grs}[1]{G_{#1,2}}

\newcommand{\D}[1]{\Delta_{#1,2}}
\newcommand{\CPA}{\mathbb{C}P^{1}_{A}}

\newcommand{\Ins}[1]{\stackrel{\circ}{#1}}

\newtheorem{theorem}{Theorem}[section]
\newtheorem{definition}{Definition}[section]
\newtheorem{lemma}[theorem]{Lemma}
\newtheorem{cor}{Corollary}[theorem]
\newtheorem{example}{Example}

\newtheorem{prop}[theorem]{Proposition}

\theoremstyle{remark}
\newtheorem{rem}{Remark}

\newcommand{\C}{\mathbb{C}}

\newcommand{\Z}{\mathbb{ Z}}

\newenvironment{dedication}
{\vspace{4ex}\begin{quotation}\begin{center}\begin{em}}
			{\par\end{em}\end{center}\end{quotation}}

\DeclareMathOperator{\s}{\mathfrak{s}}

\makeatletter
\renewcommand*{\eqref}[1]{%
	\hyperref[{#1}]{\textup{\tagform@{\ref*{#1}}}}%
}
\makeatother

\begin{document}
	
	%%%%%%%%%%%%%%%%%%%%%%%%%%%%%%%%%%%%%%%%%%%%%%%%%%%%%%%%%%%%%%%%%%%%%%%%%%%%%%
	\def\spacingset#1{\renewcommand{\baselinestretch}%
		{#1}\small\normalsize} \spacingset{1}
	%%%%%%%%%%%%%%%%%%%%%%%%%%%%%%%%%%%%%%%%%%%%%%%%%%%%%%%%%%%%%%%%%%%%%%%%%%%%%%

		\title{\bfseries \emph{$ \Z_{2} $- homology of the orbit spaces $ \Grs{n}\slash T^{n} $}}
			\author{Vladimir Ivanović and Svjetlana Terzić
					  \endnote{Vladimir Ivanović and Svjetlana Terzić\\
					  	Faculty of Science and Mathematics, University of Montenegro\\
					  	Dzordza Vasingtona bb, 81000 Podgorica, Montenegro\\
					  	vladimir.i@ucg.ac.me; sterzic@ucg.ac.me}	
				  }
			\date{}	
				\maketitle
				\begin{dedication}
					To Victor Matveevich Buchstaber with sincere gratitude and appreciation on the occasion of his 80th birthday.
				\end{dedication}
	\begin{abstract}
		We study the $\Z_2$-homology groups of the orbit space $X_n =
		G_{n,2}/T^n$ for the canonical action of the compact torus $T^n$ on a
		complex Grassmann manifold $G_{n,2}$. Our starting point is the model
		$(U_n, p_n)$ for $X_n$ constructed in~\cite{buchstaber2020resolution}, where
		$U_n = \Delta _{n,2}\times \mathcal{F}_{n}$ for a hypersimplex $\Delta
		_{n,2}$ and an universal space of parameters $\mathcal{F}_{n}$ defined
		in~\cite {buchstaber2019foundations},~\cite{buchstaber2020resolution}.  It is proved
		in~\cite{buchstaber2021orbit} that $\mathcal{F}_{n}$ is diffeomorphic to the
		moduli space $\mathcal{M}_{0,n}$ of stable $n$-pointed  genus zero
		curves. We exploit the results from~\cite{keel1992intersection},~\cite{ceyhan2009chow} on
		homology groups of $\mathcal{M}_{0,n}$ and express them in terms of the
		stratification of $\mathcal{F}_{n}$ which are incorporated in the model
		$(U_n, p_n)$.
		\par  In the result we provide the description of cycles
		in $X_n$, inductively on $ n. $  We obtain as well explicit formulas for $\Z_2$-homology
		groups for $X_5$ and $X_6$. The results for $X_5$ recover by different method the results
		from~\cite{buchstaber2021orbit} and~\cite{suss2020toric} .
		The results for $X_6$ we consider to be new.
		
		\noindent%
%				\let\footnote=\endnote

		%\newpage
		\spacingset{1.5} % DON'T change the spacing!
	\end{abstract}
	\tableofcontents
	\section{Introduction}
	
	\par The study of the orbit space $X_{n}=\Grs{n}/T^{n}$ of a complex Grassmann
	manifold $\Grs{n}$, which represents the two-dimensional complex subspaces in
	$\mathbb{C}^{n}$ under the canonical  action of the compact torus
	$T^{n}$, has garnered considerable mathematical interest from various
	perspectives including algebraic topology, algebraic geometry, group
	actions theory, matroid theory and combinatorics.
	\par The investigation of orbit spaces $X_{n}$ together with the canonical
	moment map $\hat{\mu}:X_{n}\rightarrow\D{n}$, where $\D{n}$ denotes a
	hypersimplex, has been motivated by the challenge of extending
	toric topology methods to torus actions of positive complexity, as
	discussed by Goresky-MacPherson \cite{goresky2006topology}, Gelfand-Serganova \cite{gel1987combinatorial}, Buchstaber-Terzić \cite{buchstaber2019foundations}, Ayzenberg-Masuda \cite{ayzenberg2023orbit}.
	\par In recent works, such as those by Süess \cite{suss2020toric} and Buchstaber-Terzić
	\cite{buchstaber2019toric}
	, the integral homology groups of $X_{5}$ are described using advanced algebraic
	topology theories and geometric invariant theory, respectively.
	\par We study the homology with $\Z_2$-coefficients  of the orbit space $X_n$. 
	Our starting point is the model $(U_n , p_n)$  for
	$G_{n,2}/T^n$  constructed in ~\cite{buchstaber2020resolution},
	where  $U_n= \D{n}\times \mathcal{F}_{n}$  for the smooth manifold
	$\mathcal{F}_{n}$, called the universal space of parameters, and  $p_n :
	U_n\to X_n$  a continuous projection.  The space  $X_n$ is the  quotient
	space of $U_n$ by the map $p_n$.

	\par The basic input in the construction of this model is the result
	from~\cite{buchstaber2020resolution}, which states that
	$X_n$ can be represented as the
	disjoint union of spaces  $\{C_{\omega}\times F_{\omega}\}$. Here,
	$\C_{\omega}$ are the
	chambers in the hypersimplex $\Delta _{n,2}$ which correspond to its
	decomposition given by all possible intersections of matroids, that is
	admissible polytopes.  The spaces $F_{\omega}$ are the orbit spaces of
	$\hat{\mu}^{-1}(C_{\omega})$ by the canonical action of the algebraic
	torus $(\C ^{\ast})^{n}$, where $\hat{\mu} : G_{n,2}/T^n \to \Delta
	_{n,2}$ is the map induced by the standard moment map $\mu : G_{n,2}\to
	\Delta _{n,2}$. The corresponding cell decomposition  of $X_n$ is given
	by the cell decomposition of each $F_{\omega}$.  The characteristic maps are
	given by the characteristic maps for the cell decomposition of $\Delta
	_{n,2}$ into $C_{\omega}$'s, by the characteristic maps  of the cell
	decomposition of $F_{\omega}$'s and by the maps which define the closure
	of a strata in $G_{n,2}$.  Namely, in~\cite{gel1987combinatorial}, it is shown that the closure of a stratum $W_{\sigma}$ is the union of strata
	$W_{\sigma ^{'}}$ for some $\sigma ^{'}\subset \sigma$. On the other hand, it follows from \cite{buchstaber2019toric} that $\mu (W_{\sigma})=\Ins{P}_{\sigma}$ which is the relative interior of polytope of $ P_{\sigma}\subset \D{n} $, called an admissible polytope. In particular, if $P_{\bar{\sigma}}$ is a facet of admissible polytope $P_{\sigma}$, then the stratum
	$W_{\bar{\sigma}}$ such that $\mu (W_{\bar{\sigma}})=\Ins{P}_{\bar{\sigma}}$,
	belongs to the boundary of $W_{\sigma}$. In addition, it is proved in~\cite{buchstaber2019foundations} that there exists continuous surjection $\eta _{\sigma,\bar{\sigma}} : F_{\sigma}\to F_{\bar{\sigma}}$. We prove that this map induce the continuous surjection $\eta _{\omega, \bar{\omega}} : F_{\omega}\to
	F_{\bar{\omega}}$.	
	\par The universal space of parameters $\mathcal{F}_{n}$ is defined
	in~\cite{buchstaber2019foundations} for general $T^k$-action on a smooth
	manifold $M^{2n}$ and studied in detail for $T^n$-action on $G_{n,2}$
	in~\cite{buchstaber2019toric},~\cite{buchstaber2020resolution},~\cite{buchstaber2021orbit}.	In particular, it is proved in~\cite{buchstaber2021orbit} that
	$\mathcal{F}_{n}$ is a	smooth manifold diffeomorphic to the moduli space $\mathcal{M}_{0,n}$ of	stable genus zero curves with $n$ marked ordered distinct points. Moreover, it is	proved	in~\cite{buchstaber2020resolution} that for any $\omega$ there exists continuous surjection $p_{\omega} : \mathcal{F}_{n}\to F_{\omega}$ . This map induces the relations from homology groups generators of $\mathcal{F}_{n}$ to those for  homology groups of $F_{\omega}$. The
	homology groups generators for $ \mathcal{F}_{n} $, as well as their relations, are determined by Keel in~\cite{keel1992intersection} and his results are further generalized	in~\cite{manin2000new} and~\cite{ceyhan2009chow}.
	\par
	The complexity of the study of the orbit spaces  $M^{2n}/T^k$ and their homology structure shows up to follow the complexity of torus action. The homology of quasitoric manifolds  $M^{2n}/T^n$, which belong to the
	class of manifolds with complexity zero torus action, is determined
	by the combinatorics of the moment polytope $P^k.$ It si proved in~\cite{buchstaber2014topology}  that the orbit space
	$X_4$ is homeomorphic to $S^{5}$, Later on, more general result is proved in
	\cite{karshon2020topology} stating that the orbit space of a Hamiltonian complexity
	one torus action in general position is homeomorphic to	a sphere. This is improved in~\cite{ayzenberg2023orbit}  by showing that a manifold $M^{2n}$ with complexity one torus action in	general position and finite nonempty set of fixed points, is a homology $(n+1)$-sphere if its odd homology are trivial. In~\cite{ayzenberg2023orbit} this result is developed
	further by introducing the notion of $j$-generality of weights for
	$T^k$-action on $M^{2n}$. Precisely,
	it is proved that if $X$ is equivariantly formal and the action is with
	isolated fixed points, then $j$-generality of weights implies
	$(j+1)$-acyclicity of the orbit space $M^{2n}/T^k$.
	\par The integral
	homology of  orbit space $X_5$, which is an example of
	complexity $2$ torus action, are computed in~\cite{buchstaber2021orbit}
	and~\cite{suss2020toric}.
	\par
	We recover the computation of homology groups with $\Z _2$ coefficients
	for $X_5$ by the  method different from those used in~\cite{buchstaber2021orbit} and~\cite{suss2020toric}. In addition, we
	compute the homology groups with $\Z_2$-coefficients for $X_6$
	which are, to our knowledge, not previously known. The space $X_6$ is an example
	of complexity $3$ torus action.
	\par
	We believe that the general results obtained in Section \ref{Sec:On_Gn2_homology} describing inductively the structure
	of cycles in $X_n$, may lead to  successful
	application of the presented method  for the explicit  computation  of
	homology groups for	$X_n$ with $\Z_2$-coefficients for higher $n$ as well.
	\par The context of the paper is as follows: In Section \ref{Sec:Definitions}, we provide a summary of the relevant combinatorial structure and algebraic topology of the orbit space	$X_{n}$, following the descriptions from \cite{buchstaber2019toric} and	\cite{buchstaber2020resolution}. We particularly focus on the quotient space nature of $X_{n}$, obtained from a manifold with generalized corners	$\Delta_{n,2}\times\mathcal{F}_{n}$ via an explicitly described equivalence relation involving the chamber decomposition of
	$\Delta_{n,2}$ and corresponding stratifications of the manifold
	$\mathcal{F}_{n}$.
	\par In Section \ref{Sec:On_Gn2_homology}, we study  in detail the homology
	groups  $H_{k}(X_{n};\mathbb{Z}_{2})$ for arbitrary $k$, by making use
	of the model  $(\Delta_{n,2}\times\mathcal{F}_n, p_n)$ for $X_n$. The
	structural insight that $X_{n}$ decomposes into specific subspaces
	$\{C_{\omega}\times F_{\omega}\}$ facilitates an inductive study of
	$\mathbb{Z}_{2}$-cycles within $X_{n}$. Using diffeomorphism established
	in~\cite{buchstaber2021orbit} between $\mathcal{F}_n$ and
	$\mathcal{M}_{0,n}$ and the results from
	\cite{keel1992intersection},~\cite{ceyhan2009chow} on homology of
	$\mathcal{M}_{0,n}$ , we deduce that    the homology groups  for
	$\mathcal{F}_{n}$ are spanned by the divisors outgrowing in the
	compactification of $F_n$ to $\mathcal{F}_{n}$.  Moreover, for any
	$F_{\omega}$ being compactification of $F_n$,  we show that the 
	homology groups of  $F_{\omega}$ are   spanned as well  by the divisors
	outgrowing in the compactification of $F_n$ to $F_{\omega}$ . This is in
	the line with the results from~\cite{buchstaber2024weighted} which
	relate $F_{\omega}$ with moduli spaces of weighted stables genus zero
	curves and the results from~\cite{ceyhan2009chow} on homology groups of such
	moduli spaces.
	\par Section \ref{Sec:ModuliSpace} builds upon the diffeomorphism between the
	universal space $\mathcal{F}_{n}$ and the moduli space $\mathcal{M}_{0,n}$
	established in \cite{buchstaber2021orbit}.
	\par We exploit the knowledge of the homology groups of $\mathcal{M}_{0,n}$
	generated by divisors $D_{I}$ as outlined in
	\cite{keel1992intersection}, and explicitly  relate it for $n=5,6$  to 
	the stratifications   of
	$\mathcal{F}_n$ into virtual space of parameters which correspond to the
	chambers $C_{\omega}$ of  $\Delta _{n,2}$, which is established in~\cite{buchstaber2020resolution}.
	\par  We obtain the following results:
	\begin{itemize}
		\item  $ \Z_{2} $- homology groups of $ X_{5} $ are:
		\begin{equation*}
			H_{k}(X_{5};\Z _{2})
			\begin{cases}
				\Z_{2}, & \text{if}\ k=0,5,6,8 \\
				0, & \text{otherwise}
			\end{cases}
		\end{equation*}
		
		\item  $ \Z_{2} $- homology groups of $ X_{6} $ are:
		\begin{equation*}
			H_{k}(X_{6};\Z _{2})
			\begin{cases}
				\Z_{2}, & \text{if}\ k=0,5,8,9,11 \\
				\Z_{2}^{11}, & \text{if}\ k=7\\
				\Z_{2}^{3}, & \text{if}\ k=6\\
				0, & \text{otherwise}
			\end{cases}
		\end{equation*}
	\end{itemize}

\section{Grassmann manifolds $ \Grs{n} $}\label{Sec:Definitions}
For the convenience of a reader we recall the relevant description of the
combinatorial structure and algebraic topology of the orbit space $ \Grs{n}\slash T^{n} $ from \cite{buchstaber2019toric} and \cite{buchstaber2020resolution}
without proofs, thus making our exposition self-contained.
\par We denote by $ \Grs{n} $ the complex Grassmann manifold of the two-dimensional complex subspaces in $ \C^{n} $. The coordinate-wise action of the compact torus $ T^{n} $ on $ \mathbb{C}^{n} $ induces the canonical action of $ T^{n} $ on $ \Grs{n}. $ The standard Plücker charts on $ \Grs{n} $ are defined by $ M_{I}=\{L\in \Grs{n}:\, P^{I}(L)\neq 0 \}$, where $ I\subset\{1,\dots,n \},\,|I|=2$ and $ P^{I} $ are the Plücker coordinates. Let $ Y_{ij}=\Grs{n}\backslash M_{ij}, $ a $ \textit{stratum} $ is defined by non-empty set
\[ W_{\sigma}=(\bigcap_{\{i,j\}\in \sigma} M_{ij})\cap(\bigcap_{\{i,j\}\notin \sigma} Y_{ij}), \] where $ \sigma \subset  \binom{n}{2}. $ For simplicity the elements of $ \sigma $ we further denote just by $ ij $ assuming that $ i<j. $ The stratum $ W:=\bigcap_{\text{all}} M_{ij} $ is said to be \textit{the main stratum}.

\par Let $ \mu_{n}:\Grs{n}\to\Delta_{n,2}$ be the standard moment map, that is
\[ \mu_{n}(L)=\frac{1}{\sum|P^{I}(L)|}\sum|P^{I}(L)|^{2}\Lambda_{I}, \] where $ \Lambda_{I}\in\mathbb{R}^{n},\,\Lambda_{i}=1 $ for $ i\in I $, while $ \Lambda_{i}=0 $ for $ i\notin I. $
The image of $ \Grs{n} $ by the moment map $ \mu_{n} $ is the hypersimplex $ \Delta_{n,2}, $ which is $ (n-1) $-dimensional polytope that belongs to the hyperplane $ x_{1}+\dots+x_{n}=2 $ and is defined as the convex hull of the points $ \Lambda_{ij},\, 1\leq i<j\leq n. $ 
\par The action of $ T^{n} $ on $ \Grs{n} $ is not effective, but the quotient torus $ T^{n-1}=T^{n}/S^{1} $ by the diagonal circle $ S^{1}, $ acts effectively on $ \Grs{n} $. Note that the dimension of $ \Delta_{n,2} $ is equal to the dimension of the effectively acting torus. It follows that $ \dim X_{n}=\dim \Grs{n}-(n-1)=4(n-2)-(n-1)=3n-7. $ 
\par It holds that $ \mu_{n}(W_{\sigma})=\stackrel{\circ}{P}_{\sigma} $, where  $ P_{\sigma} $ is convex hull of the vertices $ \Lambda_{ij}$ such that $ ij\in \sigma. $ A polytope $ P_{\sigma} $ is called an $ \textit{admissible polytope} $ for $ \Grs{n}. $

\par The strata $ W_{\sigma} $ are invariant under the algebraic torus $ (\mathbb{C}^{*})^{n} $-action and $ \Grs{n}=\bigcup W_{\sigma}. $ Therefore, they induce the stratification of the orbit space $ X_n=\Grs{n}/T^{n} $ by
\begin{equation}\label{Eq:stratification}
	X_{n}=\bigcup_{\sigma}W_{\sigma}/T^{n}.
\end{equation} The orbit space $ F_{\sigma}=W_{\sigma}\backslash(\mathbb{C}^{*})^{n} $ is called \textit{the spaces of parameters} of a stratum $ W_{\sigma} $.

 \par Since the boundary $ \partial \Delta_{n,2} $ consists of $ n $ copies of the hypersimplex $ \Delta_{n-1,2} $ and $ n $ copies of the simplex $ \Delta_{n-1,1}$,see also \cite{ziegler2012lectures}, we have \[ \mu^{-1}(\partial \Delta_{n,2})=\bigcup_{1\leq i\leq n}\mu^{-1}(\Delta_{n-1,2}(i))\cup \bigcup_{1\leq i\leq n}\mu^{-1}(\Delta_{n-1,1}(i))=\bigcup_{1\leq i\leq n}\Grs{n-1}(i)\cup \bigcup_{1\leq i\leq n}\C P^{n-2}(i), \] where $ \Delta_{n-1,2}(i)=\D{n}\cap\{x_{i}=0\}$ and $ \Delta_{n-1,1}(i)=\D{n}\cap\{x_{i}=1\} $.
 Thus, the admissible polytopes for $ \Grs{n} $ can be described inductively. The results on admissible polytope from \cite{buchstaber2020resolution} can be summarized as follows: 
 
\begin{theorem}\label{Th:admissiblepolytope}
	\begin{itemize}
		\item 	Any admissible polytope for $ \Grs{n} $ of dimension less or equal than $ n-3 $ \\belongs to the boundary $ \partial\Delta_{n,2}.$ 
		\item The admissible polytopes $ P_{\sigma} $ for $ \Grs{n} $ of dimension $ n-2 $ which have non-empty intersection with $ \Ins{\Delta_{n,2}} $ are given by the intersection of $ \Delta_{n,2} $ with the planes of the form:
		\[ \sum_{i\in S,\,||S||=p}x_{i}=1, \] 
		where $ S \subset \{1,\dots,n\},\, 2\leq p\leq [\frac{n}{2}].$
			\item The admissible polytopes $ P_{\sigma} $ for $\Grs{n} $ of the maximal dimension are $ \D{n} $ and the intersections with $ \D{n} $ of all collections of the half-spaces of the form 
		\begin{equation}\label{Eq:AdmissPoly}
			H_{S}: \sum_{i\in S}x_{i}\leq 1,
		\end{equation}
		for $ S \subset \{1,\dots,n\},\, ||S||=k,\, k\in\{2,\dots,n-2\}$ such that if $ H_{S_{1}},H_{S_{2}} $ belong to a collection then $ S_{1}\cap S_{2}=\emptyset. $ 
	\end{itemize}
\end{theorem}

The space of parameters of the main stratum $ W $ is denoted by  $F_{n} $ or $F$. If we want to emphasize a chart $ M_{ij} $ in which the record of a space of parameters $ F_{\sigma} $ or a stratum $ W_{\sigma} $ is written, we denote it by $ F_{\sigma,ij}\text{ or }W_{\sigma,ij}. $

%\begin{prop}
%	The non-point space of parameters $ F_{\sigma} $ of a stratum $ W_{\sigma} $, that is different from the main stratum, can be embedded in $ (\C P^{1})^{l} $ for $ l\in[1,\dots,n-4]. $
%\end{prop}
%\par Since the action of $ S_{n} $ on $ \Grs{n} $ permutes the charts, the strata and therefore the admissible polytopes, we can fix an one chart. Let $ M_{12} $ be the fixed chart. Any element $ L\in \Grs{n} $ can be presented as matrix $ n\times 2 $ where columns are basic vectors of $ L $ and first two rows form the identity matrix. Then there can be defined an induced action of $ (\C^{*} )^{n}$ on $ \C^{2 (n-2)} $ as 
%\[ \begin{pmatrix}
%	t_{1}&0\\
%	0&t_{2}\\
%	t_{3}a_{3}&t_{3}b_{3}\\
%	\vdots&\vdots\\
%	t_{n-1}a_{n}&t_{n-1}b_{n}\\
%	t_{n}a_{n}&t_{n}b_{n}
%\end{pmatrix}=\begin{pmatrix}
%	1&0\\
%	0&1\\
%	\frac{t_{3}}{t_{1}}a_{3}&\frac{t_{3}}{t_{2}}b_{3}\\
%	\vdots&\vdots\\
%	\frac{t_{n-1}}{t_{1}}a_{n-1}&\frac{t_{n-1}}{t_{2}}b_{n-1}\\
%	\frac{t_{n}}{t_{1}}a_{n}&\frac{t_{n}}{t_{2}}b_{n}
%\end{pmatrix}. \] 
%If we denote $ \tau_{i}=\frac{t_{k+1}}{t_{i}}, $ for $ 1\leq i\leq k $ and $ \tau_{k+i}=\frac{t_{k+i+1}}{t_{1}}$ for $ 1\leq i\leq n-k-1 $ then the action of $ (\C^{*})^{n} $ has the representation \[ (\tau_{1},\dots,\tau_{n})\to(\tau_{1},\dots,\tau_{k+1},\frac{\tau_{k+1}\tau_{2}}{\tau_{1}},\dots,\frac{\tau_{k+1}\tau_{k}}{\tau_{1}},\tau_{k+2},\frac{\tau_{k+2}\tau_{2}}{\tau_{1}},\dots,\frac{\tau_{k+2}\tau_{k}}{\tau_{1}},\dots,\tau_{n},\frac{\tau_{n}\tau_{2}}{\tau_{1}},\dots,\frac{\tau_{n}\tau_{k}}{\tau_{1}}) \] 
\subsection{The universal space of parameters and virtual spaces of parameters for $ \Grs{n} $}\label{SSec:UniversalSpace}
%\label{SSec:Stratum}
The main stratum $ W $, in the chart $ M_{12} $, is given by the equations\begin{align}\label{Def:mainstratumG}
	c_{ij}^{\prime}a_{i}b_{j}=c_{ij}a_{j}b_{i},\quad 3\leq i<j\leq n,
\end{align}  where $ 	c_{ij},c_{ij}^{\prime}\neq0,\,c_{ij}\neq c_{ij}^{\prime} ,\, (c_{ij}:c_{ij}^{\prime})\in\C P^{1}.$ It follows that the parameters $ c_{ij} $ satisfy the equations
\begin{align}\label{Def:mainspaceG}
	c_{ij}^{\prime}c_{ik}c_{jk}^{\prime}=c_{ij}c_{ik}^{\prime}c_{jk}.
\end{align}

\par The space of parameters $ F_{n} $ is an open algebraic manifold in $ (\C P^{1})^{N} $, for $ N=\binom{n-2}{2}, $ given by the intersection of cubics \eqref{Def:mainspaceG} and the condition $ (c_{ij}^{\prime}:c_{ij})\in \CPA, $  where $ \CPA=\C P^{1}\backslash A$ and $ A=\{(1:0),(0:1),(1:1)\}. $ The dimension of $ F_{n} $ is $ 2(n-3). $

\begin{rem}\label{cellspaceparam}
	It follows from~\cite{buchstaber2024weighted} that $F_n$ is homeomorphic to $(\C P^{1}_{A})^{n-3}$.  The cell decomposition of $\C P^{1}_{A}$ is given  $2$-dimensional and $1$-dimensional cells, which produces the cell decomposition of $F_n$.  This implies  that $F_{n}$ has a cell decomposition without cells in dimensions $\leq n-4$ which further gives that
	\begin{equation}\label{Fhom(n-4)}
		H_{k}(F_n) =0, \; \text{for}\; 0< k\leq n-4.
	\end{equation}
	It is not difficult to verify, see \cite{buchstaber2021orbit}, that  any $F_{\sigma}$ is either a point or it is homeomorphic to $F_{m}$, $4\leq m\leq n-1$, which  gives  the cell decomposition for $F_{\sigma}$ and an  analogue result to \eqref{Fhom(n-4)}  for its homology groups.
\end{rem}
\begin{prop}{(\cite{buchstaber2019toric})}\label{Prop:ArbitraryStratumG}
	A stratum $ W_{\sigma} $ is defined in the chart $ M_{12} $ by $ P^{1i}=0, P^{2j}=0 $ or $ P^{pq}=0 $, for some $ 3\leq i,j,p,q\leq n$ or in the local coordinates by $ a_{i}=b_{j}=0 $ and $ a_{p}b_{q}=a_{q}b_{p}. $ The space of parameters $ F_{\sigma}=W_{\sigma}\slash(\C^{*})^{n} $ that is not a point is given by restriction of the equations \eqref{Def:mainspaceG} to some factors $ \C P^{1}\backslash B $ in $ (\C P^{1})^{N} $ where $ B=\{(1:0),(0:1)\}. $
\end{prop}
\par The closure of $ F_{n} $ in $ (\C P^{1})^{N} $ is denoted by $ \overline{F}_{n}. $  It is given by intersection of cubics $ \eqref{Def:mainspaceG} $ in $ (\C P^{1})^{N}. $ The space $ \overline{F}_{n}  $ is proved in  \cite{buchstaber2019toric} to be a smooth manifold. The points from the boundary of $ F_{n} $ in $ (\C P^{1})^{N} $ are obtained if in \eqref{Def:mainspaceG} is allowed that some $ c_{lm}=0,c^{\prime}_{lm}=0\, $ or $ c_{lm}=c^{\prime}_{lm} $ for $  3\leq l<m\leq n.$ 
%Since the main stratum is a dense set in a fixed chart $ M_{12}, $ the subsets $ \overline{F}_{\sigma,12}\subset \overline{F} $ are well defined for the stratum $ W_{\sigma} $ as closure of the sets $ F_{\sigma,12} $ in $ \overline{F}. $ For a deeper discussion of all spaces $ \overline{F}_{\sigma,12}$, see Section 6 in \cite{buchstaber2019toric}. 
%\subsection{The universal space of parameters for $ \Grs{n} $}\label{SSec:UniversalSpace}
\par The universal space of parameters $  \mathcal{F}_{n}  $ is a compactification of the space of parameters $ F_{n} $ of the main stratum, see \cite{buchstaber2020resolution} and the references given there. 

%The next theorem, proved in \cite{buchstaber2019toric}, describe the The universal space and a virtual space of parameters.
%
%\begin{theorem}
%	There exists a topological space $ \mathcal{F} $ such that:
%	\begin{enumerate}
	%		\item for any chart $ M_{ij} $ there is a map $ H_{ij}:\mathcal{A}\to \mathcal{F},\,\sigma\to \tilde{F}_{\sigma,ij} $ where $ \tilde{F}_{\sigma,ij}\subset\mathcal{F} $ is defined by: $ c\in \tilde{F}_{\sigma} $ if and only if there exists a sequence $ (x_{k},c_{k})\subset \overset{\circ}{\Delta}_{n,2}\times F_{ij} $ such that $ c_{k} $ converges to $ c $ and $ h^{-1}(x_{k},f_{ij}(c_{k})) $ converges to a point from $ W_{\sigma}\slash T^{n}. $
	%		\item $ \tilde{F}_{\sigma,ij} $ is homeomorphic to $ \tilde{F}_{\sigma,kl} $ for any charts $ M_{ij}, M_{kl}. $ 
	%	\end{enumerate}
%\end{theorem}
%The set $ \mathcal{A} $ is the set of all admissible sets, $ h:W_{\sigma}\slash T^{n}\to \overset{\circ}{\Delta}_{n,2}\times F $ is the canonical trivilization of the main stratum and $ f_{ij}:F_{ij}\to F $ is the canonical homeomorphism for the chart $ M_{ij}. $
It is defined as a wonderful compactification of the space $ \bar{F}_{n} $. This compactification requires that the homeomorphisms of $ F_{n} $ defined by transition functions between the charts for $ \Grs{n} $ extend to the homeomorphisms for $ \mathcal{F}_{n} $ in an arbitrary chart. 

\begin{theorem}{(\cite{buchstaber2021orbit} and \cite{buchstaber2020resolution})}\label{Th:UniversalSpaceG}
	Let the manifold $ \mathcal{F}_{n} $ is obtained by the wonderful compactification of $ \overline{F}_{n} $ with the generating set of subvarieties $ \overline{F}_{I}\subset\overline{F}_{n} $, defined by $ (c_{ik}:c_{ik}^{\prime})=(c_{il}:c_{il}^{\prime})=(c_{kl}:c_{kl}^{\prime})=(1:1) $ for $ ikl\in I $, where $ I\subset \{ikl\,|\,3\leq i<k<l\leq n\} $.
	Then any homeomorphism $ f_{ij,kl}:F_{n}\to F_{n} $ induced by the transition functions between the charts $ M_{ij} $ and $M_{kl}  $ extends to the homeomorphism of $ \mathcal{F}_{n} $.
\end{theorem}

The main stratum $ W $ is a dense set in $ \Grs{n} $ and according to \cite{buchstaber2020resolution} and \cite{buchstaber2019foundations} to any stratum can be assigned the \textit{virtual space of parameters}  $ \tilde{F}_{\sigma}\subset \mathcal{F}_{n} $, in a fixed chart $ M_{12}, $ on the following way:
\begin{enumerate}
	\item If $ W_{\sigma}\subset M_{12} $, the virtual space $ \tilde{F}_{\sigma} $ is assigned using the description \eqref{Def:mainspaceG} of the main stratum
	\item If $ W_{\sigma}\cap M_{12}=\emptyset $,  we first find a chart $ M_{ij} $ such that $ W_{\sigma} \subset M_{ij}$ and then we assign to $ W_{\sigma} $ the space $ \tilde{F}_{\sigma,ij} $ using description \eqref{Def:mainspaceG}. Finally, using homeomorphism $ f_{ij,12} $ for $ \mathcal{F}_{n} $ defined in Theorem \ref{Th:UniversalSpaceG} we assign to $ W_{\sigma} $ the subspace $ \tilde{F}_{\sigma,12}. $
\end{enumerate} 
\begin{theorem}\label{Th:VirtualSpace}
	For a chart $ M_{ij} $ it holds:
	\begin{enumerate}
		\item $ \bigcup_{\sigma}\tilde{F}_{\sigma,ij}=\mathcal{F}_{n} $
		\item  There exists a canonical projection $p_{\sigma,ij}:\tilde{F}_{\sigma,ij}\to F_{\sigma} $ for any admissible set $ \sigma $.
	\end{enumerate}
\end{theorem}   
\subsection{The model for $ X_{n}=\Grs{n}/T^{n} $}
\par Let the hyperplane arrangement $ \mathcal{G}_{n} $ in $ \mathbb{R}^{n} $ be defined by 
\begin{equation}\label{Eq:ArrangGn2}
	\Pi \cup \{x_{i}=0,1 |\,1\leq i\leq n \},
\end{equation} where $ \Pi $ is the set of hyperplanes given by 
\[ \sum_{i\in S,\,||S||=p}x_{i}=1, \] 
for $ S \subset \{1,\dots,n\},\, 2\leq p\leq [\frac{n}{2}].$ Let $ L(\mathcal{G}_{n}) $ be its intersection poset. In what follows we consider the hyperplane arrangement induced by this arrangement in $ \mathbb{R}^{n-1}=\{(x_{1},\dots,x_{n})\in \mathbb{R}^{n}|\,x_{1}+\dots+x_{n}=2\} $.
\par If restrict the poset $ L(\mathcal{G}_{n}) $ to $ \Delta_{n,2} $ we get the poset $ L(\mathcal{G}(n,2))= L(\mathcal{G}_{n})\cap\Delta_{n,2} $. From now on, we call an element $ C $ of the poset $ L(\mathcal{G}(n,2)) $ \textit{a chamber} and the decomposition $ L(\mathcal{G}(n,2)) $ of $ \Delta_{n,2} $ \textit{the chamber decomposition.}

In \cite{buchstaber2020resolution} it is proved:
\begin{lemma}
	Any chamber $ C\in L(\mathcal{G}(n,2)) $ can be obtained as the intersection of the relative interiors of all admissible polytopes which contain $ C. $
\end{lemma}
\par Thus, a chamber from $ L(\mathcal{G}(n,2)) $ is defined by 
\[ C_{\omega}:=\bigcap_{\sigma\subset \omega}\stackrel{\circ}{P}_{\sigma}\quad \text{ and } C_{\omega}\cap\stackrel{\circ}{P}_{\sigma}=\emptyset  \text{ if }\sigma\notin \omega.\]  
  According to \cite{buchstaber2020resolution} the orbit space $ X_{n} = G_{n,2}/T^n $ can be represented by
  \begin{equation}\label{Eq:space_X_n}
  	 X_{n}=\bigcup_{1\leq i\leq n} X_{n-1}(i) \cup \bigcup_{1\leq i\leq n}\Delta_{n-1,1}(i) \cup \hat{\mu}^{-1}(\Ins{\D{n}}),
  \end{equation} where 
   \begin{equation}\label{CelijaC}
   	\hat{\mu}^{-1}(\Ins{\D{n}})=\bigcup\hat{C}_{\omega},\, \hat{C}_{\omega}= \hat{\mu}^{-1}(C_{\omega})
   \end{equation}  for a chamber $C_{\omega}\subset \stackrel{\circ}{\Delta}_{n,2}$. In \cite{buchstaber2019toric} it is showed that there exists the canonical homeomorphism 
\[h_{\sigma}:W_{\sigma}/T^{\sigma}\to \Ins{P_{\sigma}}\times F_{\sigma} \text{ given by } h_{\sigma}=(\hat{\mu},p_{\sigma}) ,\]
 where $ \hat{\mu}_{\sigma}:W_{\sigma}/T^{\sigma}\to \Ins{P_{\sigma}} $ is induced by the moment map $ \hat{\mu}: \Grs{n}/T^{n}\to\Delta_{n,2}  $ and $ p_{\sigma}:W_{\sigma}/T^{\sigma}\to F_{\sigma} $ is induced by the natural projection $ \Grs{n}\to \Grs{n}/(\C^{\star})^{n}. $ In particular, $ W/T^{n}\simeq\Delta_{n,2}\times F $ for the main stratum $ W. $ 
  \begin{lemma}{(\cite{buchstaber2020resolution})}
  	For any $ C_{\omega}\in L(\mathcal{G}(n,2)) $ such that $ \dim C_{\omega}=n-1 $ there exists canonical homeomorphism 
  	\[ h_{\omega}:\hat{C}_{\omega}\to C_{\omega}\times F_{\omega}, \]
  	where the manifold $ F_{\omega} $ is a compactification of $ F_{n} $ given by the spaces $ F_{\sigma} $ such that $ C_{\omega}\subset P_{\sigma} $,
  	\[ F_{\omega}=\bigcup_{C_{\omega}\subset \Ins{P_{\sigma}}} F_{\sigma}. \]
  \end{lemma}
  \begin{lemma}
	For any $ C_{\omega}\in L(\mathcal{G}(n,2)) $ such that $ \dim C_{\omega}=n-2 $ there exists canonical homeomorphism 
	\[ h_{\omega}:\hat{C}_{\omega}\to C_{\omega}\times F_{\omega}, \]
	where the manifold $ F_{\omega} $ is a compactification of the space $ F_{n} $ given by the spaces $ F_{\sigma} $ such that $ C_{\omega}\subset \Ins{P_{\sigma}} $, where $ \dim P_{\sigma}=n-1 $ and a point $ F_{\sigma} $ such that $ C_{\omega}\subset \Ins{P_{\sigma}} $, where $ \dim P_{\sigma}=n-2. $
\end{lemma}
Similar result is proved in \cite{buchstaber2020resolution} for a chambers $ C_{\omega} $ where $ \dim C_{\omega}\leq n-3.$
    \begin{lemma}
  	For any $ C_{\omega}\in L(\mathcal{G}(n,2)) $ such that $ \dim C_{\omega}\leq n-3 $ there exists canonical homeomorphism 
  	\[ h_{\omega}:\hat{C}_{\omega}\to C_{\omega}\times F_{\omega}, \]
  	where the manifold $ F_{\omega} $ is a compactification of the space $ F_{n} $ given by the spaces $ F_{\sigma} $ such that $ C_{\omega}\subset \Ins{P_{\sigma}} ,$ where $ \dim P_{\sigma}=n-1 $ and $ q $ points, where $ q\geq2 $ is the number of polytopes $ P_{\sigma} $ such that $ C_{\omega}\subset \Ins{P_{\sigma}} $ and $ \dim P_{\sigma}=n-2. $
  \end{lemma}
  
   \par In particular, $\dim F_{\omega}=\dim F_{n}$ for any $\omega$ and $\dim F_{\sigma}<\dim F_{n}$ for any $\sigma$ such that $F_{\sigma}\neq F_{n}$.
   \par Let $ \mathcal{F}_{n} $ be an universal space of parameters for $ \Grs{n} $ and  fix a chart $ M_{ij}. $ By Theorem \ref{Th:VirtualSpace} to any stratum $ W_{\sigma} $ is assigned the virtual space of parameters $ \tilde{F}_{\sigma} = \tilde{F}_{\sigma, ij} $ and the projection $ p_{\sigma,ij}:\tilde{F}_{\sigma, ij}\to F_{\sigma}. $  
\begin{theorem}{(\cite{buchstaber2020resolution})}\label{Lem:ProjG}
	For any chamber $ C_{\omega} $ and a chart $ M_{ij} $ it holds that
		\[ \mathcal{F}_{n}=\bigcup_{C_{\omega}\stackrel{\circ}{ P}_{\sigma}} \tilde{F}_{\sigma} \]
		and this union is disjoint . Moreover, it is defined the projection $ p_{\omega,ij}:\mathcal{F}_{n}\to F_{\omega} $ by $ p_{\omega,ij}(y)=p_{\sigma,ij}(y) $ where $ y\in \tilde{F}_{\sigma,ij}. $
\end{theorem}	

Altogether, this leads to the model for the orbit space $ X_{n}. $ Let us denote
\begin{equation}\label{Model:U_n}
	U_{n}=\Delta_{n,2}\times \mathcal{F}_{n}.
\end{equation}

\begin{theorem}{(\cite{buchstaber2020resolution})}\label{Th:MainModel}
	For any chart $ M_{ij}\subset \Grs{n} $ the map 
	\[p_{n}:U_{n}\to \Grs{n}/T^{n},\, \, p_{n}(x,y)=h_{\omega}^{-1}(x,p_{\omega,ij}(y))  \] if and only if $ x\in C_{\omega} $ is correctly defined. In addition, the map $ p_{n} $ is a continuous surjection and the orbit space $ \Grs{n}/T^{n} $ is homeomorphic to the quotient of the space $ U_{n} $ by the map $ p_{n}. $
\end{theorem}

\section{The moduli space $ \mathcal{M}_{0,n} $ and the universal space of parameters $ \mathcal{F}_{n} $}\label{Sec:ModuliSpace}
We denote by $ \mathcal{M}_{0,n} $ the moduli space of stable genus zero curves with $ n $ marked ordered pairwise distinct points. In \cite{tavakol2017chow} it is presented the construction of the moduli space $ \mathcal{M}_{0,n} $ based on the iterated blowing-up of the space $ (\C P^{1})^{n-3}$. We recall the basic facts about this construction and demonstrate it for the cases $ n=5,6. $ 

\par Let $ X = (\C P^{1})^{n-3} $ and $ I=\{i_{1},\dots,i_{k}\} $ be a subset of the set $ \{1,\dots,n\} $ such that 
\[ |I\cap\{n-2,n-1,n\} |\leq 1 \quad \text{ and }\,2\leq k\leq n-2. \] The subvariety $ X_{I} $ of $ X $ is given as follows:  if the set $ I\cap\{n-2,n-1,n\} $ is empty, then $ X_{I} $ is the set of all points $ (x_{1},\dots,x_{n-3}) $ in $ X $ such that coordinates with indices $ i_{j} $ are equal to each other, where $ 1\leq j\leq k. $ Otherwise, elements $ n-2,n-1,n $ are considered to be special elements with the following correspondence:\[ n-2\leftrightarrow 0,\quad n-1\leftrightarrow 1,\quad n\leftrightarrow \infty.\] In the case $ |I\cap\{n-2,n-1,n\} |=1 $ the subvariety $ X_{I} $ of $ X $ is the set of all points $ (x_{1},\dots,x_{n-3}) $ in $ X $ such that coordinates with non-special indices $ i_{j} $  are equal to the special value coordinate as defined above, where $ 1\leq j\leq k. $
\par For any other subset $ I\subset \{1,\dots,n \}$ the subvariety $ X_{I} $ is defined following the rule $ X_{I}=X_{I^{\complement}}, $ where $ I^{\complement}=\{1,\dots,n\}\backslash I $.
\begin{example}
	For $ n=5, $ we get $ X_{12}=\{(x,x)|\,x\in \C P^{1}\} $, $ X_{34}=X_{125}=\{(\infty,\infty)\}. $
\end{example}

\par The moduli space of curves of genus $ 0 $ with $ n $ marked distinct points, denoted by $ M_{0,n} $, is identified in \cite{tavakol2017chow} with the subset of $ X $:\[ M_{0,n}=\{(x_{1},\dots,x_{n-3})\in (\C P^{1})^{n-3}|\, x_{i}\neq\{0,1,\infty\},\,x_{i}\neq x_{j}\}, \] where $ 0=(0:1),1=(1:1) $ and $ \infty=(1:0). $

The moduli space $M_{0,n}  $ coincides by \cite{buchstaber2021orbit} with the space of parameters of the main stratum $ F_{n}. $ In the paper \cite{tavakol2017chow} Tavakol described the space $  \mathcal{M}_{0,n} $ which is the compactification of $ M_{0,n} $, similarly as Keel did in \cite{keel1992intersection}, starting from $ X $ by the sequence of blow-ups.  First one blows-up three points $ X_{1,\dots,n-3,n-2},$ $X_{1,\dots,n-3,n-1},X_{1,\dots,n-3,n}. $ This blow-up separates all lines $ X_{I} $ for which $ |I|=n-3. $ Second, one blows-up the proper transform of all these lines. The process continues for $ |I|\geq 2. $ We note that $ X_{I} $ for $ |I|=2 $ is a complex $ (n-4) $-dimensional submanifold in the manifold obtained by the previous blow-up construction on $ X $, so the further blowing up at $ X_{I} $ does not change anything.
\par  The exceptional divisor of the blow-up along the subvariety $ X_{I} $ is denoted by $ D_{I}. $ \par Keel in \cite{keel1992intersection} showed that the Chow ring of  $  \mathcal{M}_{0,n} $ is generated by all divisor classes $ D_{I} $ and they satisfy following relations: \begin{align}\label{relacije}
	\sum_{\substack{i,j\in I\\p,q\notin I} }D_{I}=\sum_{\substack{i,p\in I\\j,q\notin I} }D_{I}=\sum_{\substack{i,q\in I\\j,p\notin I} }D_{I} 
\end{align} for any four distinct elements $ \{i,j,p,q\}\subset \{1,\dots,n\} $ an it d holds \[D_{I}=D_{I^{\complement}}, \quad  D_{I}D_{J}=0  \text{ unless }  I\subset J,I\subset J, I^{\complement}\subset J \text{ or } J^{\complement}\subset I.  \]
In addition, he proved that the canonical map from Chow groups to homology induces an isomorphism $ A_{\star}(\mathcal{M}_{0,n})\to H_{\star}(\mathcal{M}_{0,n}). $
\par In \cite{ceyhan2009chow} Ceyhan proved this isomorphism for the moduli spaces $ \mathcal{M}_{\mathcal{A}} $ of weighted pointed stable curves of genus zero generalizing Keel's result, but using significantly different method which appeals on the construction of the stratification for $ \mathcal{M}_{\mathcal{A}}. $ More precisely, he proved that the Chow groups $ A_{i}(\mathcal{M}_{\mathcal{A}}) $ and the homology groups $ H_{2i}(\mathcal{M}_{\mathcal{A}}) $ are isomorphic. We note that for some special weights $ \mathcal{A} $ similar results are proved by Losev and Manin as well in \cite{losev2000new}, \cite{manin2000new}.
\par In the next examples we describe the divisors $ D_{I} $ in $ \mathcal{M}_{0,n} $ for $ n=5,6. $
\begin{example}\label{example:M05}
	The moduli space $  \mathcal{M}_{0,5} $ is obtained as the blow-up of $ X=\C P^{1}\times \C P^{1} $ at the points $ X_{123}=(0,0),\,X_{124}=(1,1)\text{ and }X_{125}=(\infty,\infty). $  Let $ U $ be a neighborhood of the point $ ((0:1),(0:1)) $ defined by $ \{(c_{1}:1),(c_{2}:1)\}. $ We can take $ (c_{1},c_{2})\in U(0,0)\subset \C^{2} $ as coordinates in $ U. $ The blow-up $ bl_{X_{123}}X $ of the variety $ X $ in the neighborhood $ U $ belongs to $ U\times \C P^{1} $. Let $ \{(c_{1},c_{2}),(x_{1}:x_{2})\} $ be coordinates in  $ U\times \C P^{1} $ and $ \pi:bl_{X_{123}}X\to X $ a natural projection morphism. Then the coordinates of points in $ \pi^{-1}(U-X_{123}) $ are given by the equation \[ c_{1}x_{2}=c_{2}x_{1}. \] The exceptional set satisfies $ \pi^{-1}(X_{123})\simeq\C P^{1}$ and we put $D_{123}=X_{123}\times\C P^{1}=0\times0\times X_{123}. $ Similarly, we construct blow-ups for points $ X_{124} $ and $ X_{125}. $  
\end{example}
\begin{example}\label{example:M06}
	The moduli space $  \mathcal{M}_{0,6} $ is obtained first as blow-up of $ X=\C P^{1}\times \C P^{1}\times \C P^{1} $ at three points $ X_{1234},X_{1235},X_{1236} $ and then by blowing-up of proper transform of lines $ X_{I} $ where $ |I|=3. $  
	
	\par The blow-up of $ X $ at the point $ X_{1235}=(1,1,1) $ is similar to the case $ n=5.$ Let $ U $ be a neighborhood of the point $  X_{1235} $ defined by $ \{(c_{1}:1),(c_{2}:1),(c_{3}:1)\}. $ The coordinates of the  neighborhood $ U $ are given by $ (c_{1},c_{2},c_{3})\in U(1,1,1)\subset \C^{3} $ and the coordinates of the points in  $ \pi^{-1}(U-X_{1235}) $ satisfy
	\[ (1-c_{i})x_{j}=(1-c_{j})x_{i},\, \text{ for } 1\leq i,j\leq 3. \]
	The exceptional set satisfies  $ \pi^{-1}(X_{1235})=X_{1235}\times\C P^{2}. $
	\par Similarly, we construct blow-up at the point $ X_{1234}. $ Let a neighborhood of the point $  X_{1234} $ be defined  by $ \{(c_{1}:1),(c_{2}:1),(c_{3}:1)\}. $ The coordinates of points in  $ \pi^{-1}(U-X_{1234}) $ satisfy 
	\[ c_{i}x_{j}=c_{j}x_{i},\, \text{ for } 1\leq i,j\leq 3. \]
	
	\par After blowing-up $ X $ at these points, we proceed with blow-up of the lines $ X_{I} $, where $ |I|=3. $ In doing that we consider, for example, the proper transform of $ X_{124} $ denoted by $ X^{\prime}_{124}. $ By classical definition, we obtain $ X^{\prime}_{124}=\overline{\pi^{-1}(X_{123}-X_{1234})}, $ which is the closure of $ \pi^{-1}(X_{123}-X_{1234}). $ Since $ \pi^{-1}(X_{124}-X_{1234})=\{(0:1),(0:1),(c:c^{\prime}),(0:0:1)|c\neq0\},$ its closure is $ \{(0:1),(0:1),(c:c^{\prime}),(0:0:1)\}. $ We continue with blow-up at $ X^{\prime}_{124}. $ Let $ U $ be a neighborhood of $ X^{\prime}_{124} $ defined by $ \{(c_{1}:1),(c_{2}:1),(c:c^{\prime},(0:0:1))\} $. Blowing-up further at $ X^{\prime}_{124} $ in the neighborhood $ U $ we add the coordinates $ (x_{1}:x_{2})\in \C P^{1} $ and the points from $ \pi^{-1}(U-X_{124}) $ satisfy \[ c_{i}x_{j}=c_{j}x_{i},\, \text{ for } 1\leq i,j\leq 2. \] The exceptional set is  $ \pi^{-1}(X_{124})=X_{124}\times \C P^{1}. $ Thus, the divisors $ D_{I},\, |I|=3 $ are disjoint and each of them is isomorphic to $ \C P^{1}. $
	
	\par The divisors are:
	\begin{itemize}
		%		\item$ D_{1235}\simeq1\times1\times1\times\C P^{2}\times (\C P^{1})^{4}. $ The term $ \C P^{2} $ corresponds to the blow-up $ bl_{X_{1235}}$ and the terms $ (\C P^{1})^{4} $ to the further blow-ups $ bl_{X_{125}},bl_{X_{135}},bl_{X_{235}} $ and $ bl_{X_{123}}. $	
		\item$ D_{I}\simeq X_{I}\times\C P^{2}\times (\C P^{1})^{4} $ for $ |I|=4. $ The term $ \C P^{2} $ corresponds to the blow-up $ bl_{X_{I}} $ and the terms $ (\C P^{1})^{4} $ to the blow-ups at $ X_{J},\, |J|=3$ such that $ X_{I}\subset X_{J}. $  
		\item$ D_{I}\simeq X_{I}\times\C P^{1}$ for $ |I|=3 $. The term $ \C P^{1} $ corresponds to the blow-up at the $ X_{I}. $
		\item  $ D_{I} $ for $ |I|=2 $ is the proper transform of $ X_{I} $ after final blow-up.
		%		\item $ D_{123}\simeq X_{123}\times\times(\C P^{2})^{3}\times\dot{points}\times \C P^{1}. $ The terms $ \C P^{2} $ corresponds to the coordinates of blow-up at three points $ X_{1234},X_{1235},X_{1236} $ and the term $ \C P^{1} $ to the blow-up $ bl_{X_{123}}.$
	\end{itemize}
	
\end{example}

There exists the correlation between divisors $ D_{I} $ and a virtual spaces of parameters $ \tilde{F}_{\sigma} $. We will describe this correlation explicitly for $ n=5,6 $ in Sections \ref{Sec:HomX5} and \ref{Sec:HomX6}. In establishing this correlation we use the next important result proved in \cite{buchstaber2021orbit}.
\begin{theorem}\label{Th:ModuliUniversalSpace}
	The universal space of parameters $ \mathcal{F}_{n} $ is diffeomorphic to the space $ \mathcal{M}_{0,n}. $
\end{theorem}
We explicitly demonstrate this diffeomorphism for the cases $ n=5 $ and $ n=6 $, see also \cite{buchstaber2021orbit}, in the next examples.
\begin{example}\label{ex:map05}
	We follow notations from the example $ \ref{example:M05}. $ Let us consider map $ f:(\C P^{1})^{2}\to(\C P^{1})^{3} $ given by\[ f((c_{34}:c_{34}^{\prime}),(c_{35}:c_{35}^{\prime}))=((c_{34}:c_{34}^{\prime}),(c_{35}:c_{35}^{\prime}),(c_{34}^{\prime}c_{35}:c_{34}c_{35}^{\prime})). \] This map is not defined at the points $ X_{124},X_{126}. $ After we blow-up $ (\C P^{1})^{2} $ at the points $ X_{124},X_{126} $ as in Example \ref{example:M05} we obtain the variety denoted by  $ Y=bl_{126}bl_{X_{124}} X. $  The map $ f $ extends to the homeomorphism $ \overline{f} $ between $ Y $ and $ \overline{F}_{5}. $ The divisor $ D_{124}\simeq X_{124}\times\{(x_{1}:x_{2})\subset \C P^{1}\} $ maps to $ ((0:1),(0:1),(x_{1}:x_{2})) $ and similar for the divisor $ D_{126}. $ Now we blow-up $ Y $ at $ X_{125} $ and $ \overline{F}_{5} $ at $ (1:1)^{3} $ and conclude the homeomorphism between $ \mathcal{F}_{5} $ and $ \mathcal{M}_{0,5}. $ 
\end{example}

\begin{example}\label{ex:map06}
	Let us consider the map $ f:(\C P^{1})^{3}\to(\C P^{1})^{6} $ given by
	\begin{align*}
		&f((c_{34}:c_{34}^{\prime}),(c_{35}:c_{35}^{\prime}),(c_{36}:c_{36}^{\prime}))=\\&=((c_{34}:c_{34}^{\prime}),(c_{35}:c_{35}^{\prime}),(c_{36}:c_{36}^{\prime}),(c_{34}^{\prime}c_{35}:c_{34}c_{35}^{\prime}),(c_{34}^{\prime}c_{36}:c_{34}c_{36}^{\prime}),(c_{35}^{\prime}c_{36}:c_{35}c_{36}^{\prime}))
	\end{align*}
	
	This map is not defined at the sub-varieties $ \{X_{I}|I\in S\}$ where $S=\{124,134,234,126,136,236\} $ neither at the points $ X_{1234},X_{1236}. $ We denote by $ Y $ the blow-up of $ (\C P^{1})^{3} $ at this sub-varieties according to Example $ \ref{example:M06} $. The extension $ \overline{f} $ of the map $ f $ between $ Y $ and $ \overline{F}_{6} $ can be defined as follows:
	
	\begin{itemize}
		\item  The points of the divisors $ D_{1234} $ and $ D_{1236}, $ apart from the points for which the corresponding $ \C P^{2}$ component is $ (1:0:0)$ or $(0:1:0)$ or $(0:0:1) $ and which are blown-up by blowing up the lines $ X_{I}^{\prime},$ we map by \[  D_{1234}=((0:1),(0:1),(0:1),(x_{1}:x_{2}:x_{3}))\to ((0:1),(0:1),(0:1),(x_{1}:x_{2}),(x_{1}:x_{3}),(x_{2}:x_{3})),\]
		\[  D_{1236}=((1:0),(1:0),(1:0),(x_{1}:x_{2}:x_{3}))\to ((1:0),(1:0),(1:0),(x_{1}:x_{2}),(x_{1}:x_{3}),(x_{2}:x_{3})).\] It follows from Example \ref{example:M06} that the points of  $ \overline{f}(D_{1234}) $ belong to $ \overline{F}_{6}. $
		\item The divisor $ D_{I} $, for example if $ I=124 $ maps by \[  D_{124}=((0:1),(0:1),(c:c^{\prime}),(x_{1}:x_{2}))\to ((0:1),(0:1),(c:c^{\prime}),(x_{1}:x_{2}),(1:0),(1:0)).\]
		
	\end{itemize}
	It is left to sequentially blow-up the variety $ Y $ at $ X_{1235} $ and than at $ X_{125},X_{135},X_{235},X_{123} $ and in the result we obtain the space $ \mathcal{M}_{0,6}. $ On the other side, the blow-up of the space $ \overline{F}_{6} $ at the varieties $ (1:1)^{6} $ and $ \overline{F}_{345},\overline{F}_{346},\overline{F}_{356},\overline{F}_{456} $ gives $ \mathcal{F}_{6} $.
\end{example}

	\section{On $\Z _{2}$-homology of $G_{n,2}/T^n$}\label{Sec:On_Gn2_homology}
	\par	Using the theory presented in previous sections on description of topology of the orbit space $G_{n,2}/T^n$ and in particular the model~\eqref{Model:U_n}, we prove some general results on description of  $\Z _{2}$-homology of $G_{n,2}/T^n$. It follows from \eqref{Eq:space_X_n} that the  $\Z_2$-cycles in $G_{n,2}/T^n$ can be studied inductively on $n$. 
		
%		More precisely, denote by $C_{n,2} = \{C_{\omega}\times F_{\omega}\}$ the complex associated to $\hat{\mu}^{-1}(\stackrel{\circ}{\Delta _{n,2}})$ according to \eqref{CelijaC}. It follows  that 
%		\[
%		G_{n,2}/T^n = C_{n,2}\cup n\# C_{n-1,2}\cup n\# \Delta ^{n-2}\cup n(n-1)\# C_{n-2,2} \cup n(n-1)\#\Delta ^{n-3}\cup \ldots \]\[\cup n(n-1)\cdots 5\# \Delta _{4,2} \cup n(n-1)\cdots 5\# \Delta ^{3}.
%		\]
	\par	We consider the  cell decomposition of $F_{\omega}$ given by a cell decomposition of $F_{\sigma}$'s, $\sigma \in \omega$ obtained following Remark~\ref{cellspaceparam}. We denote  further  by $P_{\sigma ^{'}}$ a facet  in $\stackrel{\circ}{\Delta} _{n,2}$ of an admissible polytope $P_{\sigma}$  and by $C_{\omega ^{'}}$ a facet in $\stackrel{\circ}{\Delta}_{n,2}$  of a chamber $C_{\omega}$, and by $P_{\bar{\sigma}}, C_{\bar{\sigma}}$ those facets which belong to $\partial \Delta _{n,2}$. 
		For a cell $C_{\omega}\times  e_{\omega}^{\sigma }$, where $e_{\omega}^{\sigma }\subset  F_{\sigma }$,   $\sigma \in \omega$  we first emphasize the following.
		\begin{rem}\label{Rem:ManjeDIm}
			If $C_{\omega ^{'}}\subset \stackrel{\circ}{P}_{\sigma}$  is a facet of $C_{\omega}$, than  only the cell $C_{\omega ^{'}}\times e_{\omega}^{\sigma}$ which belongs to $W_{\sigma}/T^n$ contributes to the boundary of $C_{\omega}^{\sigma}\times e_{\omega}^{\sigma}$.  This is because of the following: if  a cell  $C_{\omega ^{'}}\times e_{\omega ^{'}}^{\tilde{\sigma}}$ is  in the closure of $C _{\omega}^{\sigma}\times e_{\omega}^{\sigma}$, then   the points from $W_{\tilde{\sigma}}$ have  zero all  P\"ucker coordinates which are zero for the points from  $W_{\sigma}$, and some more. In particular,   $P_{\tilde{\sigma}}\subset P_{\sigma}$.
			Thus, taking into account  the cell decompositions  for    $F_{\sigma}$ and $F_{\tilde{\sigma}}$  from Remark~\ref{cellspaceparam}, this implies that if a cell $C_{\tilde{\omega}}\times e_{\tilde{\omega}}^{\tilde{\sigma}}$ is in the closure of a cell $C_{\omega}\times e_{\sigma}^{\omega}$,  then $\dim e_{\tilde{\omega}}^{\tilde{\sigma}} < \dim e_{\omega}^{\sigma}$.  It follows that $C_{\omega ^{'}}\times e_{\omega ^{'}}^{\tilde{\sigma}}$ can not contribute to $\partial ( C_{\omega}\times e_{\omega}^{\sigma})$. 
		\end{rem}
		Therefore,  the boundary operator  for $C _{\omega}^{\sigma}\times e_{\omega}^{\sigma}$ is defined  by differentiating the following cases. 
		\begin{itemize}
			\item If $C_{\omega}$ has no facet  on $\partial \Delta _{n,2}$ and either  has no facets on $\partial P_{\sigma}$ or $e_{\omega}^{\sigma}$ is a point,  then 
			\[\partial (C_{\omega}\times e_{\omega}^{\sigma } )= C_{\omega}\times \partial ^{\omega}e_{\omega}^{\sigma } + \sum _{ C_{\omega ^{'}}\; \text{facet of}\; C_{\omega}} C_{\omega ^{'}}\times e_{\omega}^{\sigma },\]
			\item If $C_{\omega}$ has no facet on $\partial \Delta _{n,2}$,  but is has a facet  $C_{\omega^{'}}^{\sigma}$ on $\partial P_{\sigma}$  then 
			
			\[\partial (C_{\omega}\times e_{\omega}^{\sigma } )= C_{\omega}\times \partial ^{\omega}e_{\omega}^{\sigma } + \sum _{\substack{ C_{\omega ^{'}}\; \text{facet of}\; C_{\omega}\\
					C_{\omega ^{'}}\neq C_{\omega ^{'}}^ {\sigma}}}C_{\omega} \times e_{\omega}^{\sigma }.\] 
				\begin{rem}\label{Nagranici}
				If $C_{\omega}$ has a facet $C_{\bar{\omega}}$  on $\partial  \Delta _{n,2}$ then $C_{\bar{\omega}}=\cap  P_{\bar{\sigma}}$ for the  facets  $P_{\bar{\sigma}}$ of $P_{\sigma}$, where  $\sigma\in   \omega$.  In addition, in~\cite{buchstaber2019foundations}, page 618, it is defined the continuous surjective map   $\eta_{\sigma, \bar{\sigma}} : F_{\sigma}\to F_{\bar{\sigma}}$, which gives continuous surjection $\eta _{\omega, \bar{\omega}}: F_{\omega} \to F_{\bar{\omega}}$.  Thus,  the chain map  $\eta ^{\#}_{\omega, \bar{\omega}} : C_{\#}(F_{\omega}) \to C_{\#}(F_{\bar{\omega}})$ induced by $\eta _{\omega, \bar{\omega}} : F_{\omega}\to F_{\bar{\omega}}$ is surjective by  Remark~\ref{cellspaceparam}
			\end{rem}
			\item  If  $C_{\omega}$ has a facet on $\partial \Delta _{n,2}$ 
				% $h_{\bar{\sigma}} (C_{\bar{\omega}}\times l^{\bar{\sigma }})\subset \partial (h_{\sigma }(C_{\omega}\times l_{\omega}^{\sigma }))$
				we obtain that 
				\[\partial (C_{\omega}\times e_{\omega}^{\sigma } )= \partial ^{0}(C_{\omega}\times e_{\omega}^{\sigma})   + \sum_{C_{\bar{\omega}}\; \text{facet of}\; C_{\omega}}  C_{\bar{\omega}}\times \eta_{\omega, \bar{\omega}}(e_{\omega} ^{\sigma}), \]
			%if $h_{\bar{\sigma _{0}}} (C_{\bar{\omega}}\times l^{\bar{\sigma _{0}}})\subset \partial (h_{\sigma _{0}}(C_{\omega}\times l_{\omega}^{\sigma _{0}}))$.			
		\end{itemize}

		where $\partial ^{0}$ denotes the boundary component indicated in the previous two cases.
%	\red{we define}	The chain $l_{\bar{\omega}}^{\bar{\sigma}}$ in $F_{\bar{\sigma}}\subset F_{\bar{\omega}}$ is defined by $    l_{\bar{\omega}}^{\bar{\sigma}} = \eta ^{\#}_{\sigma, \bar{\sigma}}(e_{\omega}^{\sigma})$, where $\eta_{\sigma, \bar{\sigma}} : F_{\sigma}\to F_{\bar{\sigma}}$ is  a continuous map defined in~\cite{buchstaber2019foundations}, page 618.
		
		For the space of parameters $F_{\omega}$ we note the following.
		\begin{theorem}\label{Th:homology_F}
			The homology groups of non-maximal dimension of the space of parameters $F_{\omega}$ are spanned by the outgrows $F_{\sigma}$, $\sigma \in \omega$  in the compactification of $F_{n}$ to $F_{\omega}$.
		\end{theorem}
		\begin{proof}
			As we pointed out the  space of parameters of a chamber $C_{\omega}$ is a compactification of the space of parameters of the main stratum $F_n$ by the spaces of parameters $F_{\sigma}$ of the strata $W_{\sigma}$, where $\sigma \in \omega$. By Theorem~\ref{Lem:ProjG} there exists continuous projection $p_{\omega} : \mathcal{F}_{n}\to F_{\omega}$ defined by the projections  
			$p_{\sigma} : \tilde{F}_{\sigma}\to F_{\sigma}$. It follows from~\cite{buchstaber2020resolution}, proof of Proposition 6.1 or~\cite{buchstaber2019toric}   that $\tilde{F}_{\sigma} = \bar{F}_{l} \times \bar{F}_{s}\times F_{m}$, where $F_{m} \cong F_{\sigma}$ and   $l+s+m\leq n+4,\,0\leq l,s,m\leq n-1.$ It implies that the the induced homomorphism between  the cycles on $\mathcal{F}_{n}$ and the cycles on $F_{\omega}$ is surjective. Since by~\cite{keel1992intersection}, the homology groups  for $\mathcal{F}_{n}$ are spanned by the divisors outgrowing in the compactification of $F_n$ to $\mathcal{F}_{n}$, it follows that  the homology groups for $F_{\omega}$ are spanned by the outgrows in the compactification of $F_{n}$ to $F_{\omega}$.
		\end{proof}
		\begin{rem}
			In~\cite{buchstaber2024weighted} it is proved that
			$F_{\omega} \cong \mathcal{M}_{\mathcal{A}} =\mathcal{M}_{0,\mathcal{A}}$, which is a moduli space of weighted stable genus zero curves with $n$ marked distinct points with a set of weights $\mathcal{A}$, defined in~\cite{hassett2003moduli} in general setting. It is also showed that the projection $p_{\omega}:\mathcal{F}_{n}\cong\mathcal{M}_{0,n} \to F_{\omega}$ coincides with the reduction morphism $\rho _{\mathcal{A}} : \mathcal{M}_{0,n}\to \mathcal{M}_{\mathcal{A}}$, defined in \cite{hassett2003moduli}.
			\par On the other side, in~\cite{ceyhan2009chow} it is obtained the
			description of the homology groups $H_{\ast}(\mathcal{M}_{\mathcal{A}})$
			in terms of a stratification for $\mathcal{M}_{\mathcal{A}}$. Moreover,
			it is pointed  that the generators and	relations in  $H_{\ast}(\mathcal{M}_{A})$ are induced from those in $H_{\ast}(\mathcal{M}_{0, n})$ by the reduction morphism
			$\rho _{\mathcal{A}} $. 

				\par Therefore, Theorem~\ref{Th:homology_F} provides the formulation and the proof of
				this result in terms of the topological  ingredients of the model $(U_n,
				p_n)$ for $X_n$.
		\end{rem}
		In  what follows we consider $\Z_2$-chains in $G_{n,2}/T^n$ which map  by the induced moment map $\hat {\mu}$ to the interior $\stackrel{\circ}{\Delta _{n,2}}$. Any such  chain has the form
		\[
		c= \sum _{\omega}C_{\omega} \times l_{\omega},\;\; C_{\omega}\subset \stackrel{\circ}{\Delta _{n,2}},
		\]
		and $l_{\omega}$ is a chain in $F_{\omega}$.   The chain  $l_{\omega}$ can be represented by $l_{\omega} =\sum _{\sigma\in\omega} l_{\omega} ^{\sigma}$ for a chain $l_{\omega}^{\sigma}$ in $F_{\sigma}$. For $F_{\sigma}=F$ we put $l_{\omega}^{\sigma}=l_{\omega}^{F}$.
		
		For  the boundary of a chain  $c$  we note  that $\partial _{p}c = \partial ^{\circ}_{p}c + \partial _{p}^{'}c$, where $\hat{\mu} (\partial _{p}^{\circ}c)\subset \stackrel{\circ}{\Delta} _{n,2}$ and  $\hat{\mu} (\partial _{p}^{'}c)\subset \partial \Delta _{n,2}$ and $p=\dim c -1$. We describe this decomposition  of $\partial _{p}c$ more explicitly using Remark \ref{Rem:ManjeDIm}:
		\begin{lemma}
		The boundary component
		
		\begin{itemize}
			\item  $\partial _{p}^{'}c$ of $c$ has the form
			\begin{equation}
					\partial _{p}^{'}c =  \sum_{\bar{\omega}}C_{\bar{\omega}}\times \eta^{\#}_{\omega, \bar{\omega}}(l_{\omega}).
			\end{equation}
			
			%  last sum goes over those $\bar{\sigma}$ such that  $h_{\bar{\sigma}}(C_{\bar{\omega}}\times l_{\bar{\omega}}^{\bar{\sigma}})\subset \partial _{p} %h_{\sigma} (C_{\omega}\times l_{\omega}^{\sigma} )$.
			
			\item 
			$\partial ^{\circ}_{p}c$ has the form
			\begin{equation}\label{boundch}
				\partial _{p}^{\circ} c = \sum   _{\omega} \sum _{\substack{\omega ^{'} \\ C_{\omega ^{'}}\subset \stackrel{\circ}{\Delta}_{n,2} }}\sum _{\substack{\sigma \in \omega  \\ C_{\omega ^{'}}\subset \stackrel{\circ}{P_{\sigma}}}}C_{\omega ^{'}}\times l_{\omega}^{\sigma} + \sum _{\omega} C_{\omega}\times \partial ^{\omega}l_{\omega},
			\end{equation}
			where $\partial ^{\omega}l_{\omega}$ denotes the boundary of $l_{\omega}$ in $F_{\omega}$.
		\end{itemize}
	\end{lemma}
		
		\begin{rem}\label{Rem:Cycle_dim_0}
Note that $H_{0}(X_n; \Z_2) \cong \Z_2$ as $X_n$ is path connected.
		\end{rem}
		
		The following known result~\cite{buchstaber2019toric} is easy to deduce:
		
		\begin{lemma}\label{lemma_3n-7}
			\[
			H_{3n-7} (X_n, \Z _2) \cong \Z _2.
			\]
		\end{lemma}
		\begin{proof}
			The top degree chain in $X_n$ is of the form $c= \sum C_{\omega}\times l_{\omega}$, where $C_{\omega}$ is a chamber of maximal dimension $n-1$ and $l_{\omega}$ is a chain in $F_{\omega}$ of maximal dimension $2n-6$. Since $\dim F_{\sigma}< \dim F_{\omega}$ for any $\sigma \in \omega$ and $F_{\sigma}\neq F$, it follows that $l_{\omega} = l_{\omega}^{F}$ is a chain in $F$. Using~\eqref{boundch} we obtain that
			\[
			\partial ^{\circ} _{3n-8} c= \sum _{\omega}\sum _{\omega ^{\lq{}}\subset \omega, \; C_{\omega ^{'}}\subset \stackrel{\circ}{\Delta _{n,2}}}C_{\omega ^{\lq{}}}\times l_{\omega}^{F} + \sum _{\omega}C_{\omega}\times \partial _{2n-7}^{\omega} l_{\omega}^{F}.
			\]
			Note that the space of parameters of any  chamber from the boundary of $\Delta _{n,2}$  has the dimension less then 2n-6, which implies that $\partial _{3n-8}^{'}c=0$.  Thus, $c$ is a cycle if and only if the following conditions are satisfied:
			\begin{enumerate}
				\item
				$\partial _{2n-7}^{\omega}l_{\omega}^{F} =0$ for any $\omega$;
				\item 
				$l_{\omega}^{F}=l^{F}$,  that is $l_{\omega}$ does not depend on $\omega$ since any $C_{\omega}^{\lq{}}$ contributes  twice in $\partial _{3n-8}c$ as the  chambers $C_{\omega}$ are of maximal dimension;
				\item chambers $C_{\omega}$ which contribute to $c$ give all chambers in $\Delta _{n,2}$ of dimension $n-1$.
			\end{enumerate}
			It follows that $c=  (\sum _{\omega} C_{\omega})\times l^{F}$. We have that $ \partial^{\mathcal{F}_n} l^{F}\subset \bar{F}$ as $ \mathcal{F}_n $ is a wonderful compactification of $ \bar{F}. $ Since $ \dim l^{F}=\dim F $ and $ \dim(\bar{F}\backslash F)=\dim F-2 $ it follows that $ \partial^{\mathcal{F}_n} l^{F}=\partial^{F} l^{F}$. The fact that $ F\subset F_{\omega} $ and $H_{2n-6}(\mathcal{F}_n,\Z_{2})\cong \Z_{2}$ concludes the proof. 
		\end{proof} 
		We further omit an index $ p $ in $ \partial_{p} c, $ as well as $ \omega $ in $ \partial^{\omega} l_{\omega} $ to simplify notation.
		\subsection{$\Z _2$ - cycles over $\stackrel{\circ}{\Delta}_{n,2}$}
		
		\begin{theorem}\label{main}
			Any $\Z_2$-cycle $c=\sum C_{\omega}\times l_{\omega}$ in $G_{n,2}/T^n$ such that $\hat{\mu}(c)\subset \stackrel{\circ}{\Delta}_{n,2}$ is homologous to a cycle  which does not contain a summand  for which  either one of the following conditions is satisfied
			\begin{enumerate}
				\item $\dim C_{\omega}< n-1$;
				\item  $l_{\omega}^{F}$ is a chain in $F$ and $\dim l^{F}_{\omega}<\dim F$;
				\item $l_{\omega}^{\sigma}$ is a chain  in some $F_{\sigma}$ and $\dim l_{\omega}^{\sigma}<\dim F_{\sigma}$.
			\end{enumerate}
		\end{theorem}

		\begin{proof}
			
			(1)  Let  $\dim C_{\omega}<n-1$ and consider such $C_{\omega}$ of lowest degree assuming it is not zero. Let $l_{\omega} = \sum _{\sigma\in \omega} l_{\omega}^{\sigma}$, and consider $\sigma\in \omega$ such that $l_{\omega}^{\sigma}\neq 0$ for $ \dim P_{\sigma}=n-1. $ Then $C_{\omega}\subset \stackrel{\circ}{P}_{\sigma}$ and there exists a facet $C_{\omega ^{'}}$ of $C_{\omega}$ such that $C_{\omega ^{'}}\subset \stackrel{\circ}{P}_{\sigma}$.  It follows that  it is defined $C_{\omega ^{'}}\times l_{\omega}^{\sigma}$ and it is a summand in $\partial  c=0$. So, there exists a chamber 
			$C_{\tilde{\omega}}\subset \stackrel{\circ}{P}_{\sigma}$ such that $C_{\omega ^{'}}$
			is a facet of $C_{\tilde{\omega}}$ and 
			$C_{\tilde{\omega}}\times l_{\omega}^{\sigma}$ is a summand in $c$. If $C_{\omega}$ is a point and by Remark \ref{Rem:Cycle_dim_0} we can assume that $ l_{\omega}^{\sigma} $ is not a point, then there obviously  exists  an $1$-dimensional chamber $C_{\tilde{\omega}} \subset \stackrel{\circ}{P}_{\sigma}$ such that $C_{\omega}$ is its vertex. Continuing for the other vertex of $ C_{\tilde{\omega}} $ and so on we can achieve to obtain a chain $ \sum C_{\tilde{\omega}} $ whose other vertex is from $ P_{\sigma^{\prime}}\subset \Ins{\Delta}_{n,2}$ that is a facet of $ P_{\sigma}. $ Then $ \partial (\sum C_{\tilde{\omega}}\times l_{\omega}^{\sigma})=C_{\omega}\times l_{\omega}^{\sigma}+\sum C_{\tilde{\omega}}\times \partial l_{\omega}^{\sigma}. $ So the lowest dimension of chambers in $ c $ can be assumed to be $ \geq 1. $
			
			If we continue in this way we obtain that $c$ contains the  chain $ d_{\omega, \sigma}= c_{\omega, \sigma}\times l_{\omega}^{\sigma}$, where $c_{\omega, \sigma} = \sum C_{\tilde{\omega}} \subset \stackrel{\circ}{P}_{\sigma}$ and  $\partial c_{\omega, \sigma}=0$ either $\partial c_{\omega, \sigma}\subset \partial P_{\sigma}$.     Thus, $c^{'}_{\omega, \sigma}=\partial c_{\omega, \sigma}$ is a cycle in $\partial P_{\sigma}\cong S^{n-2}$, while $\dim c_{\omega, \sigma}^{'}<n-2$. It follows that there exist a chain $\tilde{c}_{\omega, \sigma}$ on $\partial P_{\sigma}$ such that $\partial \tilde{c}_{\omega, \sigma} = c^{'}_{\omega, \sigma}$. Thus, $c_{\omega, \sigma} + \tilde{c}_{\omega, \sigma} =\partial \hat{c}_{\omega, \sigma}$  for some chain $\hat{c}_{\omega, \sigma}$ in $P_{\sigma}$. For $c_{\omega, \sigma}$ being a cycle we have that $\tilde{c}_{\omega, \sigma}=0$.
			If  $l_{\omega}^{\sigma}\neq 0$ only for $\dim P_{\sigma} =n-2$, then $ l_{\omega}^{\sigma} $ is a point. For   $\dim C_{\omega}< n-2 $ the same argument can be applied.  For  $\dim C_{\omega}=n-2$ we obtain the chain $d_{\omega, \sigma}=\sum C_{\omega}$ which contains all chambers of dimension $n-2$ in $\stackrel{\circ}{P}_{\sigma}$ and $\partial d_{\omega, \sigma} \subset \partial P_{\sigma}\subset \partial \Delta _{n,2}\simeq S^{n-2}$, where the last inclusion follows from  the description  of $P_{\sigma}$  given in Theorem~\ref{Th:admissiblepolytope}, so the same argument applies.
			\par If $\partial ^{'}d_{\omega, \sigma} \neq 0$, then there exists $C_{\omega}^{'}$ in a facet  $P_{\sigma ^{'}}$  of $P_{\sigma}$ such that $C_{\omega}^{'} \times l_{\omega ^{'}}^{\sigma ^{'}}$ is contained in $\partial ^{'}d_{\omega, \sigma}$. Thus, there exists summand  $C_{\hat{\omega}} \times l_{\hat{\omega}}^{\hat{\sigma}}$ in $c$  which contains $C_{\omega}^{'} \times l_{\omega ^{'}}^{\sigma ^{'}}$  in its boundary.  Repeating the same   argument for $C_{\hat{\omega}} \times l_{\hat{\omega}}^{\hat{\sigma}}$ we obtain the chain $d_{\hat{\omega}, \hat{\sigma}}$ and the chains $\tilde{c}_{\hat{\omega}, \hat{\sigma}}$, 
			$\hat{c}_{\hat{\omega}, \hat{\sigma}}$. Continuing this procedure for  all non zero boundary summands for $d_{\omega, \sigma}$ and then the  same for the chain $d_{\omega,\sigma}+\sum_{\hat{\omega},\hat{\sigma}}d_{\hat{\omega}, \hat{\sigma}}$ until we obtain a chain  $d = \sum _{\omega, \sigma} d_{\omega, \sigma}$ such that $\partial ^{'}d=0$.
			\par	Consider the chain $\hat{C} = \sum _{\omega, \sigma}  \hat{c}_{\omega,\sigma}^{0} \times l_{\omega}^{ \sigma}$, where  $\hat{c}_{\omega,\sigma}^{0}$ means that we consider only those chambers from the chain $\hat{c}_{\omega,\sigma}$ which belong to $\stackrel{\circ}{P}_{\sigma}$.  Then $\partial ^{0}\hat{C} = d +  \sum _{\omega, \sigma}  \hat{c}_{\omega,\sigma} \times \partial l_{\omega}^{ \sigma}$. Note also that  $\partial ^{'} \hat{C} =0$ since in the second repeating step of the previous procedure  we can always choose a chain $\tilde{c}_{\hat{\omega}, \hat{\sigma}}$  such that in $P_{\sigma ^{'}}$ it  overlaps with $\tilde{c}_{\omega, \sigma}$. 
			\par Altogether we obtain that  $\partial \hat{C}=  d + \sum _{\omega, \sigma}  \hat{C}_{\omega,\sigma} \times l_{\omega}^{ \sigma}$, which means that $d$ is homologous to a chain whose chambers have dimension greater by one. In the case $\partial ^{'}d_{\omega, \sigma} = 0$ we note that the chain $\hat{c}_{\omega, \sigma}$ place the role of $\hat {C}$.  Repeating this argument we prove the statement. 
			
			%Thus, $c^{'}_{\omega, \sigma}=\partial c_{\omega, \sigma}$ is a cycle in $\partial P_{\sigma}\cong S^{n-2}$, whille $\dim c_{\omega, \sigma}^{'}<n-2$. It follows that there exist a chain $\tilde{c}_{\omega, \sigma}$ on $\partial P_{\sigma}$ such that $\partial \tilde{C}_{\omega, \sigma} = c^{'}_{\omega, \sigma}$. Thus, %$c_

			%In other words, the boundary facets  $C_{\omega ^{'}}$  for $C_{\omega}$ appear only within the cells of the main stratum.  This is because of the following:  let a facet  $C_{\omega ^{'}}\subset \stackrel{\circ}{P}_{\sigma}$, then some Pl\"ucker coordinates of the  stratum  $W_{\sigma}$ are zero, so when approaching a cell $C_{\omega ^{'}}\times l_{\omega ^{'}}^{\sigma}$  from $W_{\sigma}/T^n$ by a cell of the main   stratum the dimension  space of parameters of the cell from the main stratum  
			%decreases or it may happen to increase  producing  virtual space of parameters
			%, see~\cite{buchstabertoric2019} and~\cite{buchstaberresolution2021}.   
			\par (2) Let us consider a cycle $c= \sum C_{\omega}\times l_{\omega}$, where $\dim C_{\omega}=n-1$. It follows from~\eqref{boundch} that $\partial l_{\omega}=0$ for any $\omega$. If $ l_{\omega}^{F}\neq 0 $ then $c$ contains all $(n-1)$-dimensional chambers and $l_{\omega}^{F}=l^{F}$ for  all $ \omega $, that is $c$ contains the chain $d= \sum _{\omega}C_{\omega}\times l^{F}$.
		\par Assume that  $\dim l^{F}  < \dim F$. By Theorem \ref{Th:homology_F} we have that $l_{\omega}= \tilde{l}_{\omega} + \partial \bar{l}_{\omega}$, where $\tilde{l}_{\omega}$ is a chain of fundamental cycles determined by divisors.  It follows that $l_{\omega}^{F}\subset \partial \bar{l}_{\omega}$. Now, $\bar{l}_{\omega}= \bar{l}_{\omega}^{F} + \sum _{\sigma}\bar{l}_{\omega}^{\sigma}$, so $\partial \bar{l}_{\omega} = (\partial \bar{l}_{\omega}^{F})^{F} + \sum _{\sigma}\partial(\bar{l}_{\omega}^{F})^{\sigma} + \partial (\sum _{\sigma} \bar{l}_{\omega}^{\sigma})$. Thus, $l_{\omega}^{F}= (\partial  \bar{l}_{\omega}^{F})^{F}$. 
			\par For the chain $\bar{d}= \sum _{\omega} C_{\omega}\times \bar{l}_{\omega}^{F}$,  we obtain that  
			\[
			\partial \bar{d}= d + \sum _{\sigma\in \omega} C_{\omega}\times  (\partial  \bar{l}_{\omega}^{F})^{\sigma} =  d+ \tilde{d},
			\]
			where the chain $\tilde{d}$ does  not contain a summand from the main stratum. It follows that $c$ is homologous  to the chain $c+d+\tilde{d}$ which does not contain a chain from the main stratum.

			(3) We again consider a cycle $c= \sum C_{\omega}\times l_{\omega}$, $\dim C_{\omega}=n-1$, which implies that $\partial ^{\omega} l_{\omega}=0$ for any $\omega$.    Assume that there exists $ l_{\omega}^{\sigma}\neq 0$ such that $\dim l_{\omega}^{\sigma} < \dim F_{\sigma}$ and consider such  $\sigma _{0}$ for which $\dim F_{\sigma _{0}}$ is maximal among all $ \omega $'s. It follows that $l_{\omega}$ is a chain in $ F_{<\sigma_0>} = \cup _{\sigma \in \omega} F_{\sigma}$, where  $\dim F_{\sigma}\leq \dim F_{<\sigma_0>}$. Note that $F_{<\sigma_0>}$ is a compact space which is a compactification of all $F_{\sigma}$, $\sigma \in \omega$ and $\dim F_{\sigma}=\dim F_{<\sigma_0>}$. Since by \cite{buchstaber2024weighted} any $F_{\sigma}$ is homeomorphic to some $F_m$, it follows by Theorem \ref{Th:homology_F} that the homology groups of $F_{<\sigma_0>}$ are spanned by the fundamental cycles of  contributing $F_{\sigma}$. Thus, $l_{\omega} = \tilde{l}_{\omega} + \partial \bar{l}_{\omega}$, where $\tilde{l}_{\omega}$ are a chain  of the fundamental cycles spanning homology of $F_{<\sigma_0>}$ and $ \bar{l}_{\omega}$ is a chain in $F_{<\sigma_0>}$. We deduce that $l_{\omega}^{\sigma _0} \subset \partial \bar{l}_{\omega}$, that is $l_{\omega}^{\sigma_{0}} = (\partial \bar{l}_{\omega}^{\sigma _0})^{\sigma _0}$.
			\par On the other hand, for the summand  $C_{\omega}\times l_{\omega}^{\sigma_0}$  in the cycle $c,$ there must exist a facet, so we obtain the chain $d_{\sigma _0} =\sum _{\omega} C_{\omega}\times l _{\omega}^{\sigma _0}$, where $C_{\omega}$'s  give  all chambers of dimension $(n-1)$  in $P_{\sigma_{0}}$ and $l_{\omega}^{\sigma_{0}}$ does not depend on $\omega$, that is $ d_{\sigma_{0}}=P_{\sigma_0,n-1}\times l_{\omega}^{\sigma _0}. $
			\par For the chain $\bar{d}_{\sigma _0}= P_{\sigma_0,n-1}\times \bar{l}_{\omega}^{\sigma _0}$,  we obtain that  
			\[
			\partial \bar{d} _{\sigma _0}= d_{\sigma _0} + \sum _{\sigma\in \omega} C_{\omega}\times  (\partial  \bar{l}_{\omega}^{\sigma _0})^{\sigma}  + \sum _{C_{\bar{ \omega}}\in P_{\bar{\sigma_0}}} C_{\bar{\omega}}\times \eta^{\#}_{\sigma_{0}, \bar{\sigma_{0}}}(\bar{l}_{\omega}^{\sigma _0}) =  d_{\sigma _0} +  d^{'}_{\sigma _0} + \bar{d}_{\sigma _0},  
			\]
			Therefore $c$ is homologous to the chain $c +d_{\sigma _0} + d^{'}_{\sigma _0} +\bar{d}_{\sigma _0}$, which does not have a summand with a chain from $F_{\sigma _0}$. Moreover, if $l_{\omega}^{\sigma}$ contributes to $d^{'}_{\sigma _0}$  then $\dim F_{\sigma} < \dim F_{\sigma _0}$, while  $ \hat{\mu}(\bar{d}_{\sigma _0})\subset \partial \Delta _{n,2}.$

			Note that $\partial d_{\sigma _0}  = \sum _{\bar{\sigma_{0}}} P_{\bar{\sigma_{0}},n-2}\times \eta^{\#}_{\sigma_{0}, \bar{\sigma_{0}}}(l_{\omega}^{\sigma _0}) $, where $P_{\bar{\sigma},n-2}\subset \partial \Delta _{n,2}$ is determined by a facet $ P_{\bar{\sigma}} $ of $P_{\sigma}$. Since $\partial d_{\sigma _0}\subset \partial c=0$,  it follows that  for any $\bar{\sigma_{0}}$ there exists a chain $d_{\tilde{\sigma}} = P_{\tilde{\sigma},n-1}\times l^{\tilde{\sigma}}_{\tilde{\omega}}$ where $ \tilde{\omega} $ is arbitrary such that $ C_{\tilde{\omega}}\subset P_{\tilde{\sigma},n-1},\, \dim( C_{\tilde{\omega}})=n-1 $, then $P_{\bar{\sigma_{0}}}$ is a  facet of $P_{\tilde{\sigma}}$ and $\partial ^{'}d_{\tilde{\sigma}} = \partial ^{'}d_{\sigma _0}$ on $P_{\bar{\sigma_{0}}}$. If continue in this way we obtain the cycle  $d= \sum _{\sigma}d_{\sigma}.$
			\par Note further the following:  a facet $P_{\bar{\sigma}}$  belongs to some hypersimplex $\Delta _{n-1,2}$ on $\partial \Delta _{n,2}$ as   its space of parameters is not a point. Therefore, $P_{\bar{\sigma}}$  is  the intersection of $P_{\sigma}$ and a hyperplane $x_i=0$ for some $i$.  The stratum $W_{\bar{\sigma}}$  is a common boundary stratum for $W_{\sigma}$ and $W_{\tilde{\sigma}}$ and it  is obtained  from them by removing the $i$-th row from their matrix record. This  gives that   $F_{\sigma}\cong F_{\tilde{\sigma}}$.  Therefore, we can take $l_{\omega}^{\sigma_{0}} = l_{\tilde{\omega}}^{\tilde{\sigma}}$ and $\bar{l}_{\omega}^{ \sigma_{0}}=\bar{l}_{\tilde{\omega}}^{ \tilde{\sigma}},$ which implies that $\eta^{\#}_{\sigma_{0}, \bar{\sigma_{0}}}(\bar{l}_{\omega}^{\sigma _0}) =\eta^{\#}_{\tilde{\sigma_{0}}, \bar{\sigma_{0}}}(\bar{l}_{\tilde{\omega}}^{\tilde{\sigma}})$.
			\par Consider the chain  $\bar{d}= \sum _{\sigma}d_{\sigma} = \sum _{\sigma} P_{\sigma,n-1}\times \bar{l}_{\omega}^ {\sigma}$.  Then 
			\begin{equation}\label{granicabd}
				\partial \bar{d}= d + \sum_{\sigma}\sum_{\bar{\sigma}} P_{\bar{\sigma},n-2}\times\eta ^{\#}_{\sigma, \bar{\sigma}}( \bar{l}_{\omega}^ {\sigma}) =d. 
			\end{equation}	
			
		\end{proof}

		We emphasize that the proofs of the second  and third statements in Theorem~\ref{main} imply:
		
		\begin{cor}
			Any chain  $c=\sum C_{\omega}\times l_{\omega}$ in $X_n$ such that $\dim C_{\omega} = n-1$ for any $\omega$   and $\partial^{0}c=0$ is homologous to a  chain which does not contain a summand for which $l_{\omega}^{F}\neq 0$ and $\dim l_{\omega}^{F}<\dim F$.
		\end{cor}
		
		\begin{cor}
			Any chain $c=\sum C_{\omega}\times l_{\omega}$ in $X_n$ such that  $\dim C_{\omega} = n-1$ for any $\omega$   and $\partial^{0}c=0$  is homologous to  a chain $c^{'} + c_0$, where $\hat{\mu}(c^{'})\subset \partial \Delta _{n,2}$ and $\hat{\mu}(c_0) \subset \stackrel{\circ}{\Delta}_{n,2}$ and $c_0$   does not contain a summand  for which   $l_{\omega}^{\sigma}\neq0$ and $\dim l_{\omega}<\dim F_{\sigma}$.
		\end{cor}
		
		Theorem~\ref{main} leads to the following explicit result.		
		\begin{cor}\label{lemma_3n-8}
			\[
			H_{3n-8}(X_n, \Z_2) =0
			\]
		\end{cor}
		
		\begin{proof}
			The cycles of dimension $3n-8 $ are of the form $c= \sum _{\omega}C_{\omega}\times l_{\omega}$, where $\dim C_{\omega}=n-1$ and $l_{\omega}$ is a chain in $F_{\omega}$ and $\dim l_{\omega}= 2n-7$. 
			%Recall that  $l_{\omega} = \sum _{\sigma} l_{\omega, \sigma}$, where $L_{\omega, \sigma}$ is a chain in $F_{\sigma}$,
			Since  $\dim F_{\sigma}\leq 2n-6$ it follows that $l_{\omega}=l_{\omega}^{F}$ is a chain in $F$.  As $\dim l_{\omega}^{F} <\dim F$ we deduce from  Theorem~\ref{main}  that $c$ is homologous to zero. 
		\end{proof}
				
		Using Theorem \ref{main} we deduce more explicit expressions for the cycles in $X_n$.

\begin{cor}\label{admcycle}
Any cycle $c$  in $X_n$ such that $\hat{\mu}(c)\subset \stackrel{\circ}{\Delta}_{n,2}$ is homologous to 
\[
c = \sum _{\sigma} P_{\sigma, n-1}\times l^{\sigma},
\]
where   $P_{\sigma, n-1}$ is the union of the chambers of dimension $n-1$ from  $\stackrel{\circ}{P}_{\sigma}$ of an admissible polytope   $P_{\sigma}$,  $\dim P_{\sigma}=n-1$ and  $l^{\sigma}$ is a chain in $F_{\sigma}$ of maximal dimension $\dim F_{\sigma}$. 
\end{cor}

\begin{proof}
By Theorem~\ref{main} we can assume that $c =  \sum C _{\omega}\times l_{\omega}^{\sigma}$, where $\dim C_{\omega}=n-1$ and $\dim l_{\omega}^{\sigma} = \dim F_{\sigma}$. Then $C_{\omega}\subset \stackrel{\circ}{P}_{\sigma}$ for $\sigma\in \omega$ and since $\partial C_{\omega}$ has at least one facet  $C_{\omega^{\prime}}$ which belongs to $\stackrel{\circ}{P}_{\sigma}$ we obtain the chain $P_{\sigma, n-1}\times l^{\sigma}$, where $l^{\sigma}=l_{\omega}^{\sigma}$. This chain contains all summands from $c$ of the form $C_{\omega}\times l_{\omega}^{\sigma}$. Repeating the argument with left summands in $c$, we prove the statement.
\end{proof} 

\begin{rem}
We further use notation $l^{\sigma}= F_{\sigma}$ since $\dim l^{\sigma}= \dim F_{\sigma}$.
\end{rem}

\par The following result is in the line with the result of \cite{buchstaber2019toric} for $ n=5. $
\begin{cor}\label{Cycles_3n-9}
The group of cycles  in $X_n$ of dimension $3n-9$ has a basis of the from 
\begin{equation}\label{cyclebasis}
	K_{mi, n-1}\times F_{mi}+ K_{mj, n-1}\times F_{mj}+ K_{ij, n-1}\times F_{ij},
\end{equation}
where $1\leq m\leq n-2$  and $m+1\leq i<j\leq n$ and  $K_{ij}$ is an admissible polytope which does not contain just one vertex $\Lambda _{ij}$ of $\Delta _{n,2}$, while  $F_{ij}$ it the chain  of maximal dimension  in  the space of parameters $F_{ij}$ of the corresponding stratum $W_{ij}$.
\end{cor}
\begin{proof}
A cycle $c$ in $X_n$ of dimension $3n-9$ maps by the induced moment map to $\stackrel{\circ}{\Delta}_{n,2}$ since $\dim \hat{\mu}^{-1}(\partial \Delta_{n,2}) = 3n-10$.  Thus,  by Corollary \ref{admcycle} it is of the form   $c = \sum _{\sigma} P_{\sigma, n-1}\times l^{\sigma}$, where $\dim l^{\sigma}= 2n-8$. 
%		\cite{BTn2}
The only strata whose spaces of parameters are of this dimension are the strata whose admissible polytopes are $K_{ij}$. In addition,  the spaces of parameters $F_{ij}$ are all homeomorphic to $F_{n-1}$, see~\cite{buchstaber2019toric}. There are two facets $K_{ij}(i)$ and $K_{ij}(j)$  of type $\Delta _{n-1, 2}$ of $K_{ij}$  and they are given by $x_i=0$ and $x_j=0$ respectively. The  strata corresponding to these facets are the main strata for $G_{n-1,2}(i)$ and $G_{n-1, 2}(j)$ and they are in the boundary of the strata in $G_{n,2}$ whose admissible polytope is $K_{ij}$. Therefore,
\[
\partial (K_{ij}\times F_{ij}) = K_{ij}(i)\times F_{ij}(i) + K_{ij}(j)\times F_{ij}(j) = S_i + S_j,
\]
where $F_{ij}(i) \cong F_{ij}(j)\cong F_{ij}$ and $S_i$ and $S_j$ are chains of maximal dimension in $G_{n-1,2}(i)/T^{n-1}$ and $G_{n-1, 2}(j)/T^{n-1}$
Therefore, $c= \sum _{1\leq i<j\leq n} a_{ij}(K_{ij}\times F_{ij})$ is a cycle if and only if
\[
\sum_{1\leq i<j\leq n}a_{ij}(S_i+S_j) = 0,
\]
Since $S_i$, $1\leq i\leq n$ are linearly independent, this is equivalent to
\[
\sum_{1\leq j\leq n, j\neq i}a_{ij}=0, \;\; 1\leq i\leq n.
\]
The solution space is of dimension $\frac{(n-2)(n-1)}{2}$, that is the basis for the cycles is given by~\eqref{cyclebasis}.
\end{proof}  
We  formulate as well the following results, the first one is direct generalization of Corollary \ref{admcycle}, while the second one is the direct generalization of the third statement in Theorem~\ref{main} and can be proved in the analogous way using the fact that the homology groups of any $F_{\omega}$ are evenly generated.
\begin{cor}\label{Cor:Chains_with_boundary_on_boundary}
Let $c = \sum C_{\omega}\times l_{\omega}$ be a chain in $X_n$ such that $\dim C_{\omega}=n-1$ and $\partial^{0}c=0$. Then 
\[
c = \sum _{\sigma} P_{\sigma, n-1}\times l^{\sigma},
\]
where   $P_{\sigma, n-1}$ is the union of the chambers of dimension $n-1$ from  $\stackrel{\circ}{P}_{\sigma}$ of an admissible polytope   $P_{\sigma}$,  $\dim P_{\sigma}=n-1$ and  $l^{\sigma}$ is a chain in $F_{\sigma}$ of maximal dimension $\dim F_{\sigma}$.
\end{cor}

\begin{cor}\label{odd}
Let $c = \sum C_{\omega}\times l_{\omega}$ be a chain in $X_n$ such that $\dim C_{\omega}=n-1$,  $\partial^{0}c=0$ and $\dim l_{\omega}$ is odd. Then $c$ is homologous to zero.
\end{cor}

\subsection{$\Z_2$-chains over $\stackrel{\circ}{\Delta}_{n,2}$ with boundary on $\partial \Delta _{n,2}$}
We first prove few preparatory statements by considering chains for which all contributing chambers have the same dimension.

%We note that a chain of the space of parameters $ F_{\omega} $ satisfies $ l_{\omega}=\sum_{\sigma\in\omega}l_{\omega}^{\sigma}. $
\begin{lemma}\label{Lem:c_csigma}
Let  $ c=\sum_{\omega}C_{\omega}\times l_{\omega} $ be a chain in $ X_{n}, $ such that  the following is satisfied:
\begin{itemize}
	\item  all chambers $C_{\omega}$ which contribute to $c$ have the  same dimension $p\leq n-2$,
	\item  $ \hat{\mu}(c)\subset \stackrel{\circ}{\Delta}_{n,2}$,
	\item  $ \partial c$ does not contain a chamber $C_{\omega^{'}}\subset \stackrel{\circ}{\Delta}_{n,2}$ of dimension $p-1$.
\end{itemize}

Let $([\omega], \sigma)$ denotes the set of all $\tilde{\omega}$ which contributes to $c$ such that $\sigma \in \tilde{\omega}$ and $l_{\omega}^{\sigma}= l_{\tilde{\omega}}^{\sigma}$. Then $ c $ can be written in the form 
\[ c=\sum_{([\omega], \sigma)} c_{([\omega], \sigma)}, \; \text{ where  } \;  \hat{\mu}(\partial c_{([\omega], \sigma)})\subset \partial P_{\sigma}. \]
\end{lemma}    
\begin{proof}
The boundary of the chain $ c $ is given by $ \partial c=\sum_{\omega}C_{\omega}\times \partial l_{\omega}+ \bar{c}$, where all chambers which contribute to $\bar{c}$ have dimension $\leq p-1$.   The third condition implies that $\hat{\mu}(\bar{c})\subset \partial \Delta _{n,2}$.
%Since $ \partial ^{0} c\subset \partial \Delta_{n,2} $ it follows that   $\partial l_{\omega}=0,$ for all $ \omega. $

Let us fix a chamber $  C_{\omega} $ in $c$ and an admissible polytope  $ {P}_{\sigma}  $ such that $ \sigma\in\omega $ and $ C_{\omega}\in \stackrel{\circ}{P}_{\sigma} $ and  $l_{\omega}^{\sigma}\neq 0$.  There exists a  facet $ C_{\omega^{\prime}} $ of $ C_{\omega} $ such that  $ C_{\omega^{\prime}}\subset \stackrel{\circ}{P}_{\sigma} $, so $ C_{\omega^{\prime}}\times l_{\omega}^{\sigma} $  contributes to $\partial c$, so we obtain the chain  \[ c_{([\omega], \sigma)}=(\sum_{\tilde{\omega}}C_{\tilde{\omega}})\times l_{\omega}^{\sigma},  \] where $ \sigma \in  \tilde{\omega}$ and $l_{\tilde{\omega}}^{\sigma} = l_{\omega}^{\sigma}$ and it holds  $ \partial(\sum_{\tilde{\omega}}C_{\tilde{\omega}})\subset \partial \stackrel{\circ}{P_{\sigma}}.$ 
\par For another choice of $ \sigma\in\omega $ or a chamber $ \omega $ we obtain another such chain.
\end{proof}     

\begin{lemma}\label{Th:HomologChain_n-2}
Let $ c=\sum_{\omega}C_{\omega}\times l_{\omega} $ be a chain  in $ X_{n}, $ such that the following is satisfied:
\begin{itemize}
	\item  all chambers $C_{\omega}$ which contributes to $c$ have the  same  dimension $p\leq n-2$,
	\item  $ \hat{\mu}(c)\subset \stackrel{\circ}{\Delta}_{n,2}$,
	\item $ \partial c$ does not contain a a chamber $C_{\omega^{'}}\subset \stackrel{\circ}{\Delta}_{n,2}$ of dimension $p-1$.
\end{itemize}
Then $ c $ is homologous to a chain $ c_{p+1}+c^{'}$ such that all chambers in $c_{p+1}$ have dimension $p+1$ and $\hat{\mu}(c^{'}) \subset \partial \Delta_{n,2}. $

%	$ \stackrel{\circ}{c_{\sigma}} $ be a chain in $ \stackrel{\circ}{P_{\sigma}} $ such that $ \dim \stackrel{\circ}{c_{\sigma}} \leq n-2 $ and $ \partial \stackrel{\circ}{c_{\sigma}}\in\partial P_{\sigma} $.
%	 Then the chain $ c_{\sigma}=\stackrel{\circ}{c_{\sigma}}\times l^{\sigma} $ 
%	is homologous to a chain $ c^{\prime}\in \hat{\mu}^{-1}(\partial \Delta_{n,2}). $
\end{lemma}
\begin{proof}
%By  Lemma \ref{Lem:c_csigma} we can write the chain $ c=\sum_{\omega, \sigma} c_{\omega, \sigma}$ and consider  a chain $ c_{\omega, \sigma}  = %C_{\omega}^{\sigma}\times l_{\omega}^{\sigma}$ and assume that $l_{\omega}^{\sigma}$ is not a point, which gives that $\dim P_{\sigma}=n-1$.  Note %that the same holds then any $\omega$ and $\sigma \in \omega$.  

By Lemma~\ref{Lem:c_csigma} we have that $c$ is the sum of the chains $c_{([\omega], \sigma)} = C_{([\omega], \sigma)}\times l_{\omega}^{\sigma}$, where $C_{([\omega], \sigma)} = \sum C_{\tilde{\omega}}$, the sum goes over those $\tilde{\omega}$ such that $C_{\tilde{\omega}}\subset\stackrel{\circ}{P}_{\sigma}$ and $l_{\tilde{\omega}}^{\sigma}=l_{\omega}^{\sigma}$. It holds $\partial  C_{([\omega], \sigma)}\subset \partial P_{\sigma}$. For $ \dim(P_{\sigma})=n-1 $, we have that  $\partial  C_{([\omega], \sigma)}$ is a cycle of dimension $p-1\leq n-3$ on $S^{n-2}$, so there exists a chain $\tilde{C}_{([\omega], \sigma)}$ on  $ \partial P_{\sigma}$ such that $\partial \tilde{C}_{([\omega], \sigma)} = \partial  C_{([\omega], \sigma)}$. We obtain the cycle  $ \tilde{C}_{([\omega], \sigma)} +  C_{([\omega], \sigma)}$ in $P_{\sigma}$,  so there exists a chain  $\hat{C}_{([\omega], \sigma)}$ in $P_{\sigma}$ such that  $\partial \hat{C}_{([\omega], \sigma)} = \tilde{C}_{([\omega], \sigma)} +  C_{([\omega], \sigma)}$.  Consider the chain $ \hat{C}_{([\omega], \sigma)}^{0}$ containing   those chambers from $ \hat{C}_{([\omega], \sigma)}$ which belong to $\stackrel{\circ}{P}_{\sigma}$. Let $\hat{d}_{([\omega], \sigma)}= \hat{C}_{([\omega], \sigma)}^{0} \times l_{\omega}^{\sigma}$. Then 
\[
\partial \hat{d}_{([\omega], \sigma)} = c_{([\omega], \sigma)} + \tilde{C}_{([\omega], \sigma)}^{0}\times l_{\omega}^{\sigma} +  \hat{C}_{([\omega], \sigma)}^{0} \times \partial  l_{\omega}^{\sigma},
\]
where $\tilde{C}_{([\omega], \sigma)}^{0}$ contains those  chambers from $\tilde{C}_{([\omega], \sigma)}$ for which the product with $l_{\omega}^{\sigma}$ exists. Obviously, such chambers must  belong to $\partial \Delta _{n,2}$, that is the chain  $\tilde{C}_{([\omega], \sigma)}^{0}$ belongs to $\partial \Delta _{n,2}$. If $ \dim(P_{\sigma})=n-2 $ then $ l_{\omega}^{\sigma} $ is a point and $ \partial  C_{([\omega], \sigma)}\subset \partial P_{\sigma}\subset \partial \D{n} $ by Theorem \ref{Th:admissiblepolytope}, so the previous argument can be applied.
\par It follows that
\[
c_{([\omega], \sigma)} =  \partial \hat{d}_{([\omega], \sigma)}+ \tilde{C}_{([\omega], \sigma)}^{0}\times l_{\omega}^{\sigma} + \hat{C}_{([\omega], \sigma)}^{0} \times \partial  l_{\omega}^{\sigma},
\] so $c$ can be written as 
\[
c= \partial \sum  \hat{d}_{([\omega], \sigma)}   + \sum   \tilde{C}_{([\omega], \sigma)}^{0}\times l_{\omega}^{\sigma} + \sum \hat{C}_{([\omega], \sigma)}^{0} \times \partial  l_{\omega}^{\sigma}= \partial d + c^{'} + c_{p+1},
\]
where $\hat{\mu}(c^{'})\subset \partial \Delta _{n,2}$ and $\hat{\mu}(c_{p+1})\subset \stackrel{\circ}{\Delta} _{n,2}$ and all chambers in $ c_{p+1} $ have dimension $ p+1 $.
\end{proof}

By successive application of Lemma \ref{Th:HomologChain_n-2} we obtain :

\begin{prop}
Let $ c=\sum_{\omega}C_{\omega}\times l_{\omega} $ be a chain  in $ X_{n}, $ such that the following is satisfied:
\begin{itemize}
	\item  all chambers $C_{\omega}$ which contributes to $c$ have the  same  dimension $p\leq n-2$,
	\item  $ \hat{\mu}(c)\subset \stackrel{\circ}{\Delta}_{n,2}$ and  $\partial^{0} c=0$.
\end{itemize}
Then $ c $ is homologous to a chain $ c_{n-1}+c^{'}$ such that all chambers in $c_{n-1}$ have dimension $n-1$ and $\hat{\mu}(c^{'}) \subset \partial \Delta_{n,2}.$ 
\end{prop}

A chain $c$ on $X_n$ can be written as $c = c_{0} + c_{1}$, where $\hat{\mu}(c_{0}) \subset \stackrel{\circ}{\Delta _{n,2}}$ and $\hat{\mu}(c_{1})\subset \partial \Delta _{n,2}$.

\begin{cor}
Let $ c $ be a chain  in $ X_{n}, $ such that  all chambers which contributes to $c_{0}$ have the  same  dimension $p\leq n-2$  and $ \partial c$ does not contain a chamber $C_{\omega^{'}}\subset \stackrel{\circ}{\Delta}_{n,2}$ such that $\dim C_{\omega^{'}} = p-1$.  Then $ c $ is homologous to a chain $c_{n-1}+c^{'}$ such that $\hat{\mu}(c^{'}) \subset \partial \Delta_{n,2} $ and $c_{n-1}$ contains just chambers of dimension $n-1$.
\end{cor}

\begin{prop}\label{Cor:No_Chambers_leq_d-1}
		Let $c_{0} \subset \hat{\mu}^{-1}(\stackrel{\circ}{\Delta _{n,2}})$ be a chain, such that   $ \partial c_{0}$ does not contain a chamber $C_{\omega^{'}}\subset \stackrel{\circ}{\Delta}_{n,2}$ such that $\dim C_{\omega^{'}}\leq p-1$, for some $ p\leq n-2. $  Then $ c_{0} $ is homologous to a chain $c_{\geq p+1}+c^{'}$ such that $\hat{\mu}(c^{'}) \subset \partial \Delta_{n,2} $ and $c_{\geq p+1}$ contains just chambers of dimension $\geq p+1$.
	\end{prop}
	\begin{proof}
		Let $ q\leq d $ be the lowest dimension of $ C_{\omega} $ that contributes to $ c_{0}. $ We can write $ c_{0}=\sum_{q\leq i\leq n-1} c_{0}^{i} $, where all chambers from $ c_{0}^{i} $ are of the dimension $ i. $ The chain $ c_{0}^{q} $ satisfies conditions of Lemma \ref{Th:HomologChain_n-2} and it follows that $ c_{0}^{q}=\tilde{c}_{0}^{q+1}+\bar{c}_{1}+\partial d. $ Thus, 
		\[ c_{0}=\sum_{q+2\leq i\leq n-1} c_{0}^{i}+c_{0}^{q+1}+\tilde{c}_{0}^{q+1}+\bar{c}_{1}+\partial d. \]  By abuse of notation, we denote  $ \tilde{c}_{0}^{q+1}+c_{0}^{q+1} $ by $ c_{0}^{q+1}.$ It follows that $ c_{0}^{q+1} $ satisfies the conditions of Lemma \ref{Th:HomologChain_n-2}. We prove the statement repeating the argument.
\end{proof}

\begin{cor}\label{The:HomologyChain_n-2_general}
Let $ c= c_{0} + c_{1} $ be a chain  in $ X_{n}, $ such that there exists a chamber  in  $c_{0}$ which has  the  dimension $\leq n-2$ and $ \partial ^{0} c = 0$.   Then $ c $ is homologous to a chain $ c_{n-1} + c^{'}$ such that $\hat{\mu}(c^{'}) \subset \partial \Delta_{n,2} $ and all chambers in $c_{n-1}$ have dimension $n-1$.
\end{cor}

Let $Y_n= \hat{\mu}^{-1}(\partial \Delta_{n,2})$.  We now relate  the  homology  groups of $ X_{n} $ and $Y_n$ in some dimensions.
\begin{theorem}\label{Cor:Homology_less_n-3}
$ H_{k}(X_{n};\Z _2)\simeq H_{k}(Y_{n};\Z _2)$ for $ k\leq n-3$.
\end{theorem}  
\begin{proof}
A cycle  $c$ in  $X_n$ such that  $\dim c\leq n-3$  is by Corollary~\ref{The:HomologyChain_n-2_general} homologous  to a cycle  $c^{'}$ such that $\hat{\mu}(c^{'})\subset \partial \Delta_{n,2}$. In addition, if $c=\partial d$ then $\dim d\leq n-2$, so again by Corollary~\ref{The:HomologyChain_n-2_general} we have that $d$ is homologous to a cycle $d^{'}$ such that $\hat{\mu}(d^{'})\subset \partial \Delta_{n,2}$. Therefore, the homology groups in dimensions $\leq n-3$ of $X_n$ and $Y_n$ are isomorphic.
\end{proof}
\begin{rem}\label{n-2}
We demonstrate that Theorem~\ref{Cor:Homology_less_n-3}  does not hold for $ k=n-2$.  Namely,  any  cycle of dimension $ n-2 $ is homologous to a cycle  $ c^{\prime}$ in $Y_{n}. $  
%But can it obtain as boundary of $ n-1 $-dimensional chain? 
%That chain can be in $ \hat{\mu}^{-1}(\Ins{ \Delta_{n,2}}). $ By the Theorem \ref{Th:HomologChain_n-2}, it would be enough to prove that a chain $ c=\sum C_{\omega}\times l_{\omega} $, such that $ \dim(C_{\omega})=n-1 $ and $ l_{\omega} $ is a point and $\partial c\in \hat{\mu}^{-1}(\partial \Delta_{n,2}),$ is homologous to a chain on the boundary $ \hat{\mu}^{-1}(\partial \Delta_{n,2}).$ By the Lemma \ref{Lem:c_csigma} and since the $ l_{\omega} $ is defined on $ \partial P_{\sigma} $, it follows that $ F_{\sigma}=F_{n} $ and the chain $ c $ satisfies 
Consider the chain 
$c=\sum_{\forall\omega}C_{\omega}\quad \text{and}\quad  \dim C_{\omega}=n-1.
$
Then 
$
\partial c =\sum_{\forall C_{\omega}\in\partial \Delta_{n,2}}C_{\omega},  \quad  \dim(C_{\omega})=n-2.$
There is no $(n-1)$-dimensional chain  $ \tilde{c}\in Y_{n}$  such that $ \partial \tilde{c}=\partial c$.  Such chain $\tilde{c}$ should contain the summand of the  form $\tilde{c} = \sum_{\substack{\forall C_{\omega}\in\partial \Delta_{n,2} \\ \dim C_{\omega}=n-2}}C_{\omega}\times l_{\omega}$, where $\dim l_{\omega}=1$.  But for  $C_{\omega}=\stackrel{\circ}{\Delta}_{n-1,1}\in \partial \tilde{c} $, we have that $F_{\omega}$ is a point  so  an  $ 1 $-dimensional chain $ \tilde{l}_{\omega} $ does not exist.
\end{rem}
We extended the previous results for the cycles of dimensions $ k=n-1$ and $k=n$.
%We can only prove similar theorem for cycles.

%That is, we consider a cycle $ c=\sum_{\omega}C_{\omega}\times l_{\omega}, $ such that $ \dim c= n-1$ and the chain $ c $ satisfies $ \hat{\mu}(c)\subset \stackrel{\circ}{\Delta}_{n,2}$ and $ \hat{\mu}(\partial c)\subset \partial\Delta_{n,2}. $

\begin{prop}\label{Th:Cycles_n-1}
Let $ c $ be a cycle $ c=\sum_{\omega}C_{\omega}\times l_{\omega} $ on $ X_{n} $ such that $ \dim c= n-1$ or $\dim c=n$. Then the cycle $ c $ is homologous to the cycle $ c^{\prime} $ such that $ \hat{\mu}(c^{\prime})\subset \partial{\Delta}_{n,2}.$
\end{prop}
\begin{proof}
Let us write a cycle $ c $ as $ c_{0}+c_{1} $ such that $ \hat{\mu}(c_{0})\subset \stackrel{\circ}{\Delta}_{n,2} $ and $  \hat{\mu}(c_{1})\subset \partial\Delta_{n,2} $.  
Let $\dim c=n-1$. If  all chambers $ C_{\omega}$ in $c_0$ have dimension $n-1$ then  $ c_{0}=\sum_{\omega}C_{\omega}$ is the sum of all chambers of the dimension $ n-1. $ It follows that $ \partial c_{0} $ consists of all chambers of top dimension $ n-2 $ on $ \partial \Delta_{n,2}. $ Since $ \partial c_{0}=\partial c_{1} $ and $ \Delta_{n-1,1}\in \partial c_{0} $ it follows that $\Delta_{n-1,1}$ should be a chamber in  $\partial c_{1}.  $ Since the corresponding space of parameters for simplex $ \Delta_{n-1,1} $ is a point, there does not exist a $( n-1 )$-dimensional   chain  $ \tilde{c}\in \partial \Delta_{n,2}  $ such that $\Delta_{n-1,1}$ contributes to  $\partial \tilde{c}$.  Thus,  there should exist  $  C_{\omega}$ of dimension $\leq n-2$ contributing to $c_0$. Now, by Corollary~\ref{The:HomologyChain_n-2_general} we have that $c$ is homologous to a cycle $c_{n-1}+c^{'}$,  so the previous implies that $c_{n-1}=0$. 
%\begin{theorem}\label{Th:Cycles_n}
%Let $ c $ be a cycle $ c=\sum_{\omega}C_{\omega}\times l_{\omega} $ on $ X_{n} $ such that $ \dim c= n$. Then the cycle $ c $ is homologous to a  cycle $ %c^{\prime} $ such that $ \hat{\mu}(c^{\prime})\subset \partial{\Delta}_{n,2}.$
%\end{theorem}
%\begin{proof}
%Let us write the cycle $ c $ as $ c_{0}+c_{1} $ such that $ \hat{\mu}(c_{0})\subset \stackrel{\circ}{\Delta}_{n,2} $ and $  \hat{\mu}(c_{1})\subset %\partial\Delta_{n,2} $.
% Since  the chain $ c_{0} $ cannot be a cycle by the Corollary \ref{admcycle}, it follows $ \partial c_{0}\in \hat{\mu}(\partial \Delta_{n,2}). $ 
%By the Theorem \ref{Th:HomologChain_n-2} it is sufficient to prove for the case when 
For $\dim c=n$, by Corollary~\ref{The:HomologyChain_n-2_general} we can assume that  $c_{0}= \sum_{\omega}C_{\omega}\times l_{\omega} $ such that $ \dim C_{\omega}=n-1 $  for all $\omega$. Then $\dim l_{\omega}=1$, so Corollary~\ref{odd} gives the $c_0$ is homologous to zero, which concludes the proof.  
\end{proof}

\begin{theorem}
 $H_{n-1}(X_n;\Z _2) \cong H_{n-1}(Y_n;\Z _2)$ and $H_{n-2}(Y_n; \Z _2)\cong \Z _2 \oplus H_{n-2}(X_n; \Z _2)$.
\end{theorem}
\begin{proof}
Let $c$ be a cycle in $Y_n$ of dimension $n-1$. If $\hat{c}$ is a chain in $X_n$ such that $\partial \hat{c} = c$ then $\hat{c}$ is by Corollary~\ref{The:HomologyChain_n-2_general} homologous to a chain $\hat{c}_{n-1}+\hat{c}^{'}$. If $\hat{c}_{n-1}\neq 0$ then   $\hat{c}_{n-1}=\sum C_{\omega}\times l_{\omega}$, where $\dim l_{\omega}=1$ and $\partial  ^{0} \hat{c}_{n-1}=0$. It follows from Corollary~\ref{odd} that $\hat{c}_{n-1}$ is homologous to zero. Thus,  $\hat{c}$ is homologous to a chain  from $Y_n$.
Let $\dim c=n-2$ and assume that $c$ contains a chamber of dimension $n-2$. Then $c$ contains the chain $d = \sum_{\forall C_{\omega}}C_{\omega}$, $C_{\omega}\subset \partial \Delta _{n,2}$ and $\dim C_{\omega} = n-2$. Since $d=\partial \hat{d}$, $\hat{d} =\sum_{\forall C_{\omega}}C_{\omega}$, $\dim C_{\omega}=n-1$ it follows that $c$ is homologous in $X_n$  to a cycle  $c^{'}$  from $\hat{\mu}^{-1}(\cup_{1\leq m\leq n}\partial O_m)$ not containing chambers of dimension $n-2$.  Using Remark~\ref{n-2} the second statement follows.
\end{proof}

	\begin{prop}\label{Th:Cycles_n-1_odd}
		Let $ c $ be a chain  on $ X_{n} $ such that $ \dim c= n-1+2k+1$  where $ \partial^0 c=0 $ and $ 1\leq k< n-3. $ Then the chain $ c $ is homologous to a chain  $ c^{\prime} $ such that $ \hat{\mu}(c^{\prime})\subset \partial{\Delta}_{n,2}.$
	\end{prop}
	\begin{proof}  Let 
		$c=  c_{0}+c_{1} $, where $ \hat{\mu}(c_{0})\subset \stackrel{\circ}{\Delta}_{n,2} $ and $ \hat{\mu}(c_{1})\subset \partial\Delta_{n,2} $. The chain $ c_{0} $ cannot be a cycle, since by Corollary \ref{admcycle} all chambers $ C_{\omega} $ must be of the dimension $ n-1 $ and the dimension of $ l_{\omega} $ has to be even. Since $ \partial^0 c_{0}=0, $ by Corollary \ref{The:HomologyChain_n-2_general} we can write $ c_{0}=c_{n-1}+c^{\prime} $  where  $ c^{\prime}\in \hat{\mu}^{-1}(\partial \D{n}) $ and $ c_{n-1}=\sum C_{\omega}\times l_{\omega},\, \dim(C_{\omega})=n-1.$ It follows that $ \partial^{0} c_{n-1}=0 $ and $ \dim(l_{\omega})=2k+1 $ for any $ \omega $. By Corollary \ref{odd} we have that the chain $ c_{n-1} $ is homologous to zero. 
	\end{proof}
	\begin{prop}\label{Th:Cycles_n-1_even}
		Let $ c $ be a cycle on $ X_{n} $ of the dimension $ \dim c= n-1+2k,\, 1\leq k\leq n-3$. Then  $ c $ is homologous to a cycle  $ c^{\prime} $ of the form 
		\[ c^{\prime} =c_{0}^{\prime}+c_{1}^{\prime},\text{ where } c_{0}^{\prime} \text{ and } c_{1}^{\prime} \text{ are cycles}  \] such that $ c_{0}^{\prime}\in \hat{\mu}^{-1}(\Ins{\D{n}})$ and $ c_{1}^{\prime}\in \hat{\mu}^{-1}(\partial \D{n}) $. 
		%		Additionally, the cycle $ c_{0} $ is of the form:
		%		\[ \sum _{\sigma} P_{\sigma, n-1}\times l_{\sigma}  \] 
		%		where   $P_{\sigma, n-1}$ is the union of the chambers of dimension $n-1$ from  $\stackrel{\circ}{P}_{\sigma}$ of an admissible polytope   $P_{\sigma}$,  $\dim P_{\sigma}=n-1$ and  $l_{\sigma}$ is a chain in $F_{\sigma}$ of maximal dimension $\dim F_{\sigma}$.		
		%	for $ 0\leq m< k. $}
\end{prop}

For the proof we will make us of the following result.

\begin{lemma}\label{paramspaceboundary}
	Let $C_{\omega}\subset \stackrel{\circ}{\Delta}_{n,2}$ and $C_{\bar{\omega}}$ a facet of $ C_{\omega} $ such that $C_{\bar{\omega}}\subset \Delta _{n-1,2}(i)$ for some $1\leq i\leq n$. Then $F_{\omega}\setminus F$ is homeomorphic to $F_{\bar{\omega}}$.
\end{lemma}
\begin{proof}
We first note that for any $P_{\bar{\sigma}}\subset \partial \Delta _{n,2}$, $\dim P_{\bar{\sigma}}=n-2$,  there exists $P_{\sigma}\neq \Delta _{n,2}$ such that $P_{\bar{\sigma}}$ is a facet of $P_{\sigma}$. Now, the chamber $ C_{\omega} $ must belong to some halfspace $ x_{i}+x_{j}<1,\,1\leq j\leq n,\,j\neq i. $ Otherwise, $C_{\omega}$ would belong to all halfspaces $x_i+x_j\geq1$ for $1\leq j\leq n$, $j\neq i$. If $C_{\omega}\subset \stackrel{\circ}{P}_{\sigma}\neq \Delta _{n,2}$ then Plücker coordinates $P^{kl}=0$ on the stratum $W_{\sigma}$ for all $1\leq k<l\leq n$, $k,l\neq i$, so $P_{\sigma}\subset \Delta_{n-1,1}(i)\subset \partial \Delta _{n,2}$, which is not the case.
\par Let $C_{\omega}$ belongs to a halfspace $x_i+x_j<1$. If $C_{\omega}\subset  \stackrel{\circ}{P}_{\sigma}\neq \stackrel{\circ}{\Delta}_{n,2}$,  then $P^{ij}=0$ on the stratum $W_{\sigma}$. In addition $W_{\sigma}$ belongs to a chart $M_{jl}$, otherwise $P^{jl}=0$ for all $j\neq l$, so $P_{\sigma}\subset \partial \Delta _{n,2}$. Let us take $j=1, i=3, l=2$. In the chart $M_{12}$ the stratum $W_{\sigma}$ is represented by the matrix
\begin{equation}
	\left(\begin{array}{cccccc}
		1 & 0 & b_3 & a_4& \ldots & a_n\\
		0 & 1 & 0& b_4 & \ldots & b_n,
	\end{array}
	\right)
\end{equation}
which implies $F_{\sigma} \cong F_{\bar{\sigma}}$, where $P_{\bar{\sigma}}\subset \partial \Delta _{n-1, 2}(i)$ is a facet of $P_{\sigma}$. This proves the statement.
\end{proof}

\begin{proof}{(Proposition \ref{Th:Cycles_n-1_even})}  Let  $c=c_{0}+c_{1}+c_{2}$, where  $ \hat{\mu}(c_{0})\subset \stackrel{\circ}{\Delta}_{n,2} $, $  \hat{\mu}(c_{1})\subset\bigcup_{1\leq m\leq n}\Ins{\Delta}_{n-1,2}(m)$ and $  \hat{\mu}(c_{2})\subset\bigcup_{1\leq m\leq n}\partial \Delta_{n-1,2}(m)$. By Corollary \ref{The:HomologyChain_n-2_general} the chain $ c_{0} $ is homologous to a chain $ c_{n-1}+c^{\prime}$, that is we can write $c = c_{0}^{n-1}+c_{1}+c_{2}$. 
\par Assume that $c_{0}^{n-1} = \sum C_{\omega}\times l_{\omega}$ is not a cycle. It implies that $\partial  c_{0}^{n-1} = \sum C_{\bar{\omega}}\times \eta ^{\#}_{\omega, \bar{\omega}}(l_{\omega}) \neq 0$, so $\dim l_{\omega} \leq 2n-8$.   
\par  By Proposition \ref{Cor:No_Chambers_leq_d-1}, the chain $ c_{1} $ contains only chambers of dimension $ n-2, $ since it can be written as \begin{equation}\label{Eq:Cycles_n-1_2k_Top_dim}
	 c_{1}=\sum_{1\leq m\leq n}c_{1}^{m}, 
\end{equation}where  $c_{1}^{m}\subset \hat{\mu}^{-1}(\Ins{\Delta}_{n-1,2}(m)) $ and  $ \partial c_{1}^{m} $ does not contain chambers of dimension $ \leq n-3 $.
\par Thus, $c_{1}= c_{1}^{n-2} = \sum C_{\bar{\omega}}\times l_{\bar{\omega}}$. Obviously, $ \partial c_{0}^{n-1}\subset \partial c_{1}^{n-2}$. If there would exist $C_{\bar{\omega}}\times l_{\bar{\omega}}$ in $c_{1}^{n-2}$ such that  $C_{\bar{\omega}}\times \partial^{\bar{ \omega}}  l_{\bar{\omega}}$ does not belong to $\partial c_{0}^{n-1}$, we would have that $\partial^{\bar{ \omega}} l_{\bar{\omega}} = 0$. Since $l_{\bar{\omega}}$ is of odd degree we have  $ l_{\bar{\omega}} = \partial^{\bar{ \omega}} \tilde{l}_{\bar{\omega}}$. Then $\partial (C_{\bar{\omega}} \times \tilde{l}_{\bar{\omega}}) = C_{\bar{\omega}}\times l_{\bar{\omega}} +c_{\bar{\omega}, n-3}$, where the second chain contains only the chambers of dimension $n-3$. Again using Proposition \ref{Cor:No_Chambers_leq_d-1}, $c_{1}^{n-2}$ is homologous to a chain for which it does not contain such chambers and it holds $\eta ^{\#}_{\omega, \bar{\omega}}(l_{\omega}) = \partial ^{\bar{\omega}}l_{\bar{\omega}}$.
\par By Lemma~\ref{paramspaceboundary} we deduce that  $(\eta ^{\#}_{\omega, \bar{\omega}})^{-1}(l_{\bar{\omega}}) = l_{\bar{\omega}}$ and $l_{\omega} = \partial ^{\bar{\omega}} l_{\bar{\omega}} = \partial ^{\omega}l_{\bar{\omega}} $. 
\par For  the chain $\bar{d} = \sum C_{\omega} \times l_{\bar{\omega}}$  we have $\partial \bar{d} = \sum C_{\omega}\times \partial ^{\omega}l_{\bar{\omega}} + \sum C_{\omega^{'}}\times l_{\bar{ \omega}^{'}} + \sum C_{\bar{\omega}}\times l_{\bar{\omega}} = c_{0}^{n-1} + c_{0}^{n-2} +c_{1}^{n-2}$. We can write $ c=c_{0}^{n-2}+ c^{\leq n-3}_{2}$, where all chambers in the second summand are belong to  $ \bigcup_{1\leq m\leq n}\partial\Delta_{n-1,2}(m) $ and are of dimension $ \leq n-3 .$ 
\par We write  $ c^{\leq n-3}_{2}=c_{2}+c_{3}+c_{4} $, where $ c_{2} $ is the chain from the union of all $ \Ins{\Delta}_{n-3,2}(ij) $, the chain $ c_{3} $ is from the union of all  $ \Ins{\Delta}_{n-3,2}(ijk)  $ and $ c_{4} $ is the rest. The claims that $ c_{2} $ contains only chambers of dimension $ n-3 $ and $ c_{3} $ contains only chambers of $ n-4 $ follow in the same way as for $ c_{1}. $ We denote them by $ c_{2}^{n-3}=\sum C_{\omega}\times l_{\omega} $ and $ c_{3}^{n-4}=\sum C_{\hat{\omega}}\times l_{\hat{\omega}} $ resp. Then :
\[ \partial c_{2}^{n-3}=\partial c_{0}^{n-2}+\sum C_{\hat{\omega}}\times l_{\hat{\omega}},\quad \partial c_{3}^{n-4}=\sum C_{\hat{\omega}}\times\partial l_{\hat{\omega}}+ \tilde{c}^{n-5}_{4}, \]
 where $ \tilde{c}^{n-5}_{4} $ is a chain of dimension $ n-5. $ It follows that $ \sum C_{\hat{\omega}}\times l_{\hat{\omega}}=\sum C_{\hat{\omega}}\times\partial l_{\hat{\omega}} $, so $ C_{\bar{\omega}}=C_{\hat{\omega}} $ and $ l_{\omega}=\partial \hat{l}_{\bar{ \omega}}. $ We consider the chain:
\[ \tilde{c}_{2}^{n-3}=\sum C_{\omega}\times \hat{l}_{\bar{ \omega}}.\] It holds:
\begin{align*}
	\partial \tilde{c}_{2}^{n-3}&=\sum C_{\omega}\times\partial\hat{l}_{\bar{ \omega}}+\sum C_{\bar{ \omega}}\times \hat{l}_{\bar{ \omega}}+\sum C^{\prime}_{\omega}\times\hat{l}^{\prime}_{\bar{ \omega}}=\\
	&= c_{2}^{n-3}+c_{3}^{n-4}+\tilde{c}_{2}^{n-4},
\end{align*} where $ \tilde{c}_{2}^{n-4} $ is a chain whose chambers are of dimension $ n-4 $ and from the union of $ \Ins{\Delta}_{n-3,2}(ij). $ We obtain that the cycle $ c=c_{0}^{n-2}+ c_{2}^{n-3}+c_{3}^{n-4}+c_{4} $ is homologous to:
\begin{equation}\label{Eq:Cycles_n-1_2k_Iterat}
	c^{\prime}=c_{0}^{n-2}+\tilde{c}_{2}^{n-4}+c^{\leq n-5},
\end{equation}
where a chain $ c^{\leq n-5}\subset \hat{\mu}^{-1}(\partial \D{n}) $ contains chambers of dimension $ \leq n-5 $. We conclude the proof if we set $ c_{0}^{\prime}=c_{0}^{n-2} $.
\end{proof}

	\begin{cor}\label{Cor:Spaces_Y_n}
		Let $ c $ be a cycle of dimension $ q $ in  $ Y_{n}\subset X_{n}.$
		\begin{itemize}
			\item If $ q=n-2+2k+1,\,k\geq0 $ or $ q\leq n-2 $, then $ c $ is homologous to a cycle from $ \hat{\mu}^{-1}(\bigcup_{1\leq m\leq n}\partial\Delta_{n-1,2}(m)). $
			\item If $q=n-2+2k,\,k\geq1 $ then $ c $ is homologous to a cycle $ c^{\prime} $ of the following form:
			\[ 	c^{\prime}=\sum_{1\leq m\leq n} c^{m}+c_{2}, \,   \text{ for }   c^{m}\in \hat{\mu}^{-1}(\Ins{\Delta}_{n-1,2}(m)) \text{ and } c_{2}\in \hat{\mu}^{-1}(\bigcup_{1\leq m\leq n}\partial\Delta_{n-1,2}(m)), \]
			where:
			\[ c^{m}=\sum_{\sigma_{m}}\Ins{P}_{\sigma_{m},n-2}\times l^{\sigma_{m}},\text{ where } \dim (l_{\sigma_{m}}^{n-2})=\dim ( F^{m}_{\sigma})\text{ and } \Ins{P}_{\sigma_{m},n-2}\subset \Ins{\Delta}_{n-1,2}(m).\]
		\end{itemize}
	\end{cor}
	\begin{proof} Assume first that $ q\neq n-2. $
		Then \[ c=\sum_{1\leq m\leq n} c^{m}+c_{2},\,\text{ where } c^{m}\subset \hat{\mu}^{-1}(\Ins{\Delta}_{n-1,2}(m))\text{ and }c_{2}\subset\hat{\mu}^{-1}(\bigcup_{1\leq m\leq n}\partial \Delta_{n-1,2}(m)) .\] By Corollary \ref{The:HomologyChain_n-2_general} we can assume that $ c^{m} $ contains only chambers of the dimension $ n-2. $
		\par Obviously for $  q\leq n-3 $ the statement follows. Let $q=n-2+2k+1,\,k\geq0 $, so the dimension of $ l_{\omega} $  is odd. By Corollary \ref{odd} a chain $ c^{m} $ is homologous to a chain from $ \hat{\mu}^{-1}(\partial  \Delta_{n-1,2}(m)).$
		\par Let $ q=n-2+2k,\,k\geq1.$ If $  c^{m} $ is cycle then by Theorem \ref{Th:Cycles_n-1_even}, it is of the form:
		\[ \sum_{\sigma_{m}}\Ins{P}_{\sigma_{m},n-2}\times l_{\sigma_{m}}^{n-2} \] with the desired properties. If  $  c^{m} $ is not a cycle, then by Corollary \ref{Cor:Chains_with_boundary_on_boundary} the statement follows.
		Finally, we consider the case $ q=n-2. $ We can write $ c $ as \[ c=\sum_{1\leq m\leq n} c^{m}+c_{1}+c_{2},\,\text{ where } c_{1}\in\hat{\mu}^{-1}(\bigcup_{1\leq m\leq n} \Delta_{n-1,1}(m)) \] and $ c_{2} $ and $ c^{m} $ are as in the first case. If some $ c^{m}\neq0 $, it follows that \[ c^{m}=\sum_{\forall C_{\omega}}C_{\omega} ,  \quad  \dim(C_{\omega})=n-2.\]
		 Since a simplex $ \Delta_{n-2,1}(mr)\subset \partial c^{m} ,\, 1\leq r,m\leq n $ and $ r\neq m $ cannot be canceled by some boundary of $ (n-2) $-dimensional chain from $ c_{2} $, it follows that the chain $ \tilde{c}_{1}=\sum_{\forall s\neq m} C_{s} $, where $ C_{s} $ is a $ (n-2) $-dimensional chamber of the simplex $ \Delta_{n-1,1}(s) $  satisfies $ \tilde{c_{1}}\subset c_{1}. $ Similarly, any boundary $ C_{sl},\,l\neq m $ of $ C_{s} $ can only be canceled by the boundary of chain 
		\[ c^{l}=\sum_{\forall C_{\omega}}C_{\omega} ,  \quad  \dim(C_{\omega})=n-2.\]  We get the cycle \[ d=\sum_{\forall C_{\omega}\in\partial \Delta_{n,2}}C_{\omega},  \quad  \dim(C_{\omega})=n-2 \] in $ c. $ But the cycle $ d $ is homologous to zero on $ X_{n}. $
\end{proof}

\section{$\Z _{2}$-homology groups of $X_{5}$}\label{Sec:HomX5}
%\par In the Section \ref{Sec:ModuliSpace} is given the definition of moduli space of curves of genus $ 0 $ with $ n $ marked distinct points, denoted by $ \mathcal{M}_{0,n} $. 
\par The integral homology groups of $X_{5}$ are calculated both in \cite{buchstaber2019toric} using the stratification of $ X_{5} $ given by \eqref{Eq:stratification} and in \cite{suss2020toric} using Geometric invariant theory. We present here the calculations for  $\Z _{2}$-homology of $X_{5}$ in different way, applying the model from Theorem \ref{Th:MainModel}, the existence of diffeomorphisam between $ \mathcal{M}_{0,5} $ and the universal space of parameters $ \mathcal{F}_{5} $ and relaying on the homology of $ \mathcal{M}_{0,5} $ from \cite{keel1992intersection}.
\par By Lemma \ref{lemma_3n-7} and Corollary \ref{lemma_3n-8},  $\Z _{2} $-homology groups in dimensions 8 and 7 are given by  
\[ H_{8}(X_{5};\Z _{2})\simeq \Z_{2},\quad H_{7}(X_{5};\Z _{2})\simeq 0. \] 

Let $ c=\sum C_{\omega}\times l_{\omega} $ be a cycle on $ X_{5} $ of the dimension $k,\, 0\leq k\leq 4 $. By Proposition~\ref{Th:Cycles_n-1} and Theorem \ref{Cor:Homology_less_n-3} we have that $ c $ is homologous to a cycle from $ Y_{5}. $ Then Corollary \ref{Cor:Spaces_Y_n} gives that $ c $ is homologous to a cycle from $ \hat{\mu}^{-1}(\bigcup_{1\leq m\leq 5}\partial \Delta_{4,2}(m)).$  For $ k=3 $ or $ k=4 $ it holds $ \dim(l_{\omega})\geq 1$ and since the space of parameters of $ \partial \Delta_{4,2}(m) $ is a point, it follows
\[ H_{k}(X_{5};\Z _{2})\simeq 0 \text{ for }k=3,4.\] 
For $ k\leq 2 $, we have $ c=\sum C_{\omega}.$ But $ \partial c\subset\partial \D{5}\simeq S^{3} $, so 

\[ H_{k}(X_{5};\Z _{2})\simeq 0 \quad \text{ for }1\leq k\leq2,\quad  H_{0}(X_{5};\Z _{2})\simeq \Z_{2}. \] 
\par We recall the description of an admissible polytopes, the spaces of parameters and the virtual spaces of parameters for $ X_{5} $ from \cite{buchstaber2019toric}. An admissible polytopes for $ X_{5}$ are: 

%where the corresponding virtual space of parameters $ \tilde{F}_{\sigma} $ is not a point. in the sense of Proposition \ref{cellspaceparam} of the Section \ref{Sec:Definitions}, for our calculation of the homology group $ H_{k}(X_{5}) $, where $ k\in \{5,\dots,8\} .$ An admissible polytopes of $ X_{5} $ are
\begin{itemize}
	\item 4-dimensional
	\begin{itemize}
			\item $ \Delta_{5,2} $
		\item $ K_{ij}\text{ or }K_{ij}(9): \,\{\textbf{x}\in \D{5}\,|x_{i}+x_{j}\leq 1 \} $
		\item $ K_{rpq} $ or $ K_{ij}(7):\, \,\{\textbf{x}\in \D{5}\,|x_{i}+x_{j}\geq 1 \} =\{\textbf{x}\in \D{5}\,| x_{r}+x_{p}+x_{q}\leq 1\},$ where $ \overline{ij}=\{1,\dots,5\}\backslash rpq $
		\item $ K_{ij,pq}:\,=\{\textbf{x}\in \D{5}\,| x_{i}+x_{j}\leq 1,\,x_{p}+x_{q}\leq 1 ,\, \{i,j\}\cap \{p,q\}=\emptyset\} $
	\end{itemize}
\item 3-dimensional
	\begin{itemize}
	\item $ O_{i}:\, \{\textbf{x}\in \D{5}\,|x_{i}=0 \},\,  P_{ij}:\, \{\textbf{x}\in \D{5}\,|x_{i}+x_{j}= 1 \},\,T_{i}:\,\{\textbf{x}\in \D{5}\,|x_{i}=1 \}$ 
	
\end{itemize}
\end{itemize}
According to \eqref{Def:mainspaceG} in Section \ref{Sec:Definitions}.
\begin{cor}
	 The space of parameters  $ F_{5} $ of the main stratum $ W $ is homeomorphic to \[ \{((c_{12}:c_{12}^{\prime}),(c_{13}:c_{13}^{\prime}),(c_{23}:c_{23}^{\prime}))\in \CPA\times\CPA\times\CPA|\,c_{12}^{\prime}c_{13}c_{23}^{\prime}=c_{12}c_{13}^{\prime}c_{23} \} \]
\end{cor}

% Dimension of the main space of parameters is $ \dim(F)=4 $.
%\begin{cor}\label{Corr:SpaceofparametersX5}
%	The spaces of parameters for a stratum corresponding to an admissible polytope $ K_{ij} $ and $ O_{i} $ are homeomorphic to the $ \CPA$. Other spaces of parameters $ F_{\sigma} $ different from the main space of parameters $ F $ are homeomorphic to a point.
%\end{cor}
% The space of parameters $ F_{23} $ which corresponds to the admissible polytope $ K_{23} $ is homeomorphic to \[ \{((0:1),(0:1),(c_{23}:c_{23}^{\prime}))|\,c_{23},c_{23}^{\prime}\neq0,\,c_{23}\neq c_{23}^{\prime} \} \]
%\par We recall that the main stratum is an everywhere dense set in $ \Grs{5} $ and any other stratum is the boundary of the main stratum. The space $ \overline{F} $ is given by \[ \overline{F}=\{ ((c_{1}:c_{1}^{\prime}),(c_{2}:c_{2}^{\prime}),(c_{3}:c_{3}^{\prime}))|\,c_{1}^{\prime}c_{2}c_{3}^{\prime}=c_{1}c_{2}^{\prime}c_{3}\}. \]  

%In \cite[Proposition 6.1]{buchstaber2019toric} are described all $ \overline{F}_{\sigma} $ in the chart $ M_{12}. $
Theorem \ref{Th:UniversalSpaceG} for $n=5  $ gives, see also \cite{buchstaber2019toric}
\begin{cor}
	 The universal space of parameters $ \mathcal{F}_{5} $ for $ \Grs{5} $ is defined as the blowup of $ \overline{F} $ at the point $ ((1:1),(1:1),(1:1)). $
\end{cor} \par	The virtual spaces of parameters are collected, see also Chapter 11 in \cite{buchstaber2019toric}, as follows:
\begin{cor}\label{Corr:VirtualspaceX5}
	For admissible polytopes $ P_{\sigma} $ mentioned above the spaces of parameters $ F_{\sigma} $ and virtual spaces of parameters $ \tilde{F}_{\sigma} $ are:   
	\begin{enumerate}
		\item  $ P_{\sigma}=\Delta_{5,2}\Rightarrow F_{\sigma}\simeq\tilde{F}_{\sigma}\simeq F; \quad P_{\sigma}=K_{ij} \Rightarrow F_{\sigma}\simeq\tilde{F}_{\sigma}\simeq \C P_{A}^{1}, $		 
		\item  $ P_{\sigma}=O_{i} \Rightarrow F_{\sigma}\simeq \C P_{A}^{1} $ and $ \tilde{F}_{\sigma}\simeq(\C P_{A}^{1}\times \C P^{1}\times \C P^{1})\cap \mathcal{F}_{5};$ 
		\item $ P_{\sigma}=K_{ij,pq}\Rightarrow F_{\sigma}=\tilde{F}_{\sigma} $ is a point;     $ P_{\sigma}=K_{ijk}\Rightarrow F_{\sigma}$ is a point and $\tilde{F}_{\sigma}\simeq \C P^{1}; $ 
		\item for others $ P_{\sigma}\Rightarrow F_{\sigma}$ is a point and $ \tilde{F}_{\sigma} $ is not a point.
		\end{enumerate}
\end{cor}
Following notation above of $ P_{\sigma} $ by $ K_{J} $ we further denote $ F_{\sigma} $ and $ \tilde{F}_{\sigma} $ by $ F_{J} $ and $ \tilde{F}_{J} $ as well.
\subsection{ $ H_{6}(X_{5};\Z_{2}) $ and $ H_{5}(X_{5};\Z_{2}) $}
\par
We want to write relations \eqref{relacije} in the terms of the virtual spaces of parameters for $ X_{5}. $   

\begin{lemma}
	The homology group $ H_{2}(\mathcal{F}_{5}) $ is generated by the classes determined by the virtual spaces of parameters $ \tilde{F}_{\overline{ij}} $ and the following relations hold:
		\begin{equation}\label{Th:UrelationsX5}
		\tilde{F}_{\overline{ij}}+\tilde{F}_{\overline{pq}}=
		\tilde{F}_{\overline{ip}}+\tilde{F}_{\overline{jq}}=
		\tilde{F}_{\overline{iq}}+\tilde{F}_{\overline{jp}}	  
	\end{equation}
for every distinct $ \{i,j,p,q\}\subset \{1,\dots,5\} $, where $ \overline{ij}=\{1,\dots,5\}\backslash\{i,j\} $
\end{lemma}
\begin{proof}
	 We demonstrate the proof for the particular indices $ \{1,3,4,5\}.$ Let us fix the chart $ M_{12} $ and use the map $ \tilde{f}:\mathcal{M}_{0,5}\to \mathcal{F}_{5} $ from Example \ref{ex:map05}. We need to find the image by $ \tilde{f} $ of a divisor $ D_{I} $ from \eqref{relacije}. For example, the element $ D_{12} $ maps to $ \{((c:c^{\prime}),(c:c^{\prime}),(1:1))\}.$  We need to find a stratum $ W_{\sigma} $ for which this is virtual space of parameters in $ M_{12}. $  By \eqref{Def:mainstratumG} such stratum $ W_{\sigma} $ satisfies 
	$ 	c^{\prime}a_{1}b_{2}=cb_{1}a_{2},c^{\prime}a_{1}b_{3}=cb_{1}a_{3},a_{2}b_{3}=b_{2}a_{3}, $ where $ (c:c^{\prime})\in \C P^{1} $. It follows that $ P^{45}=(1:1)$ for the points from $ W_{\sigma}. $  The virtual space of parameters $ \tilde{F}_{123} $ is, in the chart $ M_{14} $, written as $ (1:0)\times (1:0)\times \C P^{1} $. Using the homeomorphism $ \tilde{f}^{-1}_{12,14} $, defined in Theorem \ref{Th:UniversalSpaceG}, we obtain the record $\tilde{F}_{123}  $ in the chart $ M_{12}$ to be $ \{((c:c^{\prime}),(c:c^{\prime}),(1:1))\}.$ This proves that  $ \tilde{f}(D_{12}) $ maps to $ \tilde{F}_{123}. $ 
\par The divisor $ D_{124} $ maps to the exceptional set $ \{(((1:1),(1:1),(1:1)),(c:c^{\prime}))\}.$ We need to find the corresponding stratum $ W_{\sigma} $ that is defined by $ P^{34}=P^{35}=P^{45}=0 $. If we write down the stratum whose admissible polytope is $ K_{12}(7)=K_{345} $, in the chart $ M_{13} $ we obtain that its coordinates are $ a_{1},a_{2}=a_{3}=0,b_{1},b_{2},b_{3}. $ Therefore, we have that $ \tilde{F}_{345}\simeq (1:0)\times(1:0)\times \C P^{1}\simeq \C P^{1}$  in the chart $ M_{13}. $ Using the homeomorphism $ \tilde{f}_{12,13}:\mathcal{F}_{12}\to \mathcal{F}_{13} $, defined in Theorem \ref{Th:UniversalSpaceG}, we deduce that $ D_{124} $ maps to $ \tilde{F}_{345}\simeq  \C P^{1}$.  Similarly, we proceed in other cases and obtain:
\begin{align}\label{Th:UconnectX5}
	\begin{split}
			D_{12}\to  \{((c:c^{\prime}),(c:c^{\prime}),(1:1))\}\simeq\tilde{F}_{\overline{45}}\quad D_{23}\to  \{((c:c^{\prime}),(0:1),(0:1))\}\simeq\tilde{F}_{\overline{15}}\\
		D_{13}\to  \{((0:1),(c:c^{\prime}),(1:0))\}\simeq\tilde{F}_{\overline{14}}\quad D_{24}\to  \{((c:c^{\prime}),(1:1),(c^{\prime}:c),)\}\simeq\tilde{F}_{\overline{35}}\\
		D_{14}\to  \{((1:1),(c:c^{\prime}),(c:c^{\prime}))\}\simeq\tilde{F}_{\overline{34}}\quad D_{25}\to  \{((c:c^{\prime}),(1:0),(1:0))\}\simeq\tilde{F}_{\overline{25}}\\
		D_{15}\to  \{((1:0),(c:c^{\prime}),(0:1))\}\simeq\tilde{F}_{\overline{24}}\quad D_{125}\to  \{((1:0),(1:0),(c:c^{\prime}))\}\simeq\tilde{F}_{\overline{13}}\\
		D_{123}\to  \{((0:1),(0:1),(c:c^{\prime}))\}\simeq\tilde{F}_{\overline{23}}\quad
		D_{124}\to  \{((1:1),(1:1),(1:1)),(c:c^{\prime})\}\simeq\tilde{F}_{\overline{12}}
	\end{split}
\end{align} where $ (c:c^{\prime})\in\C P^{1}. $ Finally, using relations \eqref{relacije} we get that in $ H_{2}(\mathcal{F}_{5}) $ for indices $ \{1,3,4,5\} $ it holds: \[ \tilde{F}_{\overline{45}}+\tilde{F}_{\overline{13}}=\tilde{F}_{\overline{14}}+\tilde{F}_{\overline{35}}=\tilde{F}_{\overline{15}}+\tilde{F}_{\overline{34}}. \]	
\end{proof}
\par We first consider the chamber $ C_{\omega_{0}} $ defined by the intersections of all admissible polytopes $ \Ins{K}_{ij}. $
%, that is the region of the type $ \textbf{Reg6} $ from the subsection \ref{regionsG52}.
 The union \[ \mathcal{F}_{\omega_{0}}= \bigcup_{C_{\omega_{0}}\in \overset{\circ}{P_{\sigma}} }\tilde{F}_{\sigma} \] represents a subdivision of the space $ \mathcal{F}_{5} $. We can express this union as: 
\[ \mathcal{F}_{\omega_{0}}= \bigcup_{ij}\tilde{F}_{ij}\cup\bigcup_{ij,pq} \tilde{F}_{ij,pq}, \] where the notation $ \tilde{F}_{ij,pq} $ here and further on assumes that $ ij\cap pq=\emptyset. $
\par Note that every virtual space of parameters $ \tilde{F}_{\overline{ij}}$ is the disjoint union of $ \CPA $ and three points:
\begin{equation}\label{X5Particija}
	\tilde{F}_{\overline{ij}}=\tilde{F}_{ij}\cup\bigcup_{pq }\tilde{F}_{ij,pq}\quad \text{ where }  \overline{ij}=\{1,\dots,5\}\backslash\{i,j\} 
\end{equation} 

%\begin{cor}
%		In the homology group $ H_{2}(\mathcal{F}) $ there holds the relations of the virtual spaces of parameters for every $ \{i,j,p,q\}\subset \{1,\dots,5\} $:
%	\begin{align}\label{Corr:UrelationsX5}
%		\tilde{F}_{ij}+\tilde{F}_{pq}&= \nonumber\\
%		\tilde{F}_{ip}+\tilde{F}_{jq}&=\\
%		\tilde{F}_{iq}+\tilde{F}_{jp}	\nonumber	  
%	\end{align} and in the homology group $  H_{0}(\mathcal{F}) $ holds :
%	\begin{align}
%	\sum_{sm\neq ij}\tilde{F}_{ij,sm}+\sum_{sm\neq pq}\tilde{F}_{pq,sm}&= \nonumber\\
%	\sum_{sm\neq ip}\tilde{F}_{ip,sm}+\sum_{sm\neq jq}\tilde{F}_{jq,sm}&=\\
%	\sum_{sm\neq iq}\tilde{F}_{iq,sm}+\sum_{sm\neq jp}\tilde{F}_{jp,sm}	\nonumber	  
%\end{align}
%	 for every $ \{i,j,p,q\}\subset \{1,\dots,5\}.$
%\end{cor}

%\begin{rem}\label{Remark VirtReal}
%	It follows from \ref{Def:mainstratum} that $ \tilde{F}_{\sigma}\simeq F_{\sigma} $ for the stratum $ W_{\sigma} $ where $\sigma=\{ij\}$ or $ \sigma=\{ij,pq\}. $ For all other strata the corresponding virtual and real space of parameters are not homeomorphic.
%\end{rem}
Now, we use the projection $ p_{\omega_{0},12}:\mathcal{F}_{\omega_{0}}\to F_{\omega_{0}} $ defined in Lemma \ref{Lem:ProjG} to map the virtual spaces of parameters from relations \eqref{Th:UrelationsX5} to the space $ F_{\omega_{0}}.$ It directly follows from \eqref{X5Particija}. 
\begin{lemma}\label{L:relacijeG52}
	The homology group $ H_{2}(F_{\omega_{0}}) $ is generated by the classes determined by the virtual spaces of parameters $ F^{\prime}_{ij} $ and the following relations hold:
	\begin{equation*}
		 F^{\prime}_{ij}+F^{\prime}_{pq}=
		F^{\prime}_{ip}+F^{\prime}_{jq}=
		F^{\prime}_{iq}+F^{\prime}_{jp}
	\end{equation*}
for $ \{i,j,p,q\}\subset \{1,\dots,5\} $ where  \[ F^{\prime}_{ij}:=F_{ij}\cup\bigcup_{pq }F_{ij,pq}.\]
\end{lemma}
\par Let $ C_{\omega_{1}} $ be a chamber of the dimension 4 which differs from $ C_{\omega_{0}} $ by one $ K_{lm}, $ that is \\$ C_{\omega_{1}}\subset \Ins{K}_{ij},\,$ for $ ij\neq lm $ and $ C_{\omega_{1}}\nsubseteq \Ins{K}_{lm}.$  We write $  \mathcal{F}_{\omega_{1}} $ as:\[ \mathcal{F}_{\omega_{1}}= \tilde{F}_{\overline{lm}}\cup\bigcup_{ ij\neq lm }\tilde{F}_{ij}\cup\bigcup_{ij,pq\neq lm}\tilde{F}_{ij,pq}. \] 
\begin{itemize}
	\item If $ \tilde{F}_{\overline{ij}} $ is such that $ ij\neq lm $ and $ ij\nsubseteq \overline{lm} $  then $ p_{\omega_{1},12}(\tilde{F}_{\overline{ij}})=p_{\omega_{0},12}(\tilde{F}_{\overline{ij}}).$
	\item If $ ij\subset \overline{lm} $ then $ \tilde{F}_{\overline{ij}} $ is represented in $ \mathcal{F}_{\omega_{1}} $ as 
	\[ \tilde{F}_{ij}\cup\bigcup_{pq\neq lm }\tilde{F}_{ij,pq}\cup \tilde{F}_{\overline{lm}} \] and  \[ p_{\omega_{1},12}(\tilde{F}_{\overline{ij}})=F_{ij}\cup\bigcup_{pq\neq lm }F_{ij,pq}\cup F_{\overline{lm}}\simeq \C P^{1}. \] For simplicity, we write $ F_{ij}^{\prime} $ for this last union.
\end{itemize}
\par	A virtual space of parameters $ \tilde{F}_{\overline{lm}} $ projects to the space $ F_{\overline{lm}} $ and it is a point by Corollary \ref{Corr:VirtualspaceX5}. Therefore, the relations in the homology group $  H_{2}(F_{\omega_{1}})  $ are as in the homology group $ H_{2}(F_{\omega_{0}}) $, where we replace the space of parameters $ F^{\prime}_{lm} $ by the zero class. In this way we get relations in any $  H_{2}(F_{\omega})  $ by substituting $ F^{\prime}_{ij} $ with the zero class in relations for $ H_{2}(F_{\omega_{0}}) $ where $ C_{\omega}\notin K_{ij}. $

\par Thus, we obtain the result analogous to that of Lemma \ref{L:relacijeG52} for an arbitrary chamber.
\begin{lemma}\label{Rem:X5Cwrelacije}
	Let $ C_{\omega}= \bigcap_{ij\notin S}K_{ij}\cap\bigcap_{ij\in S}K_{\overline{ij}} $ where $ S\subset \{ij|1\leq i <j\leq 5\} $. The homology group $ H_{2}(F_{\omega}) $ is generated by the classes determined by the spaces of parameters $ F_{ij} $ and the following relations hold:
	\begin{align*}
		F_{ij}+F_{pq}=  
		F_{ip}+F_{jq}=
		F_{iq}+F_{jp}		  
	\end{align*}
	for distinct $ \{i,j,p,q\}\subset \{1,\dots,5\} $, where if any two indices of $ i,j,p,q $ are in $ S $ then we replace the corresponding class by the zero class.
\end{lemma}
\begin{definition}
	 The projection  $ p_{\omega_{0},\omega}:C_{2}(F_{\omega_{0}})\to C_{2}(F_{\omega}) $ is defined by:
	 \begin{equation}\label{eq:projomega0omega}
	 	p_{\omega_{0},\omega}(c)=\sum_{ij\in \hat{\omega}\cap \omega}F_{ij}
	 \end{equation}
 where  $ c=\sum_{ij\in \hat{\omega}} F_{ij} $ for $ \hat{\omega}\subset \omega_{0}.$
\end{definition}

%Note that the same relations hold for an arbitrary region $ C_{\omega_{1}} $ except the members $ F_{ij} $ which vanish in the relations for which $ C_{\omega_{1}}\notin K_{ij}. $
\begin{lemma}\label{L:regionG52}
		If a cycle $ l $ is the zero class in $ H_{2}(F_{\omega_{0}}) $, then $ p_{\omega_{0},\omega}(l) $ is the zero class in $ H_{2}(F_{\omega}) $ for a chamber $ C_{\omega} $ of maximal dimension.

%	Let $ c=(\sum_{\omega\in S}C_{\omega})\times l $ be a chain in $ X_{5} $ such that  $\omega_{0}\in S$ and $ l $ is a fixed two dimensional chain of the form $ \sum_{ij\in \hat{\omega}} F_{ij} $ for some set of indices $ D $. Suppose that the chain $ l $ is homologicaly equal to zero class in $ H_{6}(X_{5}) $ then the chain $ c $ is zero chain in $ H_{6}(X_{5}) $.
\end{lemma}
\begin{proof}
	Let a cycle $ l=\sum_{ij\in \hat{\omega}} F^{\prime}_{ij} $ be  zero in $ H_{2}(F_{\omega_{0}}) $ for some $ \hat{\omega}\subset\omega_{0}$. Let $ C_{\omega} $ be a chamber such that $ \omega\neq \omega_{0},$ that is  $ C_{\omega}\nsubseteq \Ins{K}_{pq} $ for some $ p,q\in \{1,\dots,5\} .$ Let $ pq\notin \hat{\omega}$, by Lemma \ref{L:relacijeG52} and \ref{Rem:X5Cwrelacije} we get that the cycle $ p_{\omega_{0},\omega}(l) $ is zero in $ H_{2}(F_{\omega})$. If $ pq\in \hat{\omega} $, then by Lemma \ref{Rem:X5Cwrelacije} we obtain that $ p_{\omega_{0},\omega}(l) =\sum_{\substack{ij\in \hat{\omega}\\ij\neq pq}} F_{ij} $, so it is zero $  H_{2}(F_{\omega}). $
\end{proof}
\begin{theorem}\label{Th:H6X5}
	 $ H_{6}(X_{5};\Z_{2})\simeq  \Z_{2}.$ 
\end{theorem}
\begin{proof} By Corollary \ref{Cycles_3n-9}  cycles of dimension 6 are generated by the cycles \[ c_{mi}:=e_{mi}+e_{m5}+e_{i5}, \] where $ e_{ij}=K_{ij,4}\times F_{ij} $ and $ 1\leq m \leq 3 $ and $ m+1 \leq i  < 5 $. 
	
%	 	\par We proceed to show if some sum of generators $ c_{mi} $ can be obtained as boundary of some 7 dimensional chain. If $ c^{\prime} $ is a chain $ c^{\prime}=\sum C_{\omega}\times l_{\omega}$ only possibility is that $ \dim(C_{\omega})=4 $ and $ \dim(l_{\omega})=3. $ It follows that $ l_{\omega} $ is a chain from the main space $ F. $ Then $ \partial c^{\prime} $ satisfies $ \partial c^{\prime}=\sum \partial C_{\omega} \times l_{\omega}+\sum  C_{\omega} \times \partial l_{\omega} $. This can be some cycle $ \sum c_{mi} $ only if the first sum vanishes and $ l_{\omega} $ satisfies $ \partial l_{\omega}=\sum F_{ij} .$
	 
	  Let $ c=\sum c_{mi}=\sum e_{pq}.$ On the other hand, we can write the cycle $ c $ \[ c=\sum_{\omega}\sum_{ij\in \omega} C_{\omega}\times F_{ij}=\sum_{\omega} C_{\omega}\times (\sum_{ij\in \omega}F_{ij}). \]
	  If there exists a chain $ \tilde{c}=\sum C_{\tilde{\omega}}\times l_{\tilde{\omega}} $ such that $ \partial \tilde{c}=c $ we have that \[  \partial \tilde{l}_{\omega}= \sum_{ij\in \omega}F_{ij},\,\text{ for all }\omega. \] Thus, a cycle $ c $ is homologically zero if $ \sum_{ij\in \omega}F_{ij}=0 $ in $ H_{2}(F_{\omega}) $ for all $ \omega $ appearing in $ c. $  
	  \par In particular, consider $ c $ defined by \[  c=c_{12}+c_{13}=e_{12}+e_{13}+e_{25}+e_{35}=\sum_{\omega} C_{\omega}\times(F_{12}+F_{13}+F_{25}+F_{35}). \]
	   We have that $ C_{\omega}\in K_{ij,4}$ for at least one $ ij\in\{12,13,25,35\} $ and observe that  $ C_{\omega_{0}} $ participates in this sum. To prove that $ c $ is zero in $ H_{6}(X_{5};\Z_{2}) $, it is sufficient by Lemma \ref{L:regionG52} to prove that $ C_{\omega_{0}}\times (F_{12}+F_{13}+F_{25}+F_{35})$ is zero in $ H_{6}(X_{5};\Z_{2}) $, that is $ F_{12}+F_{13}+F_{25}+F_{35}=0 $ in $ H_{2}(F_{\omega_{0}}). $  This is true by Lemma \ref{L:relacijeG52} taking indices $ \{1,2,3,5\}$, so $ c_{12}=c_{13} $ in $ H_{6}(X_{5};\Z_{2}) $. Similarly, using other choices of indices  in Lemma \ref{L:relacijeG52}, we deduce that $ c_{13}=c_{14},\,c_{13}=c_{23},\,c_{23}=c_{24} $ and $ c_{24}=c_{34} $ in $ H_{6}(X_{5};\Z_{2}). $
\end{proof}

\begin{theorem}\label{Th:Homology_H5_X5}
	 $ H_{5}(X_{5};\Z_{2})\simeq \Z_{2}.$
\end{theorem}
\begin{proof}
	By Proposition \ref{Th:Cycles_n-1} cycles of the dimension 5 are homologous to cycles from $ Y_{5} $. By Corollary \ref{Cor:Spaces_Y_n} only possible cycles of dimension 5 are:
	\[ c=\sum_{\sigma_m}O_{m,3}\times l^{m},\] where $ l^{m} $ is a chain on $ \hat{\mu}^{-1}(O_{m})$ that is a top dimensional cycle in $ F_{4} $  for $ 1\leq m\leq 5.$ Thus, the potential generators are:
	 \[ g_{m}= O_{m,3}\times l,  \]
	  for $ 1\leq m\leq 5.$  \par Let $\tilde{c}$ be a chain of dimension $ 6 $ such that $ \partial \tilde{c}=\sum g_{m} $. By Corollary \ref{The:HomologyChain_n-2_general} we can assume that every $ C_{\omega} $ is of the dimension $ 4 .$ By Corollary \ref{Cor:Chains_with_boundary_on_boundary} the chain $ \tilde{c} $ is of the form $ e_{ij}=K_{ij,4}^{\circ}\times F_{ij} $ and $ \partial e_{ij}=g_{i}+g_{j}. $ We conclude that the sum of every two generators $ g_{i} $ is  homologically equals to zero.
\end{proof}
\section{$\Z _{2}$-homology groups of $ X_{6} $} \label{Sec:HomX6}
The $\Z _{2}$-homology groups of $ X_{6} $ are not known up to our knowledge.
 By Lemma \ref{lemma_3n-7} and Corollary \ref{lemma_3n-8},  $\Z _{2} $-homology groups in dimensions 11 and 10 are : 
\[ H_{11}(X_{6};\Z _{2})\simeq \Z_{2}\text{, } H_{10}(X_{6};\Z _{2})\simeq 0. \] 
\par  Theorem \ref{Cor:Homology_less_n-3} implies that $ H_{k}(X_{6};\Z _{2})\simeq H_{k}(Y_{6};\Z _{2}),$ for $ 1\leq k\leq3. $ By Corollary \ref{Cor:Spaces_Y_n} a cycle $ c $ of this dimension $ k $ is homologous to a cycle from $ \hat{\mu}^{-1}(\bigcup_{1\leq m\leq 6}\partial \Delta_{5,2}(m)).$ We write  $ c $ as:
	\[c=c_{1}+c_{2}+c_{3},   \]where 
	\[  c_{1}\in \hat{\mu}^{-1}(\bigcup_{1\leq p<q\leq 6}\Ins{\Delta}_{4,2}(pq)) ,\, c_{2}\in \hat{\mu}^{-1}(\bigcup_{\substack{1\leq p\leq 6\\1\leq q\leq 6}\\p\neq q}\Ins{\Delta}_{4,1}(pq))\text{ and }  c_{3}\in \hat{\mu}^{-1}(\bigcup_{1\leq p<q\leq 6}\partial \Delta_{4,2}(pq)). \]
%	we denoted $ O_{pq}=O_{p}\cap O_{q}\simeq \D{4} $ and  $ S_{pq} $ the simplex with zero on $ p$-th coordinate and one on $ q $-th coordinate.
	 By Corollary \ref{The:HomologyChain_n-2_general} we can assume that $ l_{\omega} $ is a point for every $ C_{\omega} $ in $ c. $ In this case $ c $ is a cycle on $ \partial \D{6} = S^{4}, $ so
	\[ H_{k}(X_{6};\Z _{2})\simeq 0 \quad \text{ for } 1\leq k\leq3.\]
%	 Obviously,
%	\[ H_{0}(X_{6};\Z _{2})\simeq \Z_{2}. \] 
\par We give the admissible polytopes, the space of parameters and the virtual space of parameters, defined in Section \ref{Sec:Definitions}, for $X_{6}.$

% of the interest, in the sense of Corollary \ref{Corr:VirtualspaceX6}, for our calculation of the homology group $ H_{k}(X_{6}); $ for $ k\in\{9,8,7,6\}. $
% 
  For a non-orbit strata the admissible polytopes are given by:\begin{itemize}
	\item 5-dimensional
	\begin{itemize}
		\item $ \Delta_{6,2} $
		\item $ K_{ijk}: \,\{\textbf{x}\in \D{6}\,|x_{i}+x_{j}+x_{k}\leq 1 \};\quad  K_{ij}: \,\{\textbf{x}\in \D{6}\,| x_{i}+x_{j}\leq 1 \}  $  
		\item $ K_{ijpq}: \,\{\textbf{x}\in \D{6}\,|  x_{i}+x_{j}+x_{p}+x_{q}\leq 1 \}$
		\item$  K_{ij,pq}:\,\{\textbf{x}\in \D{6}\,|x_{i}+x_{j}\leq 1 ,\,x_{p}+x_{q}\leq 1 \} $ where $ \{i,j\}\cap\{p,q\}=\emptyset $ 
			\item $ K_{ij,pq,lm}:\,\{\textbf{x}\in \D{6}\,| x_{i}+x_{j}\leq 1,\, x_{p}+x_{q}\leq 1,\, x_{l}+x_{m}\leq 1\}$ where\\ $ \{i,j\}\cap\{p,q\}=\emptyset,\{i,j\}\cap\{l,m\}=\emptyset,\{p,q\}\cap\{l,m\}=\emptyset$
			\item $ K_{ij,pqr}:\,\{\textbf{x}\in \D{6}\,| x_{i}+x_{j}\leq 1,\, x_{p}+x_{q}+x_{r}\leq 1,\,\{i,j\}\cap\{p,q,r\}=\emptyset\} $  
	\end{itemize}

	\item 4-dimensional
	\begin{itemize}
		\item $ O_{m}: \,\{\textbf{x}\in \D{6}\,|x_{m}=0 \} $
			\end{itemize}
			\end{itemize}
		According to Subsection \ref{SSec:UniversalSpace} we have:
		\begin{cor}\label{Cor:MSpaceofparamX6}
			The space of parameters $ F_{6} $ of the main stratum $ W $ is homeomorphic to 
			\begin{align*}
				\{((c_{12}:c_{12}^{\prime}),(c_{13}:c_{13}^{\prime}),(c_{14}:c_{14}^{\prime}),(c_{23}:c_{23}^{\prime}),(c_{24}:c_{24}^{\prime}),(c_{34}:c_{34}^{\prime}))\in (\CPA)^{6}|\\\,c_{ij}^{\prime}c_{ik}c_{jk}^{\prime}=c_{ij}c_{ik}^{\prime}c_{jk},\,\text{ for } ijk\in\{123,124,134,234\}\}
			\end{align*} 
		\end{cor}
	The spaces of parameters of the strata we obtain by straightforward calculations.
	\begin{lemma}\label{Corr:SpaceofparametersX6}
		The spaces of parameters $ F_{\sigma} $ for the strata $ W_{\sigma} $ are:
		\begin{itemize}
			\item  $ P_{\sigma}=\Delta_{6,2} \Rightarrow F_{\sigma}\simeq F_{6};$
			\item $ P_{\sigma}=K_{ij} \Rightarrow F_{\sigma}\simeq F_{5};\quad P_{\sigma}=K_{ijp}\Rightarrow F_{\sigma}\simeq \CPA $
			\item  $ P_{\sigma}=K_{ij,pq} \Rightarrow F_{\sigma}\simeq \CPA ;\quad P_{\sigma}=O_{i} \Rightarrow F_{\sigma}\simeq F_{5} $
			\item  $ P_{\sigma}=K_{ijkl}\Rightarrow F_{\sigma} $  is a point;  $ P_{\sigma}=K_{ij,pq,lm}\Rightarrow F_{\sigma} $ is a point
			\item for others $ P_{\sigma}\Rightarrow F_{\sigma}$ is a point			
 		\end{itemize}			
	\end{lemma}
\par According to Theorem \ref{Th:UniversalSpaceG}, see also \cite{buchstaber2019toric} it holds:
\begin{theorem}\label{Th:Universal_theorem}
	The universal space of parameters $ \mathcal{F}_{6} $  is the blowup of $ \overline{F_{6}} $ firstly at the point $ ((1:1)^{6})$ and then at the  submanifolds $ \overline{F}_{345},\overline{F}_{346},\overline{F}_{356},\overline{F}_{456} $
\end{theorem}
 The virtual spaces of parameters for the strata we obtain using the definitions from Subsection \ref{SSec:UniversalSpace} by direct computations:
\begin{lemma}\label{Corr:VirtualspaceX6}
	The virtual spaces of parameters are given by:
	\begin{enumerate}
		\item  $ P_{\sigma}=\Delta_{6,2} \Rightarrow F_{\sigma}\simeq\tilde{F}_{\sigma}\simeq F_{6};\quad P_{\sigma}=K_{ij} \Rightarrow F_{\sigma}\simeq\tilde{F}_{\sigma}\simeq F_{5}, $
		\item  $ P_{\sigma}=K_{ijp} \Rightarrow \tilde{F}_{\sigma}\simeq \C P^{1}\times \CPA;\quad  P_{\sigma}=K_{ijpq} \Rightarrow \tilde{F}_{\sigma}\simeq \mathcal{F}_{5} $

%		 F_{5}\cup \bigcup_{\text{pt}\in B}(\text{pt}\times(c_{1}:c_{1}^{\prime})\times (B\backslash pt))\cup \bigcup_{pt\in B}((c_{1}:c_{1}^{\prime})\times pt\times (B\backslash pt))\cup\\ \cup ((1:1)\times (c:c^{\prime})\times (c:c^{\prime}))\cup ((c:c^{\prime})\times (1:1)\times (c:c^{\prime}))\cup\bigcup_{pt \in B}pt\times pt \times (c:c^{\prime})	where $ B=\{(1:0),(0:1)\} $ and $ (c:c^{\prime})\in \C P^{1} $ and $ (c_{1}:c_{1}^{\prime})\in \C P^{1}\backslash pt $ of the union. This union is isomorphic to 
%			\[ ((\C P^{1}\times \C P^{1})\backslash(\bigcup_{pt\in B}pt\times pt))\cup \bigcup_{pt\in B}pt\times pt\times \C P^{1}. \]
		\item  $ P_{\sigma}=K_{ij,pq} \Rightarrow F_{\sigma}=\tilde{F}_{\sigma}\simeq \CPA ;\quad P_{\sigma}=K_{ij,pq,lm} \Rightarrow F_{\sigma}=\tilde{F}_{\sigma} $ is a point,
		\item  $ P_{\sigma}=K_{ij,pqr}\Rightarrow\tilde{F}_{\sigma}\simeq \C P^{1} $ ,  
		\item  $ P_{\sigma}=O_{i} \Rightarrow \tilde{F}_{\sigma}\simeq \pi^{-1}(( F_{5}\times \C P^{1}\times \C P^{1}\times \C P^{1})\cap \bar{F}_{6})\subset \mathcal{F}_{6} $, where $ \pi:\mathcal{F}_{6}\to \bar{F}_{6}  $ is a natural projection morphism.
		\item for others $ P_{\sigma} \Rightarrow F_{\sigma}$ is a point and $ \tilde{F}_{\sigma} $ is not a point.
	\end{enumerate}
\end{lemma}
\begin{proof}
	The proof for 2. and 5. follows from the description of $ \mathcal{F}_{6} $ as iterated blow ups given by Theorem \ref{Th:Universal_theorem}.
\end{proof}
\par Following notation above of $ P_{\sigma} $ by $ K_{J} $ we further denote $ F_{\sigma} $ and $ \tilde{F}_{\sigma} $ by $ F_{J} $ and $ \tilde{F}_{J} $ as well.
\par We establish the generators and relations in the homology group of $ \mathcal{F}_{6} $ using the map defined in Example \ref{ex:map06}, and the description of homology for $ \mathcal{M}_{0,6} $
\begin{lemma}
	The homology group $ H_{4}(\mathcal{F}_{6}) $ is generated by $ \tilde{f}(D_{I}) $ and the following relations hold:
	\begin{align}\label{Th:UrelationsX6}
		\sum_{\substack{i,j\in I\\p,q\notin I} }\tilde{f}(D_{I})=\sum_{\substack{i,p\in I\\j,q\notin I} }\tilde{f}(D_{I})=\sum_{\substack{i,q\in I\\j,p\notin I} }\tilde{f}(D_{I})
	\end{align}
	
%	\begin{align}
%		\tilde{F}_{ijrs}+\tilde{F}_{ijr}+\tilde{F}_{ijs}+\tilde{F}_{pqrs}&= \nonumber\\
%		\tilde{F}_{iprs}+\tilde{F}_{ipr}+\tilde{F}_{ips}+\tilde{F}_{jqrs}&=\\
%		\tilde{F}_{iqrs}+\tilde{F}_{iqr}+\tilde{F}_{iqs}+\tilde{F}_{jprs}	\nonumber	  
%	\end{align}
	
	where  $ \{i,j,p,q\}\subset \{1,\dots,6\} $, the map $ \tilde{f}:\mathcal{M}_{0,6}\to\mathcal{F}_{6}  $ is given in Example  \ref{ex:map06} and $ \tilde{f}(D_{I}) $ are given by:
		\begin{align}\label{Th:UconnectX6}
		\begin{split}
			D_{12}\to \tilde{F}_{1236}\quad D_{13}\to \tilde{F}_{1235}\quad D_{14}\to \tilde{F}_{2356}\quad D_{15}\to \tilde{F}_{1256}\quad D_{16}\to \tilde{F}_{1356}\\
			D_{23}\to \tilde{F}_{1234}\quad D_{24}\to \tilde{F}_{2346}\quad D_{25}\to \tilde{F}_{1246}\quad D_{26}\to \tilde{F}_{1346}\quad D_{34}\to \tilde{F}_{2345}\\
			D_{35}\to \tilde{F}_{1245}\quad D_{36}\to \tilde{F}_{1345}\\
			\hline
			D_{123}\to\tilde{F}_{456}\cup\tilde{F}_{456,12}\cup\tilde{F}_{456,13}\cup\tilde{F}_{456,23}\quad D_{124}\to\tilde{F}_{145}\cup\tilde{F}_{145,36}\cup\tilde{F}_{145,26}\cup\tilde{F}_{145,23}\\ D_{125}\to\tilde{F}_{345}\cup\tilde{F}_{345,16}\cup\tilde{F}_{345,26}\cup\tilde{F}_{345,12}\quad
			D_{126}\to\tilde{F}_{245}\cup\tilde{F}_{245,16}\cup\tilde{F}_{245,36}\cup\tilde{F}_{245,13}\\
			D_{134}\to\tilde{F}_{235}\cup\tilde{F}_{235,16}\cup\tilde{F}_{235,46}\cup\tilde{F}_{235,14}\quad
			D_{135}\to\tilde{F}_{346}\cup\tilde{F}_{346,15}\cup\tilde{F}_{346,25}\cup\tilde{F}_{345,12}\\ D_{136}\to\tilde{F}_{246}\cup\tilde{F}_{246,15}\cup\tilde{F}_{246,35}\cup\tilde{F}_{246,13}\quad D_{234}\to\tilde{F}_{156}\cup\tilde{F}_{156,34}\cup\tilde{F}_{156,24}\cup\tilde{F}_{156,23}\\ D_{235}\to\tilde{F}_{356}\cup\tilde{F}_{356,14}\cup\tilde{F}_{356,24}\cup\tilde{F}_{356,12}\quad D_{236}\to\tilde{F}_{256}\cup\tilde{F}_{256,14}\cup\tilde{F}_{256,34}\cup\tilde{F}_{256,13}\\
			\hline
			D_{1234}\to \tilde{F}_{1456}\quad D_{1235}\to\tilde{F}_{3456}\quad D_{1236}\to\tilde{F}_{2456}
		\end{split}
	\end{align}
	
\end{lemma}
\begin{proof}
	Let us fix chart $ M_{12}. $ We consider the image $ \tilde{f}(D_{I})$ and we want to write it down as the union of a virtual spaces of parameters.
	\par  For example, the divisor $ D_{12} $ maps to the subset 
	\begin{align}\label{eq:divisorD12}
	&D_{12}=((c:c^{\prime}),(c:c^{\prime}),(c_{1}:c_{1}^{\prime}))\\ &\mapsto((c:c^{\prime}),(c:c^{\prime}),(c_{1}:c_{1}^{\prime}),(1:1),(cc_{1}^{\prime}:c^{\prime}c_{1}),(cc_{1}^{\prime}:c^{\prime}c_{1})).\nonumber
	\end{align} If $ \tilde{F}_{\sigma}\subseteq\tilde{f}(D_{12}) $ the points of its stratum $ W_{\sigma} $ satisfies $ P^{45}=0.$ On the other hand, the virtual space of parameters $ \tilde{F}_{1236} $ in the chart $ M_{14} $ is written as \[ ((c_{1}:c_{1}^{\prime}),(1:0),(c_{2}:c_{2}^{\prime}),(1:0),(c_{3}:c_{3}^{\prime}),(0:1)).\]
Then by the homeomorphism $ \tilde{f}_{14,12}:\mathcal{F}_{14}\to\mathcal{F}_{12} $ we get $  \tilde{F}_{1236}=\tilde{f}(D_{12}) $. 
\par The divisor $ D_{124} $ is defined as  
\[(0:1)\times(0:1)\times(c:c^{\prime})\times(0:0:1)\times\C P^{1}.  \] We determine the virtual spaces of parameters $ \tilde{F}_{\sigma}\subseteq\tilde{f}(D_{124}) $ for different value of $ (c:c^{\prime}) $ as:
\begin{itemize}
	\item If $ (c:c^{\prime})\notin A $ then we have \[ ((0:1),(0:1),(c:c^{\prime}),(x_{1}:x_{2}))\to ((0:1),(0:1),(c:c^{\prime}),(x_{1}:x_{2}),(1:0),(1:0)) \] as described in Example  \ref{ex:map06}. By \eqref{Def:mainstratumG}, we get $ \tilde{F}_{145}$.
	\item If $ (c:c^{\prime})=(1:0) $ we get  $ \tilde{F}_{145,26}. $
	\item If $ (c:c^{\prime})=(1:1) $ we get  $ \tilde{F}_{145,36}. $
	\item If $ (c:c^{\prime})=(0:1) $ we get  $ \tilde{F}_{145,23}. $
\end{itemize} Altogether, $ \tilde{f}(D_{124})=\tilde{F}_{145}\cup\tilde{F}_{145,26}\cup\tilde{F}_{145,36}\cup\tilde{F}_{145,23}. $
We deduce the rules for mapping the divisors:

\begin{itemize}
	\item divisor $ D_{I} $, where $ |I|=2 $ or $ |I|=4 $ maps to the corresponding virtual space of parameters $ \tilde{F}_{J}\simeq \C P^{1}\times \C P^{1} $ such that $ |J|=4 $ according to table \eqref{Th:UconnectX6},
%	\item The divisor $ 	D_{123}\to\tilde{F}_{456}\cup\tilde{F}_{1456}\cup\tilde{F}_{3456}\cup\tilde{F}_{2456} $ as the divisor $ D_{123} $ contains the divisors $ D_{1234},D_{1235} $ and $ D_{1236}, $
	\item  divisor $ D_{I} $, where $ |I|=3 $  maps to 
	\[ \tilde{F}_{J}\cup\bigcup_{kl}\tilde{F}_{J,kl}\simeq (\C P^{1}\times \CPA) \cup (\C P^{1}\times A)\simeq \C P^{1}\times \C P^{1} \] where we determine the set of indices $ J\,,|J|=3 $ similarly as for $ D_{124} $.
\end{itemize}
The relations follow from the relations in homology of $ \mathcal{M}_{0,6} $ given by \eqref{relacije}.
\end{proof}
 A chamber $ C_{\omega} $ of the maximal dimension in $ \D{6} $ is obtained as $ C_{\omega}=\bigcap\Ins{K}_{1ij} \cap\bigcap \Ins{K}_{pq} $ for all $ 1ij\in\omega $ and $ pq\in\omega. $  Let us denote by $ C_{\omega_{0}} = \bigcap_{ ij}\Ins{K}_{1ij}.$  We have disjoint union \[ \mathcal{F}_{6}=\mathcal{F}_{\omega_{0}}= \bigcup_{C_{\omega_{0}}\in \overset{\circ}{P_{\sigma}} }\tilde{F}_{\sigma}. \]  It can be expressed in more details by
 
 \begin{equation}\label{X6eqF_omeqa0}
 	\mathcal{F}_{\omega_{0}}= \bigcup_{ij }\tilde{F}_{ij}\cup\bigcup_{pq,ij}\tilde{F}_{ij,pq}\cup\bigcup_{pq }\tilde{F}_{1pq}\cup\bigcup_{ ij,pq}\tilde{F}_{ij,1pq}\cup\bigcup_{ij,pq,lm}\tilde{F}_{ij,pq,lm},
 \end{equation} where all indices in each union are distinct according to the description of admissible polytopes from the beginnings of this section.
  
    Therefore, in $ \mathcal{F}_{\omega_{0}} $ we can decompose any $ \tilde{F}_{\sigma} $ contributing to \eqref{Th:UrelationsX6}, and  then, using the projection $ p_{\omega_{0},12}:\mathcal{F}_{\omega_{0}}\to F_{\omega_{0}} $ in the chart $ M_{12} $, we determine the relations in a homology group of $ F_{\omega_{0}} $.
\begin{prop}\label{Cor:relacijeG62}
	The homology group $ H_{4}(F_{\omega_{0}}) $ is generated by the homology classes $ F_{ij} $ and the following relations hold:
	\begin{align}\label{Cor:G62_4dimRelacije}
		F_{ij}+F_{pq} = F_{ip}+F_{jq}=F_{iq}+F_{jp}.  
	\end{align}
The homology group $ H_{2}(F_{\omega_{0}}) $ is generated by the homology classes $ F_{ij,kl}, F_{1kl} $ and the following relations hold :
		\begin{align}\label{Cor:G62_2dimRelacije}
		\sum_{sm} F_{ij,sm}+\sum_{sm} F_{pq,sm}=\sum_{sm} F_{ip,sm}+\sum_{sm} F_{jq,sm}=\sum_{sm} F_{iq,sm}+\sum_{sm} F_{jp,sm}
	\end{align}
	for distinct $ \{i,j,p,q\}\subset \{1,\dots,6\}.$
\end{prop}
\begin{proof}
	Homology generators from relations \eqref{Th:UconnectX6} decompose in $ \mathcal{F}_{\omega_{0}} $ as follows:
	\begin{itemize}
		\item  $ \tilde{F}_{1jk} $ is an element of $ \mathcal{F}_{\omega_{0}} $,
		\item   For $ \tilde{F}_{ijk} $ , $ i,j,k\neq1 $ it holds: \[  \tilde{F}_{ijk}=\tilde{F}_{\overline{ijk}}\cup\tilde{F}_{\overline{ijk},ij}\cup\tilde{F}_{\overline{ijk},ik}\cup\tilde{F}_{\overline{ijk},jk} \]  
		\item For $ \tilde{F}_{1jpq} $ it holds \[ \tilde{F}_{1jpq}=\tilde{F}_{\overline{1jpq}}\cup\tilde{F}_{1\overline{1jpq}}\cup \bigcup_{ kl}\tilde{F}_{\overline{1jpq},kl}\cup\bigcup_{ kl}\tilde{F}_{\overline{1jpq},1kl}\cup\bigcup_{ kl}\tilde{F}_{1\overline{1jpq},kl}\cup\bigcup_{ kl,mn}\tilde{F}_{\overline{1jpq},kl,mn}   \]
		\item  For $ \tilde{F}_{ijpq} $ it holds \[ \tilde{F}_{ijpq}=\tilde{F}_{\overline{ijpq}}\cup\bigcup_{k\in\{i,j,p,q\}}\tilde{F}_{k\overline{ijpq}}\cup\bigcup_{ kl}\tilde{F}_{\overline{ijpq},kl}\cup\bigcup_{ k,lm}\tilde{F}_{k\overline{ijpq},lm} \cup\bigcup_{ kl,mn}\tilde{F}_{\overline{ijpq},kl,mn} \]
	\end{itemize}
	By Lemma \ref{Corr:VirtualspaceX6} an element $ \tilde{F}_{ij} $ projects to the 4-dimensional space of parameters $ F_{ij} $ and  $ \tilde{F}_{ij,pq}$, $\tilde{F}_{ijp} $ project to the 2-dimensional space of parameters $ F_{ij,pq} $,$ F_{ijp} $ respectively. Any other virtual space of parameters projects to a point. Substituting this in \eqref{Th:UrelationsX6} and grouping by dimension we deduce the result. We demonstrate calculations for $ I=\{2,3,5,6\}:$
			\begin{align}\label{Eq:X6relacije}
		\tilde{F}_{1456}+(\tilde{F}_{456}\cup\tilde{F}_{456,12}\cup\tilde{F}_{456,13}\cup\tilde{F}_{456,23})+(\tilde{F}_{156}\cup\tilde{F}_{156,34}\cup\tilde{F}_{156,24}\cup\tilde{F}_{156,23})+\tilde{F}_{1234}&= \\ \nonumber
		\tilde{F}_{1346}+(\tilde{F}_{346}\cup\tilde{F}_{346,15}\cup\tilde{F}_{346,25}\cup\tilde{F}_{346,12})+(\tilde{F}_{245}\cup\tilde{F}_{245,16}\cup\tilde{F}_{245,36}\cup\tilde{F}_{245,13})+\tilde{F}_{1245}&=\\
		\tilde{F}_{1246}+(\tilde{F}_{345}\cup\tilde{F}_{345,16}\cup\tilde{F}_{345,26}\cup\tilde{F}_{345,12})+(\tilde{F}_{246}\cup\tilde{F}_{246,15}\cup\tilde{F}_{246,35}\cup\tilde{F}_{246,13})+\tilde{F}_{1345}.	  \nonumber
	\end{align}
%		\begin{align}\label{Eq:X6relacije}
%		\tilde{F}_{1236}+\tilde{F}_{136}+\tilde{F}_{126}+\tilde{F}_{2345}&= \\ \nonumber
%		\tilde{F}_{1234}+\tilde{F}_{134}+\tilde{F}_{124}+\tilde{F}_{2356}&=\\
%		\tilde{F}_{1235}+\tilde{F}_{135}+\tilde{F}_{125}+\tilde{F}_{2346}.	  \nonumber
%	\end{align}

\par For example, $ p_{\omega_{0},12}(\tilde{F}_{1456}) $ is the union of the $4$=dimensional  cell $ F_{23} $, $ 2 $-dimensional cells $ F_{23,kl},F_{123} $ and points $ F_{123,kl},F_{23,1kl},F_{23,kl,mn}$ for all possible indices  $ kl,mn.$ We can denote this union as $ F_{23}^{\prime}\simeq \C P^{1}\times \C P^{1}. $
%  Since there are not $ 3 $-dimensional cells and $ 1 $-dimensional cells of $ F_{23}^{\prime} $ then  $ 4 $-dimensional and $ 2 $-dimensional cells are cycles.
    We do the same for each element from~\eqref{Eq:X6relacije}.

% We subdivision each virtual space of parameters on $ \mathcal{F}_{\omega_{0}} $ and project on $ F_{\omega_{0}} $. For example, after projecting of the element $ \tilde{F}_{1236} $ we get $ 4 $-dimensional cell $ F_{45} $, $ 2 $-dimensional cells $ F_{45,kl} $ and points $ F_{145,kl},F_{45,1kl},F_{45,kl,mn}$ for all possible indexes $ kl,mn.$ We can denote this union as $ F_{45}^{\prime}. $ Since there are not $ 3 $-dimensional cells and $ 1 $-dimensional cells of $ F_{45}^{\prime} $ then  $ 4 $-dimensional and $ 2 $-dimensional cells are cycles.  We continue in this fashion for each element.

%\begin{align*}
%	F_{45}^{\prime}+F_{136}+F_{126}+F_{16}^{\prime}&=\\
%	F_{56}^{\prime}+F_{134}+F_{124}+F_{14}^{\prime}&=\\
%	F_{46}^{\prime}+F_{135}+F_{125}+F_{15}^{\prime}
%\end{align*}
%\begin{align*}
%	F_{45}+F_{145}+\sum_{kl}F_{45,kl}+\sum_{kl}F_{145,kl}+\sum_{kl}F_{45,1kl}+\sum_{kl,mn}F_{45,kl,mn}+F_{136}+F_{126}&+\\+F_{16}+\sum_{k}F_{1k6}+\sum_{kl}F_{16,kl}+\sum_{k}\sum_{mn}F_{16k,mn}+\sum_{kl,mn}F_{16,kl,mn}&=\\=
%	F_{56}+F_{156}+\sum_{kl}F_{56,kl}+\sum_{kl}F_{156,kl}+\sum_{kl}F_{1kl,56}+\sum_{kl,mn}F_{56,kl,mn}+F_{134}+F_{124}&+\\+F_{14}+\sum_{k}F_{1k4}+\sum_{kl}F_{14,kl}+\sum_{k}\sum_{mn}F_{14k,mn}+\sum_{kl,mn}F_{14,kl,mn}&=\\=
%	F_{46}+F_{146}+\sum_{kl}F_{46,kl}+\sum_{kl}F_{146,kl}+\sum_{kl}F_{1kl,46}+\sum_{kl,mn}F_{46,kl,mn}+F_{135}+F_{125}&+\\+F_{15}+\sum_{k}F_{1k5}+\sum_{kl}F_{15,kl}+\sum_{k}\sum_{mn}F_{15k,mn}+\sum_{kl,mn}F_{15,kl,mn}.
%\end{align*}

 Taking only the 4-dimensional chains, that are obtained only by the spaces $ \tilde{F}_{J},\,|J|=4 $, we get:
\begin{align*}
	F_{23}+F_{56}=F_{25}+F_{36}=F_{35}+F_{26}.
\end{align*} The 2-dimensional chains $ F_{ijk} $ cancel out as it can be seen from relations:

\begin{align*}
	F_{123}+\sum_{kl}F_{23,kl}+F_{123}+F_{156}+F_{156}+\sum_{kl}F_{56,kl}&=\\
	F_{125}+\sum_{kl}F_{25,kl}+F_{125}+F_{136}+F_{136}+\sum_{kl}F_{36,kl}&=\\
	F_{135}+\sum_{kl}F_{35,kl}+F_{126}+F_{135}+F_{126}+\sum_{kl}F_{26,kl}.
\end{align*} It follows:

\begin{align*}
	\sum_{kl}F_{23,kl}+\sum_{kl}F_{56,kl}=\sum_{kl}F_{25,kl}+\sum_{kl}F_{36,kl}=\sum_{kl}F_{35,kl}+\sum_{kl}F_{26,kl}.
\end{align*} 

%\begin{align*}
%	F_{145}+F_{146}+F_{156}+\sum_{kl}F_{45,kl}+\sum_{kl}F_{16,kl}&=\\
%	F_{156}+F_{145}+F_{146}+\sum_{kl}F_{56,kl}+\sum_{kl}F_{14,kl}&=\\
%	F_{146}+F_{156}+F_{145}+\sum_{kl}F_{46,kl}+\sum_{kl}F_{15,kl}.
%\end{align*}
 Similarly, we proceed for any indices $ \{i,j,p,q\}\subset \{1,\dots,6\}. $
\end{proof}
\par	Let now $ C_{\omega} $ be an arbitrary  chamber. We note that every chamber of the maximal dimension is determined by the choice of $ K_{ijk} $ and $ K_{pq}$ where $ i $ is fixed. Let us generalize the  notation for $ C_{\omega_{0}} $ by $ C_{\omega_{s}} $ to denote the chamber given by $ \bigcap_{ij}K_{(s+1)ij} $ for fixed $ s\in\{1,\dots,5\} $.
\par Consider a chamber $ C_{\omega_{s}} $ and the decomposition of $ \mathcal{F}_{6} $ given by 
	\begin{equation}\label{X6eqF_omeqax}
		\mathcal{F}_{6}=\mathcal{F}_{\omega_{s}}= \bigcup_{ij }\tilde{F}_{ij}\cup\bigcup_{pq,ij}\tilde{F}_{ij,pq}\cup\bigcup_{ pq }\tilde{F}_{(s+1)pq}\cup\bigcup_{ ij,pq}\tilde{F}_{ij,(s+1)pq}\cup\bigcup_{ ij,pq,lm}\tilde{F}_{ij,pq,lm}
	\end{equation} The relations from  Corollary \ref{Cor:relacijeG62} for $ F_{\omega_{0}} $ are true for the space $ F_{\omega_{s}} $ as well. This can be proved similarly if we replace index $ 1 $ by $ s+1 $.
	\par Now, it is enough to take the chamber $ C_{\omega}= \bigcap_{ij\notin S}K_{1ij}\cap \bigcap_{pq\notin D}K_{pq} $ where $ S,D\subset \{ij|1\leq i <j\leq 6\} $ are chosen such that $ C_{\omega}\neq \emptyset. $ Then $ \mathcal{F}_{\omega } $ is equal to
	\begin{align*}
		 \mathcal{F}_{\omega}=\bigcup_{ij\notin D }\tilde{F}_{ij}\cup\bigcup_{ ij\in D }\tilde{F}_{\overline{ij}}\cup\bigcup_{\substack{pq\notin D\\ ij\notin D}}\tilde{F}_{ij,pq}\cup\bigcup_{ij,pq,lm\notin D }\tilde{F}_{ij,pq,lm}\cup\\
		 \cup\bigcup_{\substack{pq\notin S\\ij\notin D}}\tilde{F}_{ij,1pq}
		 \cup\bigcup_{\substack{ pq\in S\\ ij\notin D}}\tilde{F}_{ij,\overline{1pq}}\cup\bigcup_{ pq\notin S }\tilde{F}_{1pq}\cup\bigcup_{pq\in S }\tilde{F}_{\overline{1pq}}.
	\end{align*}
	 
\par We consider relations in $ H_{4}(F_{\omega}). $ By Corollary \ref{Corr:VirtualspaceX6}, we can get a 4-dimensional elements $ F_{ij} $ only as projection of the virtual spaces of parameters $ \tilde{F}_{ij} $ . Since $ \tilde{F}_{ij}\in\mathcal{F}_{\omega} $ if $ ij\notin D $, it follows that $ F_{ij} $ will participate in the relations for  $ H_{4}(F_{\omega})  $  if $ ij\notin D. $ 
Then the relations in $ H_{4}(F_{\omega}) $ are the same as \eqref{Cor:G62_4dimRelacije}, where we replace the space of parameters $ F_{ij} $ by the zero class for $ ij\in D $ .

  \par  The relations in $  H_{2}(F_{\omega}) $ are affected by the projection of the virtual spaces of parameters of the form $ \tilde{F}_{ijp} $ and $ \tilde{F}_{ij,pq}. $ Applying previous procedure we also have that every element $ F_{ijk} $ cancels out. If $ \tilde{F}_{(1pq}\in\mathcal{F}_{\omega}$ it will be canceled out as in the calculations proceeding \eqref{Cor:G62_2dimRelacije}. If $ \tilde{F}_{1pq}\notin\mathcal{F}_{\omega}$ then $ \tilde{F}_{\overline{1pq}}\in\mathcal{F}_{\omega} $ and the element $ F_{\overline{1pq}} $ stands in the relations of $ H_{2}(F_{\omega}) $ exactly as the element $ F_{1pq}   $ stands in the relations of $  H_{2}(F_{\omega_{0}}).$   We demonstrate calculations in a particular example as the general calculation requires a cumbersome using of indices. 
  \par Consider the chamber $ C_{\omega}=\bigcap_{ij\notin S}K_{1ij}\cap\bigcap_{ij\notin S}K_{ij}, $ for the set $ S=\{26,46\}. $ By decomposing any element from \eqref{Eq:X6relacije} in $ \mathcal{F}_{\omega} $ and  projecting it to $ F_{\omega} $, we obtain:
  \begin{itemize}
  	\item the relation in $ H_{4}(F_{\omega}) $:
  	\begin{align*}
  		F_{23}+F_{56}=F_{25}+F_{36}=F_{35},
  	\end{align*} where the element $ F_{26} $ is missing out since $ \tilde{F}_{1345} $ is an element of $ \mathcal{F}_{\omega} $ and it is projected to a  a point,
  \item the relation in $ H_{2}(F_{\omega}) $:
      \begin{align*}
  	F_{123}+F_{235}+\sum_{kl\notin S}F_{23,kl}+F_{123}+F_{156}+F_{156}+\sum_{kl\notin S}F_{56,kl}&=\\
  	=F_{125}+F_{235}+\sum_{kl\notin S}F_{25,kl}+F_{125}+F_{136}+F_{136}+\sum_{kl\notin S}F_{36,kl}&=\\
  	=F_{135}+F_{345}+F_{235}+\sum_{kl\notin S}F_{35,kl}+F_{345}+F_{135}.
  \end{align*}

%    \begin{align*}
%  F_{145}+\sum_{kl\notin S}F_{45,kl}+F_{136}+F_{126}+\sum_{k\neq 4}F_{1k6}+\sum_{kl\notin S}F_{16,kl}&=\\=
%  	F_{156}+F_{256}+\sum_{kl\notin S}F_{56,kl}+F_{256}+F_{124}+F_{124}+F_{145}+\sum_{kl}F_{14,kl}&=\\=
%  	F_{135}+F_{125}+\sum_{k}F_{1k5}+\sum_{kl\notin S}F_{15,kl}.
%  \end{align*} 
It follows:
  
  \begin{align*}
  &\sum_{kl\notin S}F_{23,kl}+\sum_{kl\notin S}F_{56,kl}=\sum_{kl\notin S}F_{25,kl}+\sum_{kl\notin S}F_{36,kl}=\sum_{kl\notin S}F_{35,kl}.
  \end{align*}
  \end{itemize}
%  \begin{align*}
%  	F_{45}+F_{145}+\sum_{kl\notin S}F_{45,kl}+\sum_{kl\notin S}F_{145,kl}+\sum_{kl\notin S}F_{45,1kl}+\sum_{kl,mn\notin S}F_{45,kl,mn}+F_{136}+F_{126}&+\\+F_{16}+\sum_{k\neq 4}F_{1k6}+\sum_{kl\notin S}F_{16,kl}+\sum_{k\neq4}\sum_{mn\notin S}F_{16k,mn}+F_{16,235}+\sum_{kl,mn\notin S}F_{16,kl,mn}&=\\=
%  	F_{56}+F_{156}+F_{256}+\sum_{kl\notin S}F_{56,kl}+\sum_{kl\notin S}F_{156,kl}+\sum_{kl}F_{1kl,56}+\sum_{kl,mn\notin S}F_{56,kl,mn}+F_{256}+F_{124}&+\\+F_{14}+F_{124}+F_{145}+\sum_{kl}F_{14,kl}+\sum_{k\neq\{3,6\}}\sum_{mn}F_{14k,mn}+\sum_{kl,mn}F_{14,kl,mn}&=\\=
%  	F_{1235}+F_{135}+F_{125}+F_{15}+\sum_{k}F_{1k5}&+\\+\sum_{kl\notin S}F_{15,kl}+\sum_{k}\sum_{mn\notin S}F_{15k,mn}+\sum_{kl,mn\notin S}F_{15,kl,mn}.
%  \end{align*} Next, we get:
%
%\begin{align*}
%		F_{45}+\sum_{kl}F_{45,kl}+F_{16}+\sum_{kl\notin S}F_{16,kl}+\text{points}&=\\=
%	F_{56}+\sum_{kl\notin S}F_{56,kl}+F_{14}+\sum_{kl}F_{14,kl}+\text{points}&=\\=
%	F_{15}+\sum_{kl\notin S}F_{15,kl}+\text{points}.
%\end{align*}

Therefore, an analogous result to those for $ C_{\omega_{0}} $ can be obtained for arbitrary chamber $ C_{w} $.

%   The virtual spaces $ \tilde{F}_{ijpq} $ and $ \tilde{F}_{ijp} $ that affect the 2-dimensional part projects in the following way:
%  \begin{itemize}
%  	\item If $\overline{ijpq}\in D$ then $ \tilde{F}_{ijpq} $ is part of the union $ \mathcal{F}_{\omega} $, therefore it projects to itself and then by the projection $p_{C_{\omega},12} $ to a point.
%  	\item If $\overline{ijpq}\notin D$ then the  $ \tilde{F}_{ijpq} $ projects to DOSTA SLUCAJEVA POJASNITI UOPSTENIJE
%  	\[ \ \]
%  \end{itemize}
\begin{prop}\label{Cor:relacijeG62regioni}
	Let $ s \in \{0,\dots,5\} $ and let  $ C_{\omega} $ be the chamber defined by $ C_{\omega}= \bigcap_{ij\notin S}K_{(s+1)ij}\cap \bigcap_{pq\notin D}K_{pq} $ for $ S,D\subset \{ij|1\leq i <j\leq 6\} $. The homology group $ H_{4}(F_{\omega}) $ is generated by the homology classes $F_{ij},\,ij\in \omega\backslash D$ and the following relations hold:
	\begin{align}\label{Cor:G62_4dimRelacije_reg}
		F_{ij}+F_{pq}= 
		F_{ip}+F_{jq}=
		F_{iq}+F_{jp}
	\end{align} for $ \{i,j,p,q\}\subset \{1,\dots,6\} $, where we substitute for any element from $ D $ the corresponding class by the zero class.
\par  The homology group $ H_{2}(F_{\omega}) $ is generated by the homology classes $F_{(s+1)ij}$ for $(s+1)ij\in \omega ,\,ij\notin S$ and $ F_{ij,sm}$ for $ ij,sm\in \omega\backslash D $ and the following relations hold:
	\begin{align}\label{Cor:G62_2dimRelacije_reg}
	\sum_{sm\notin D}F_{ij,sm}+\sum_{sm\notin D}F_{pq,sm}=
	\sum_{sm\notin D}F_{ip,sm}+\sum_{sm\notin D}F_{jq,sm}=
	\sum_{sm\notin D}F_{iq,sm}+\sum_{sm\notin D}F_{jp,sm}
\end{align}
	for $ \{i,j,p,q\}\subset \{1,\dots,6\} $, where $ \sum_{sm\notin D}F_{rl,sm} $ is  substituted by the zero class for any $ rl $ such that $ r,l\in\{i,j,p,q\} $ and $ rl\in D. $
\end{prop}
\begin{definition}
	Let $ c $ be a chain of the form $ \sum_{ij\in \hat{\omega}} F_{ij} $ or $ \sum_{\{ij,pq\}\in \hat{\omega}} F_{ij,pq} $ in $ F_{\omega_{0}}$ where $ \hat{\omega}\subset \omega_{0} $ and let us fix a chamber $ C_{\omega}=\bigcap_{ij\notin S}K_{sij}\cap\bigcap_{ij\notin D}K_{ij} $ for fixed $ s\in\{1,2,3,4,5,6\}. $ \textit{The projection} $ p_{\omega_{0},\omega}:H_{k}(F_{\omega_{0}})\to H_{k}(F_{\omega}) $ is defined by
	\begin{itemize}
		\item if $ c=\sum_{ij\in \hat{\omega}} F_{ij} $ and $ k=4 $ then \begin{equation}\label{eq:projomega0omegaX6}
			p_{\omega_{0},\omega}(c)=\sum_{ij\in \hat{\omega} \backslash D}F_{ij}
		\end{equation}
		
		\item if $ c=\sum_{\{ij,pq\}\in \hat{\omega}} F_{ij,pq} $ and $ k=2 $ then
		\begin{equation}\label{eq:projomega02omegaX6}
			p_{\omega_{0},\omega}(c)=\sum_{ij,pq\in \hat{\omega}\backslash D} F_{ij,pq}.
		\end{equation}
	\end{itemize}

\end{definition}

By Proposition \ref{Cor:relacijeG62regioni} this projection is well defined and satisfies the following:
\begin{lemma}\label{Lem:omega0_to_all}
	If a cycle $ l $ is the zero class in $ H_{4}(F_{\omega_{0}}) $ or $ H_{2}(F_{\omega_{0}})  $ then the cycle $ p_{\omega_{0},\omega}(l) $ is the zero class in $ H_{4}(F_{\omega}) $ or $ H_{2}(F_{\omega}) $ for any chamber $ C_{\omega} $ respectively. 
\end{lemma}
\begin{proof}
By Proposition \ref{Cor:relacijeG62regioni}, if $ l $ is the zero class in $  H_{4}(F_{\omega_{0}}) $ then it is also the zero class in $   H_{4}(F_{\omega_{k}}),\, k\in\{1,\dots,5\} $ . Thus, we can assume that $ C_{\omega}= \bigcap_{ij\notin S}K_{1ij}\cap \bigcap_{pq\notin D}K_{pq} $. We note that $ 	p_{\omega_{0},\omega}(l) $ does not contain an element $ F_{\sigma} $ for $ \sigma\in D $ and $ F_{\sigma} $ neither participates in relations \eqref{Cor:G62_4dimRelacije_reg} in $  H_{4}(F_{\omega}) $. Then by Corollary \ref{Cor:relacijeG62regioni} the cycle $ p_{\omega_{0},\omega}(l) $ can be spanned by relations \ref{Cor:G62_4dimRelacije_reg} and therefore it is the zero in $  H_{4}(F_{\omega}). $
\end{proof}
\subsection{Homology groups $ H_{k}(X_{6};\Z_{2}) $ for $ 4\leq k\leq9 $}
\begin{theorem}
 $ H_{9}(X_{6};\Z_{2})\simeq \Z_{2}. $
\end{theorem}
\begin{proof}
	 By Corollary \ref{Cycles_3n-9} the group of cycles of dimension 9 has basis of the form \[ c_{mi}:=e_{mi}+e_{m6}+e_{i6}, \] where $ e_{ij}=K_{ij,5}\times F_{ij} $, $ 1\leq m \leq 4 $ and $ m+1 \leq i  < 6 $. \par Let  $ \tilde{c}= \sum C_{\omega}\times l_{\omega} $ be a 10-dimensional chain such that $ \partial \tilde{c}=c $ for some cycle $ c=\sum c_{mi}=\sum_{\omega} C_{\omega}\times (\sum F_{ij}) $. A chain $ l_{\omega} $ must be of the dimension 5, so it is a chain in $ F_{6}. $  We get that $ \partial \tilde{c}= \sum C_{\omega}^{\prime}\times l_{\omega}+\sum C_{\omega}\times \partial l_{\omega}.$ Since the first sum must be zero we have that $ l_{\omega}=l ,\,\forall\omega$ and $ \bigcup C_{\omega}=\Delta_{6,2,5}.  $ Thus, $ \partial l=\sum F_{ij},$ so it is homology zero $ H_{4}(F_{\omega}) $.
	 \par We consider the sum $ c_{12}+c_{13}=e_{12}+e_{13}+e_{26}+e_{36}=\sum C_{\omega}\times(F_{12}+F_{13}+F_{26}+F_{36}). $ By \eqref{Cor:G62_4dimRelacije} for the indices $ \{1,2,3,6\} $  we get that $ F_{12}+F_{13}+F_{26}+F_{36} $ is zero  in $ H_{4}(F_{\omega_{0}}) $. We get by Lemma \ref{Lem:omega0_to_all} that the chain $ p_{\omega_{0},\omega}(F_{12}+F_{13}+F_{26}+F_{36}) $ is  zero in $ H_{4}(F_{\omega}) $ for all $ C_{\omega}. $ Therefore, the sum $ c_{12}+c_{13} $ is zero in $ H_{9}(X_{6};\Z_{2}).$ Similarly, we prove for every sum $ c_{ij}+c_{ik} $ to be zero class.
\end{proof}
\begin{theorem}
	$ H_{8}(X_{6};\Z_{2})\simeq \Z_{2}. $
\end{theorem}
\begin{proof}
	By Theorem \ref{Th:Cycles_n-1_odd} a cycle $ c $ of dimension 8 is homologous to a cycle from $ Y_{6}.$ Since there are no chains of dimension $ 8 $ in $ \hat{\mu}^{-1}(\partial O_{m}) $, by Corollary \ref{Cor:Spaces_Y_n} it follows that the cycle $ c $ is of the form:
	\[ c=\sum_{1\leq m\leq 6} a_{m}(O_{m,4}\times F^{m}_{5}),\, a_{m}\in \Z_{2} . \]
	  Every chain $ O_{m,4}\times F^{m}_{5} $ is a cycle that we denote by $ g_{m}. $
	\par Let $ p=\sum C_{\omega}\times l_{\omega} $ be a chain of the dimension 9 such that $ \partial p=l $ where $ l=\sum g_{m}$. It follows that $ C_{\omega} $ is a 5-dimensional and $ l_{\omega} $ is a 4-dimensional chain. Following Corollary \ref{Cor:Chains_with_boundary_on_boundary} we consider the chains $ e_{ij}=K_{ij,5}\times F_{ij}$ . It follows that \[  \partial e_{ij}=O_{i,4}\times F^{i}_{5}+O_{j,4}\times F^{j}_{5}=g_{i}+g_{j}. \] We conclude that the sum of every two generators $ g_{i} $ is zero in $ H_{8}(X_{6};\Z _{2}). $
\end{proof}
\begin{lemma}\label{Th:H7X6}
	Generators of $ H_{7}(X_{6};\Z_{2}) $ are from the following set of 7-dimensional cycles:
	\[ g_{ijkl}=e_{ijk}+e_{ijl}+e_{ikl}+e_{jkl}+e_{ij,kl}+e_{ik,jl}+e_{il,jk} \]
	where $ e_{ijk}=K_{ijk,5}\times F_{ijk} $ and $ e_{ij,kl}=K_{ij,kl,5}\times F_{ij,kl} $ for distinct $ i,j,k,l\in \{1,\dots,6\}. $
\end{lemma}
\begin{proof}
	By Theorem \ref{Th:Cycles_n-1_even} a cycle $ c $ of the dimension 7 is homologous to cycle $ c_{0}^{\prime}+c_{1}^{\prime}  $ where $ c_{0} $ is a cycle in $\hat{\mu}^{-1}(\Ins{\D{6}}) $ and $ c_{1}$ is a cycle in $\hat{\mu}^{-1}(\partial \D{6}).$. By Corollary \ref{Cor:Spaces_Y_n} a cycle $ c_{1}^{\prime}$ is in $\hat{\mu}^{-1}(\bigcup_{1\leq m\leq 6}\partial O_{m})$. Since there are no 7 dimensional chains on $ \partial O_{m} $, a cycle $ c=c_{0}^{\prime}. $ In addition, by Corollary \ref{Cor:Spaces_Y_n} and Lemma \ref{Corr:SpaceofparametersX6} any cycle must be written in the form $ \sum e_{ij,kl}+\sum e_{ijk}.$ We get
	\begin{align*}
		\partial_{5}e_{ij,kl}=e_{ij}^{k}+e_{ij}^{l}+e_{kl}^{i}+e_{kl}^{j},\quad
		\partial_{5}e_{ijk}=e_{ij}^{k}+e_{ik}^{j}+e_{jk}^{i}
	\end{align*} where $ e_{ij}^{k}$ denotes the chain $ e_{ij}=K_{ij,4}(k)\times F_{ij}(k) $ for $ K_{ij}(k)\subset O_{k}. $ We check at once that a chain $ g_{ijkl} $ is the cycle and cannot be obtained as the sum of some cycles of less number of summands and it follows that every cycle is generated by the set of $ g_{ijkl}.$ A summand $ e_{ij,kl} $ appears just in one element $ g_{ijkl},$ therefore the set of $ g_{ijkl} $ is linearly independent.
\end{proof}
\begin{theorem}
	 $ H_{7}(X_{6};\Z_{2}) \simeq (\Z_{2})^{11}. $
\end{theorem}
\begin{proof}
	Let $ c=\sum g_{ijkl}=\sum_{\omega} C_{\omega}\times(\sum_{\sigma\in\omega} F_{\sigma}) $ be a cycle. It will be the zero class in $ H_{7}(X_{6};\Z_{2}) $ only if $ \sum F_{\sigma} $ is zero in $ H_{2}(F_{\omega},\Z_{2})$ for every $ C_{\omega}. $ It follows from Relations \ref{Cor:G62_2dimRelacije_reg} that the cycle $ c $ is zero class only if it is of the form $ \sum e_{ij,kl}. $
	\par Let us consider the cycles 
	\begin{equation}\label{Eq:H7X6}
	g_{s}=\sum_{\substack{ijkl\\s\notin ijkl}} g_{ijkl}=\sum_{\substack{ijkl\\s\notin ijkl}} e_{ij,kl}.	
	\end{equation}
	  for $ s\in\{1,\dots,6\}. $ Note that if we remove a summand $ g_{ijkl} $ from the sum \eqref{Eq:H7X6}, what is left cannot be expressed in terms of $ e_{ij,kl} $.  We show that $ g_{6} $ is homologically equal to $ g_{5}. $ The chamber $ C_{\omega_{0}} $ participates in the cycles $ g_{6} $ and $ g_{5}.$ Firstly we show that  \[ \sum_{ijkl\neq6}F_{ij,kl}=\sum_{ijkl\neq5}F_{ij,kl} \quad \text{in }  H_{2}(F_{\omega_{0}},\Z_{2}). \] It holds in $ H_{2}(F_{\omega_{0}},\Z_{2})$ :
	\begin{align*}
		&\sum_{ijkl\neq6}F_{ij,kl}=\sum_{kl\in\{34,35,45\}}F_{12,kl}+\sum_{kl\in\{24,25,45\}}F_{13,kl}+\sum_{kl\in\{15,35\}}F_{24,kl}+\sum_{kl\in\{15,25\}}F_{34,kl}+\sum_{kl\in\{25,35\}}F_{14,kl}+\\
		&+\sum_{kl\in\{15,14,45\}}F_{23,kl}=\\
		&=\sum_{kl\in\{36,46,56\}}F_{12,kl}+\sum_{kl\in\{26,46,56\}}F_{13,kl}+\sum_{kl\in\{13,16,36,56\}}F_{24,kl}+\sum_{kl\in\{12,16,26,56\}}F_{34,kl}+\sum_{kl\in\{25,35\}}F_{14,kl}+\\
		&+\sum_{kl\in\{15,14,45\}}F_{23,kl}=\\
		&=\sum_{kl\in\{36,46\}}F_{12,kl}+\sum_{kl\in\{26,46\}}F_{13,kl}+\sum_{kl\in\{13,16,36\}}F_{24,kl}+\sum_{kl\in\{12,16,26\}}F_{34,kl}+\sum_{kl\in\{25,35\}}F_{14,kl}+\\
		&+\sum_{kl\in\{15,14,45\}}F_{23,kl}+\sum_{kl\in\{12,13,24,34\}}F_{56,kl}=\\
		&=\sum_{kl\in\{36,46\}}F_{12,kl}+\sum_{kl\in\{26,46\}}F_{13,kl}+\sum_{kl\in\{13,16,36\}}F_{24,kl}+\sum_{kl\in\{12,16,26\}}F_{34,kl}+\sum_{kl\in\{26,36\}}F_{14,kl}+\\
		&+\sum_{kl\in\{14,16,46\}}F_{23,kl}=\\
		&=\sum_{kl\in\{23,24,34\}}F_{16,kl}+\sum_{kl\in\{13,14,34\}}F_{26,kl}+\sum_{kl\in\{12,14,24\}}F_{36,kl}+\sum_{kl\in\{12,13,23\}}F_{46,kl}+F_{12,34}+F_{13,24}+F_{14,23}=\\
		&=\sum_{ijkl\neq5}F_{ij,kl},
	\end{align*}
	where the third and the fifth equality are obtained just regrouping of a summands, the second equality follows from the relation
	\[ \sum_{kl}F_{12,kl}+\sum_{kl}F_{34,kl}=\sum_{kl}F_{13,kl}+\sum_{kl}F_{24,kl} \] and the fourth equality follows from the relation \[ \sum_{kl\in\{25,26,35,36\}}F_{14,kl}+\sum_{kl\in\{14,15,16,45,46\}}F_{23,kl}+\sum_{kl\in\{12,13,24,34\}}F_{56,kl}=0. \] By Lemma \ref{Lem:omega0_to_all} it follows that $ g_{6}=g_{5} $ in $ H_{7}(X_{6};\Z_{2}).$ Similarly, we can show that any two $ g_{s} $ and $ g_{r} $ are homologous. 
		\par Next, we want to eliminate linearly dependent generators among $ g_{ijkl} $ : 
	\begin{align*}
		g_{6}=g_{5}\implies \quad &g_{2345}=g_{1235}+g_{1245}+g_{1345}+g_{1236}+g_{1246}+g_{1346}+g_{2346}:=T_{1}+g_{2346}\\
			g_{5}=g_{4}\implies \quad &g_{2356}=g_{1234}+g_{1246}+g_{1346}+g_{2346}+g_{1235}+g_{1256}+g_{1356}:=T_{2}+g_{2346}\\
			g_{4}=g_{3}\implies \quad &g_{2456}=g_{1235}+g_{1236}+g_{1356}+g_{2356}+g_{1245}+g_{1246}+g_{1456}:=T_{3}+g_{2356}=\\
			&=T_{3}+T_{2}+g_{2346}=T_{3}+T_{2}+T_{1}+g_{2345}\\
		g_{3}=g_{2}\implies \quad &g_{3456}=g_{1245}+g_{1246}+g_{2456}+g_{1256}+g_{1345}+g_{1346}+g_{1356}:=T_{4}+g_{2456}\\
			g_{2}=g_{1}\implies \quad &g_{1345}+g_{1346}+g_{1356}+g_{1456}=g_{2345}+g_{2346}+g_{2356}+g_{2456}=\\
			&=g_{2345}+g_{2346}+T_{2}+g_{2346}+T_{3}+T_{2}+T_{1}+g_{2345}=T_{3}+T_{1}\\
			&=g_{1345}+g_{1346}+g_{1356}+g_{1456},			
	\end{align*} where it can be seen that the last implication follows from the previous ones. 
We get that $ g_{2345},g_{2356},g_{2456} $ and $ g_{3456} $ are linearly dependent.
\end{proof}

%\begin{rem}
%	In the previous theorem it is easy computation to check if a chain $ g_{x} $ is not the zero class on $ H_{7}(X_{6}). $
%\end{rem}
\begin{theorem}
	$ H_{6}(X_{6};\Z_{2})\simeq (\Z_{2})^{3}. $
\end{theorem}
\begin{proof}
		By Proposition \ref{Th:Cycles_n-1} a cycle of dimension $ 6 $ is homologous to a cycle in $ Y_{6}.$ Such cycle is by Corollary \ref{Cor:Spaces_Y_n} homologous to a cycle \[ 	c^{\prime}=\sum_{1\leq m\leq 6} c^{m}+c_{2}, \text{ where }  c_{m}\in \hat{\mu}^{-1}(\Ins{O}_{m}) \text{ and } c_{2}\in \hat{\mu}^{-1}(\bigcup_{1\leq m\leq 6}\partial O_{m}).\] Since there are no 6-dimensional chains on $ \hat{\mu}^{-1}(\partial O_{m}) $ we have that $ c_{2}=0. $ A chain $ c^{m} $ has to be of the form:  
		\[ c^{m}=\sum_{pq}e_{pq}^{m}, \] where $ e_{pq}^{m}$ is a chain in $ \hat{\mu}^{-1}(\Ins{O}_{m}) $ defined in Theorem  \ref{Th:H6X5}. If $ c^{m} $ is a cycle, by Theorem \ref{Th:H6X5} we get $ c^{m}= e_{12}^{m}+e_{13}^{m}+e_{23}^{m}.$
		\par If $ c^{m} $ is not a cycle, then $ \partial c^{m}\in \hat{\mu}^{-1}(\partial O_{m}),$ which must be canceled by some chain $ \partial\sum_{r\neq m} c^{r} $. It holds $ \partial e_{pq}^{m}=O_{pm,3}\times l_{pm}+O_{qm,3}\times l_{qm} $, where $ O_{ij}=(O_{i}\cap O_{j})\simeq\D{4} $ for $ 1\leq i<j\leq 6 $. It can be easily seen that the only cycles of this type are:
		\begin{align}\label{Eq:seven_dimensional_chains}
			e_{ij}^{k}+e_{ij}^{l}+e_{kl}^{i}+e_{kl}^{j}=\partial_{5}e_{ij,kl},\quad 
			e_{ij}^{k}+e_{ik}^{j}+e_{jk}^{i}=\partial_{5}e_{ijk}.
		\end{align} They are homologous to zero.  It follows that there are six potentially generators for $ H_{6}(X_{6};\Z_{2}): $ 
	\[ g_{i}=e_{12}^{i}+e_{13}^{i}+e_{23}^{i}  \text{ for each complex }   O_{i},\, 1\leq i\leq 6.  \]
		\par We need to check if a combination of $ g_{i} $ is homologous to zero. Let $ c=\sum C_{\omega}\times l_{\omega} $ such that $ c\in\hat{\mu}^{-1}(\Ins{\D{6}}) $  and $ \partial^0 c=0 $ be a 7-dimensional chain, where $ \partial c $ is a combination of $ g_{i}. $ By Corollary \ref{The:HomologyChain_n-2_general} every $ C_{\omega} $ of $ c $ is a  chamber of maximal dimension and by Corollary \ref{Cor:Chains_with_boundary_on_boundary} a chain $ l_{\omega} $ is a combination of $ F_{ij,kl} $ and $ F_{ijk}. $ It follows that $ c $ is a sum of $ e_{ij,kl} $ and $ e_{ijk}. $ For example, by \eqref{Eq:seven_dimensional_chains} we obtain: \[ \partial (e_{12,36}+e_{13,26}+e_{16,23})=g_{6}+g_{3}+g_{1}+g_{2}=\partial(e_{123}+e_{126}+e_{136}+e_{236}) . \] In this way we deduce that the sum of every four $ g_{i} $'s is linearly dependent in $ H_{6}(X_{6};\Z_{2}). $
\end{proof}
\begin{theorem}
	 $ H_{5}(X_{6};\Z_{2})\simeq \Z_{2}. $
\end{theorem}
\begin{proof}
	By Proposition \ref{Th:Cycles_n-1} a cycle $ c $ of dimension $ 5 $ is homologous to a cycle in $Y_{6},$ so by Corollary \ref{Cor:Spaces_Y_n} it is homologous to a cycle from $ \hat{\mu}^{-1}(\bigcup_{1\leq m\leq 6}\partial O_{m}).$  It follows that $ \dim(l_{\omega})\geq 1 $ for all $ \omega $  and we can write $ c $ as:
	\[ c=\sum c_{ij}\text{ where } c_{ij}\in \hat{\mu}^{-1}(\Ins{O}_{ij}) \] and $ O_{ij}=O_{i}\cap O_{j}\simeq\D{4} $ for $ 1\leq i<j\leq 6 .$  By Theorem \ref{Th:Homology_H5_X5}  we have that $ H_{5}(X_{5}(j),\Z_{2})\simeq \Z_{2} $ is spanned by any of $ g^{j}_{i}= O_{ij,3}\times l_{ij},\,1\leq i\leq 6 $ and $ i\neq j $ for $ 1\leq j\leq 6 $. Since $ g_{i}^{j}=g_{j}^{i} $ we are left with one generator that we denote by $ g. $ Such generator cannot be obtained as the boundary of some $ 6 $-dimensional chain $ \tilde{c}\in \hat{\mu}^{-1}(\Ins{\Delta_{6,2}}),\,\tilde{c}=\sum C_{\omega}\times l_{\omega} $, where $ \dim(C_{\omega})=4 $ and $ \dim(l_{\omega})=2 $. Since $ \partial^0 \tilde{c}=0,$ by Corollary \ref{The:HomologyChain_n-2_general} the chain $ \tilde{c} $ is homologous to a chain from $ Y_{6} $ and $ \tilde{c}  $ cannot be a chain in $\hat{\mu}^{-1}(O_{m})  $ for any $ 1\leq m\leq 6. $ Only possibility is that it is the sum of two chains from $ \Ins{O}_{i} $ and $ \Ins{O}_{j} $, that is $ \tilde{c}=\tilde{c}_{i}+\tilde{c}_{j}.$ Since $ \partial \tilde{c}=g  $ by Corollary \ref{Cor:Chains_with_boundary_on_boundary} the chain $ \tilde{c}_{i}=e_{pq}^{i}$, where $ e_{pq}^{i}$ is the chain from $ \hat{\mu}^{-1}( O_{i}) $ defined in Theorem $ \ref{Th:H6X5}.$ We conclude that such chain $ \tilde{c} $ does not exist.
\end{proof}

\begin{theorem}
	$ H_{4}(X_{6};\Z_{2})\simeq 0.$ 
\end{theorem}
\begin{proof}
	By Corollary \ref{The:HomologyChain_n-2_general} cycles $ c $ of the dimension $ 4 $ are homologous to cycles from $ Y_{6}. $ By Corollary \ref{Cor:Spaces_Y_n} a cycle of the dimension $ 4 $ is homologous to a cycle from $ \hat{\mu}^{-1}(\bigcup_{1\leq m\leq 6}\partial O_{m}).$ It follows that $ \dim(l_{\omega})\geq 1 $ for all $ \omega $  and we can write a cycle $ c $ as:
	\[ c=\sum c_{ij}\text{ where } c_{ij}\in \hat{\mu}^{-1}(\Ins{O}_{ij}) \] and $ O_{ij}=O_{i}\cap O_{j}\simeq\D{4} $ for $ 1\leq i<j\leq 6 .$ It holds $ \partial c_{ij}\in \hat{\mu}^{-1}(\partial O_{ij}) $ and by Corollary \ref{The:HomologyChain_n-2_general} it follows that $ c_{ij} $ is homologous to a chain from $ \hat{\mu}(\partial O_{ij}) $ which has to be zero.
\end{proof}

	\nocite{*}
	\printbibliography
 \begingroup
\parindent 0pt
\parskip 2ex
\def\enotesize{\normalsize}
\theendnotes
\endgroup
 \end{document}